\pdfoutput=1
\documentclass[11pt,letterpaper,onecolumn]{article}
\usepackage{arxiv-setting}
\usepackage[letterpaper,margin=1in]{geometry}
\renewcommand{\epsilon}{\varepsilon}

\newcommand{\toromanlower}[1]{\ifcase#1 \or i\or ii\or iii\or iv\or v\or vi\or vii\or viii\or ix\or x\fi}

\title{Majority Dynamics on Assortative Sparse Stochastic Block Models\blfootnote{Author names are listed in alphabetical order.}}
\author{
Ioana Dumitriu%
\thanks{\mbox{\scriptsize Department of Mathematics, University of California San Diego, La Jolla, CA 92093, USA; \texttt{idumitriu@ucsd.edu}.}}
\and
Muchen Ju\orcidlink{0009-0004-6131-6911}%
\thanks{\mbox{\scriptsize Department of Mathematics, University of Pennsylvania, Philadelphia, PA 19104, USA; \texttt{muchenju@sas.upenn.edu}.}}
\and
Hai-Xiao Wang\orcidlink{0000-0003-2730-1439}%
\thanks{\mbox{\scriptsize Department of Applied Mathematics, University of Washington, Seattle, WA 98195, USA; \texttt{haixwang@uw.edu}.}}
}

\date{}

\begin{document}

\maketitle

\begin{abstract}
    Majority dynamics is a two-opinion process in which each vertex repeatedly
    updates to the majority opinion among its neighbors. We study this process on a
    resampled sparse binary stochastic block model in the assortative regime. At
    each time step, a graph is sampled from the
    current opinion partition: vertices with the same opinion are joined with
    probability \(\alpha=a\log N/N\), while vertices with differing opinions are
    joined with probability \(\beta=b\log N/N\), where \(a>b>1\). Let
    \(\gB_t\) and \(\gR_t\) denote the blue and red camps at time \(t\). We show that
    the weighted advantage \(\widetilde{\Delta}_{t}=b|\gB_t|-a|\gR_t|\), rather than the
    unweighted advantage \(\Delta_{t}=|\gB_t|-|\gR_t|\) alone, governs the pace to unanimity.
    
      Our results, which hold with high probability as \(N\to\infty\), identify three regimes for blue unanimity under the initial blue advantage, i.e., $\Delta_0>0$: constant time, subpolynomial time, and polynomial time. First, when \(\widetilde{\Delta}_0 \gtrsim -N/\sqrt{\log N}\), blue unanimity occurs within three updates. Second, when \(\widetilde{\Delta}_0 < 0\) and \(|\widetilde{\Delta}_0| = o(N)\), blue unanimity occurs within \(N^{o(1)}\) updates. Furthermore, when \(\widetilde{\Delta}_0 < 0\), \(|\widetilde{\Delta}_0| = O(N)\), and \(\Delta_0\gg\sqrt{N/\log N}\), blue unanimity still occurs within \(N^{I_0+o(1)}\) updates, where
    \[I_0=\big(\ReLU(\sqrt{a|\gR_0|/N}-\sqrt{b|\gB_0|/N})\big)^2\]
    and \(\ReLU(x)=\max\{x,0\}\). Conversely, away from the weighted threshold,
    when \(|\gB_0|/|\gR_0|\le a/b-\kappa\) and \(\Delta_0>0\), \(N^{I_0 - o(1)}\) updates are necessary for blue unanimity. Our analysis relies on the detailed estimates for one-vertex flip probabilities in sparse binomial differences, which could be of independent interest.
\end{abstract}

\noindent \textbf{Keywords:} Majority Dynamics; Stochastic Block Models; Sparse Random Graphs.

\newpage
\setcounter{tocdepth}{1}
\tableofcontents
\newpage


\section{Introduction}\label{sec:introduction}

\emph{Majority dynamics} has been frequently used as a model for studying the evolution of opinions on networks; see the survey \cite{mossel2017opinion}. Mathematically, let \(\gV=[N]\) denote the set of individuals, and \(\rvy_t\in\{\pm1\}^{\gV}\) record the opinions of individuals at time \(t \geq 0\), where \(+1\) denotes \emph{blue} and \(-1\) denotes \emph{red}. The two opinion classes are denoted by
\begin{align}
\gB_t\coloneqq\{v\in\gV:\ervy_t(v)=+1\},
\qquad
\gR_t\coloneqq\{v\in\gV:\ervy_t(v)=-1\}.
\label{eqn:blue-red-camps}
\end{align}
At update \(t+1\), the opinion of each individual is determined by its interactions with its neighbors in the graph \(\gG_{t+1}=(\gV,\gE_{t+1})\). Denote the number of blue and red neighbors of \(v\) in \(\gG_{t+1}\) by
\begin{align}
N^{\gB}_{t}(v)=\#\{w\in\gB_t:\{v,w\}\in\gE_{t+1}\},
\qquad
N^{\gR}_{t}(v)=\#\{w\in\gR_t:\{v,w\}\in\gE_{t+1}\}.
\label{eqn:neighbor-count}
\end{align}
All vertices update synchronously by comparing these two neighbor counts:
\begin{align}
\ervy_{t+1}(v)
=
\begin{cases}
+1, & N^{\gB}_{t}(v)>N^{\gR}_{t}(v),\\
-1, & N^{\gB}_{t}(v)<N^{\gR}_{t}(v),\\
\ervy_t(v), & N^{\gB}_{t}(v)=N^{\gR}_{t}(v).
\end{cases}
\label{eqn:opinion-update}
\end{align}
Thus a vertex retains its current opinion when the neighbor vote is tied. We
say that blue \emph{unanimity} occurs at time \(T\) if
\(\gR_T=\emptyset\), while red unanimity occurs if \(\gB_T=\emptyset\).

The central problem is to understand how long it takes for blue unanimity to
occur, as a function of the initial configuration and the rule generating the
update graphs at each update. When the update graph is an Erd\H{o}s--R\'enyi graph
\(\gG(N,p)\), the ``power of few'' phenomenon, in which a small initial advantage
can cause one opinion to win with high probability after only a few
updates, has been extensively studied in the literature. In the dense regime,
where \(p\) remains constant or decays slowly with \(N\), the works in
\cite{benjamini2016convergence, fountoulakis2020resolution, tran2023reaching,
berkowitz2020central} prove rapid stabilization and show that constant-size or
even one-vertex advantages can affect the winning probability. The
asymptotically exact winning probability was calculated explicitly in
\cite{sah2024majority}. In the sparse regime, isolated vertices
below the connectivity threshold\footnote{With high probability as \(N\to\infty\), \(\gG(N,p)\) has
isolated vertices when \(p\le(1-\epsilon)\log(N)/N\) for some \(\epsilon>0\).}
create an immediate obstruction to unanimity; above the connectivity threshold,
\cite{zehmakan2020opinion, chakraborti2023majority, tran2025power, kim2025new,
jaffe2025new} show that sparse random graphs can still amplify small initial
majorities to consensus, and subsequent work progressively lowers the required
edge density for this sparse ``power of few'' behavior.

In the Erd\H{o}s--R\'enyi setting, every pair of vertices is treated homogeneously, so the update graph carries no community structure tied to the current opinions. Instead, we focus on the case where the update graph is resampled at every update from a binary \emph{Stochastic Block Model} (SBM), whose two communities are precisely the current blue and red opinion classes, formally defined below.
\begin{definition}[Binary Stochastic Block Model]
Let \(\rvy\in\{\pm1\}^{N}\) denote the membership vector. Independently for each distinct pair of vertices \(u,v\in[N]\), include the edge \(\{u,v\}\) with probability \(\alpha\) if \(\ervy(u)=\ervy(v)\), and with probability \(\beta\) otherwise.
\end{definition}
As a consequence, the graph distribution itself reflects the evolving opinion
partition: pairs with the same current opinion and pairs with different current
opinions are assigned different edge probabilities. Previous work
\cite{wang2022consensus} studies opinion evolution under SBM update graphs in
the dense regime. The present paper instead focuses on the \emph{critical
sparse} regime, where the connectivity scale and the fluctuation scale interact
with the evolving opinion imbalance.
\begin{assumption}[Sparse assortative SBM]
\label{ass:sparse-assortative-sbm}
There are constants \(a>b>0\) such that
\begin{align}
\alpha=a\log(N)/N,
\qquad
\beta=b\log(N)/N.
\label{eqn:alpha-beta}
\end{align}
\end{assumption}

We now specify the stochastic process used throughout the paper. For each
\(N\), fix an initial configuration \(\rvy_0\), and let
\begin{align}
\mathcal F_t
\coloneqq
\sigma(\rvy_0,\gG_1,\ldots,\gG_t),
\qquad t\ge0,
\label{eqn:natural-filtration}
\end{align}
where \(\mathcal F_0=\sigma(\rvy_0)\). Recursively, conditional on
\(\mathcal F_t\), sample the edge indicators of \(\gG_{t+1}\) independently,
with
\begin{align}
\P\left(\{u,v\}\in\gE_{t+1}\,\middle|\,\mathcal F_t\right)
=
\begin{cases}
\alpha, & \ervy_t(u)=\ervy_t(v),\\
\beta, & \ervy_t(u)\ne\ervy_t(v),
\end{cases}
\qquad u\ne v.
\label{eqn:conditional-update-graph-law}
\end{align}
Having sampled \(\gG_{t+1}\), obtain \(\rvy_{t+1}\) by applying
\eqref{eqn:opinion-update} simultaneously at every vertex. Thus each update
graph is sampled afresh from the SBM determined by the current opinion
partition. The graphs are not independent without conditioning, because their
laws depend on the evolving configurations, but \((\rvy_t)_{t\ge0}\) is a
time-homogeneous Markov chain. Unless stated otherwise, all probabilities are
with respect to the update graphs for the fixed initial configuration.

As noted above, isolated vertices immediately obstruct unanimity. We therefore
impose the following convenient condition to ensure connectivity uniformly over
the evolving opinion partitions.
\begin{assumption}[Connectivity]
\label{ass:connectivity}
In addition to Assumption~\ref{ass:sparse-assortative-sbm}, assume that \(b>1\).
\end{assumption}
The assumption above ensures that the SBM is connected with high probability
for any partition imbalance that may arise during the dynamics. This condition
is slightly stronger than the connectivity condition for an SBM with a given
partition, derived in Lemma~\ref{lem:connectivity-sparse-sbm}.

\subsection{Main Results}

To illustrate the main results, we first introduce the unweighted advantage by
\begin{align}
\Delta_t\coloneqq |\gB_t|-|\gR_t|.
\label{eqn:unweighted-advantage}
\end{align}
To account for the different within- and across-community edge probabilities, we define the weighted advantage by
\begin{align}
\widetilde{\Delta}_t\coloneqq b|\gB_t|-a|\gR_t|.
\label{eqn:weighted-advantage}
\end{align}
\begin{remark}\label{rem:weighted-advantage-orientation}
    The weighted advantage is oriented toward the initial blue advantage and is used under our standing assumption \(\Delta_0>0\). If \(\Delta_0<0\), all our results hold after interchanging the color labels and defining \(\widetilde{\Delta}_t\coloneqq b|\gR_t|-a|\gB_t|\); the corresponding conclusions then concern red unanimity.
\end{remark}
Since \(|\gB_t|+|\gR_t|=N\), the two camp sizes can be recovered from either advantage as
\begin{subequations}
\begin{align}
|\gB_t|
&=\frac{N+\Delta_t}{2}
=\frac{aN+\widetilde{\Delta}_t}{a+b} \label{eqn:blue-size-from-advantage},\\
|\gR_t|
&=\frac{N-\Delta_t}{2}
=\frac{bN-\widetilde{\Delta}_t}{a+b} \label{eqn:red-size-from-advantage},
\end{align}
\end{subequations}
which further implies the affine relation between the two advantages:
\begin{align}
\widetilde{\Delta}_t
&=\frac{a+b}{2}\Delta_t-\frac{a-b}{2}N,
\qquad
\Delta_t
=\frac{2\widetilde{\Delta}_t+(a-b)N}{a+b}.
\label{eqn:affine-delta-tilde-delta}
\end{align}

Our first result identifies nested weighted-advantage thresholds for constant-time blue unanimity: a weighted advantage bounded below by a constant multiple of
\(-N/\sqrt{\log N}\) guarantees unanimity within three updates, while progressively stronger positive advantages reduce the convergence time to two updates and, ultimately, one. Under our standing regime \(a>b>1\), we define
\begin{align}
r_{*} \coloneqq r_{*}(a, b) =
\left(
\frac{
\sqrt{b(a+b-1)}-\sqrt a
}{
a+b
}
\right)^2.
\label{eqn:r-star}
\end{align}
Since \(a>b>1\) under Assumptions~\ref{ass:sparse-assortative-sbm} and~\ref{ass:connectivity}, we have \(0<r_{*}<1/2\). Let \(\Phi\) denote the CDF of the standard normal distribution. We further define
\begin{align}
K_2(a,b)
\coloneqq
\sqrt{\frac{2ab}{a+b}}\,
\Phi^{-1}\left(1-\frac{a+b}{b}r_{*}\right).
\label{eqn:K2-def}
\end{align}
\begin{remark}
\(K_2(a,b)\) is not necessarily positive. For example, when
\(a=9\), \(b=8\), we have
\[
r_*(9,8)\approx0.23916,
\qquad
K_2(9,8)\approx-0.05996.
\]
Thus \(K>K_2(a,b)\) may hold even when \(K<0\). In that case, the
corresponding hypothesis permits a slightly negative weighted advantage of
order \(N/\sqrt{\log N}\).
\end{remark}

\begin{theorem}[Constant-time blue unanimity]
\label{thm:three-day-unanimity}
Under Assumptions~\ref{ass:sparse-assortative-sbm} and
\ref{ass:connectivity}, the following statements hold.
\begin{enumerate}[label=\textup{(\roman*)}]
\item \textup{(Three updates.)} For every fixed \(H<\infty\), if
\begin{align}
\widetilde{\Delta}_0\ge -H\frac{N}{\sqrt{\log N}},
\label{eqn:three-day-unanimity-condition}
\end{align}
then there exist constants \(c_H=c_H(a,b,H)>0\) and \(\xi=\xi(a,b)>0\) such that,
\[
\P(\gR_3=\emptyset)
\ge
1-2\exp(-c_HN)-2\exp\big(-\big(\log N\big)^2\big)-2N^{-\xi},
\]
for all sufficiently large \(N\).
\item \textup{(Two updates.)} If, for some \(K> K_2(a,b)\), we have
\begin{align}
\widetilde{\Delta}_0\ge K\frac{N}{\sqrt{\log N}},
\label{eqn:two-day-unanimity-condition}
\end{align}
then there exists \(\xi=\xi(a,b,K)>0\) such that, for all sufficiently large
\(N\),
\[
\P(\gR_2=\emptyset)
\ge
1-2\exp\big(-\big(\log N\big)^2\big)-2N^{-\xi}.
\]
\item \textup{(One update.)} If, for some fixed $r$ satisfying \(0<r<r_*\), we have
\begin{align}
\widetilde{\Delta}_0\ge \left[b-(a+b)r\right]N,
\label{eqn:one-day-unanimity-condition}
\end{align}
then there exists \(\xi=\xi(a,b,r)>0\) such that, for all sufficiently large \(N\),
\[
\P(\gR_1=\emptyset)
\ge
1-2N^{-\xi}.
\]
\end{enumerate}
\end{theorem}

Beyond the constant-time regime, when \(-\widetilde{\Delta}_0\ll N\), the initial weighted disadvantage can be overcome by accumulating almost-linear one-step gains, leading to the following subpolynomial-time blue unanimity result.

\begin{theorem}[Subpolynomial-time blue unanimity]
\label{thm:super-polylog-unanimity}
Under Assumptions~\ref{ass:sparse-assortative-sbm} and
\ref{ass:connectivity}, let \(h_N\to\infty\) satisfy
\(h_N=o(\sqrt{\log N})\). Suppose that
\[
\widetilde{\Delta}_0\ge -h_N\frac{N}{\sqrt{\log N}}.
\]
For some constant $A$ depending on $a, b$, define the time horizon by
\begin{align}
T_N\coloneqq\left\lceil\exp(Ah_N^2)\right\rceil.
\label{eqn:super-polylog-time-horizon}
\end{align}
Then there exist constants \(c,C,\xi>0\), depending only on \(a,b\), such that,
for all sufficiently large \(N\),
\[
\P\left(\gR_{T_N+3}=\emptyset\right)\ge1-2\exp(Ah_N^2)\exp\left(-cN\frac{\exp(-C h_N^2)}{(1+h_N)^2}\right)-4\exp\left(-(\log N)^2\right)-2N^{-\xi}.
\]
In particular, the right-hand side is \(1-O(N^{-\xi+o(1)})\).
\end{theorem}

When the weighted disadvantage is of linear order, the one-step increase is
sublinear and its scale is described by the following large-deviation
exponent:
\begin{align}
I_t\coloneqq I(a|\gR_t|/N, b|\gB_t|/N),
\label{eqn:It-exponent}
\end{align}
where \(\ReLU(x)=\max\{x,0\}\) and the rate function is defined by
\begin{align}
    I(x, y) \coloneqq \big( \ReLU(\sqrt{x}-\sqrt{y}) \big)^2. \label{eqn:LDP-rate-function}
\end{align}

By the affine relation \eqref{eqn:affine-delta-tilde-delta}, the increments of
the weighted and unweighted advantages differ only by the fixed factor
\((a+b)/2\). We may therefore track the same amplification through
\(\Delta_t\). The theorem below assumes explicitly that blue initially has
the unweighted advantage and that this advantage satisfies
\(\Delta_0\gg\sqrt{N/\log N}\).

\begin{theorem}[Polynomial-time blue unanimity]
\label{thm:polynomial-minimal-bias}
Under Assumptions~\ref{ass:sparse-assortative-sbm} and \ref{ass:connectivity}, suppose
\[
\Delta_0>0,
\qquad
\Delta_0\gg\sqrt{N/\log N}.
\]
Then for every \(\varepsilon>0\), there exist constants \(c=c(a,b,\varepsilon)>0\) and \(\xi=\xi(a,b)>0\) such that, for all sufficiently large \(N\),
\[
\P\left(\gR_{\lceil N^{I_0+\varepsilon}\rceil}=\emptyset\right)
\ge
1-
\exp\left(-c\frac{\Delta_0^2\log N}{N}\right)-4N^{-\xi}.
\]
\end{theorem}

The preceding theorem gives a polynomial-time upper bound. For the lower bound, define the initial camp ratio by
\begin{align}
\rho_0\coloneqq\frac{|\gB_0|}{|\gR_0|}.
\label{eqn:initial-camp-ratio}
\end{align}
The next result shows that
the exponent \(I_0\) is essentially unavoidable when the initial camp ratio
remains a fixed distance below the weighted threshold \(a/b\).

\begin{theorem}[Polynomial-time lower bound]
\label{thm:polynomial-lower-bound}
Under Assumptions~\ref{ass:sparse-assortative-sbm} and \ref{ass:connectivity}, suppose
\[
1<\rho_0\le a/b-\kappa,
\]
where \(\kappa>0\) is fixed. Then for every fixed \(\varepsilon>0\), there is \(C=C(a,b,\kappa,\varepsilon)>0\) such that, for all sufficiently large \(N\),
\[
\P\left(\gR_{\lfloor N^{I_0-\varepsilon}\rfloor}\neq\emptyset\right)
\ge 1-CN^{-\varepsilon/2}.
\]
\end{theorem}

We summarize our results in a phase diagram, displayed in Figure~\ref{fig:phase-diagram}.
\begin{figure}[H]
\centering
\begin{tikzpicture}[
    x=1cm,
    y=1cm,
    axis/.style={-{Latex[length=4mm, width=3mm]}, thick},
    tick/.style={black, line width=0.8pt},
    guide/.style={black!65, densely dashed, line width=0.55pt},
    regime/.style={line width=4pt},
    unsure/.style={regime, gray!85},
    poly/.style={regime, blue!70!black},
    polylog/.style={regime, green!55!black},
    3day/.style={regime, magenta!80!violet},
    2day/.style={regime, orange!85!black},
    1day/.style={regime, cyan!80!blue},
    toplabel/.style={font=\scriptsize, align=center, anchor=south},
    botlabel/.style={font=\scriptsize, align=center, anchor=north}
]

\def\xA{-2.0}
\def\xB{-1.0}
\def\xC{0.0}
\def\xD{1.0}
\def\xE{5.4}
\def\xF{7.0}
\def\xG{8.6}
\def\xH{9.8}
\def\xI{12.6}
\def\xJ{14.0}
\def\yTopLabel{0.68}
\def\yBottomLabel{-0.68}

\draw[axis] (\xB,0) -- (\xJ+0.4,0);

\draw[unsure]      (\xB,0) -- (\xD,0);
\draw[poly] (\xD,0) -- (\xE,0);
\draw[polylog]  (\xE,0) -- (\xF,0);
\draw[3day]    (\xF,0) -- (\xH,0);
\draw[2day]    (\xH,0) -- (\xI,0);
\draw[1day]    (\xI,0) -- (\xJ,0);

\foreach \x in {\xC, \xG} {
    \draw[guide, -{Latex[length=1.5mm, width=1.2mm]}]
        (\x,0.20) -- (\x,0.68);
}
\foreach \x in {\xD, \xF, \xH, \xI} {
    \draw[guide, -{Latex[length=1.5mm, width=1.2mm]}]
        (\x,-0.20) -- (\x,-0.68);
}
\foreach \x in {\xC, \xD, \xF, \xG, \xH, \xI} {
    \draw[tick] (\x,0.2) -- (\x,-0.2);
}

\node[toplabel] at (\xC,\yTopLabel) {$\rho_0 =1$};
\node[toplabel] at ({(\xD+\xE)/2},\yTopLabel) {$1<\rho_0 < a/b-\kappa$};
\node[toplabel] at (\xG,\yTopLabel) {$\rho_0 =a/b$};

\draw[decorate, decoration={brace, amplitude=5pt, mirror=false}]
    (\xD,0.25) -- (\xE,0.25);
\draw[decorate, decoration={brace, amplitude=5pt, mirror=false}]
    (\xE,0.25) -- (\xG,0.25);
\node[toplabel, font=\tiny, xshift=-6pt] at ({(\xE+\xG)/2},\yTopLabel)
    {$\rho_0=a/b-o(1)$};


\node[botlabel] at ({\xF},\yBottomLabel) {$\frac{\widetilde{\Delta}_0}{N} = \frac{-H}{\sqrt{\log N}}$};
\node[botlabel] at ({\xH},\yBottomLabel) {$\frac{\widetilde{\Delta}_0}{N} = \frac{K_2(a,b)}{\sqrt{\log N}}$};
\node[botlabel] at ({\xI},\yBottomLabel)
    {$\frac{\widetilde{\Delta}_0}{N}=b-(a+b)r_{*}$};

\draw[decorate, decoration={brace, amplitude=5pt, mirror}]
    (\xB,-0.25) -- (\xD,-0.25);
\node[botlabel] at ({(\xB+\xD)/2},\yBottomLabel)
    {$\frac{\Delta_0}{N} \lesssim\frac{1}{\sqrt{N\log N}}$};
\draw[decorate, decoration={brace, amplitude=5pt, mirror}]
    (\xD,-0.25) -- (\xE,-0.25);
\node[botlabel] at ({(\xD+\xE)/2},\yBottomLabel)
    {$\frac{\Delta_0}{N} \gg \frac{1}{\sqrt{N\log N}}$};

    \begin{scope}[shift={(0,-2.85)}]

    \draw[1day] (0.0,0.25) -- (0.7,0.25);
    \node[font=\scriptsize, anchor=west] at (0.85,0.25) {Thm. \ref{thm:three-day-unanimity} (iii): $1$ update};

    \draw[2day] (4.7,0.25) -- (5.4,0.25);
    \node[font=\scriptsize, anchor=west] at (5.55,0.25) {Thm. \ref{thm:three-day-unanimity} (ii): $2$ updates};

    \draw[3day] (10.0,0.25) -- (10.7,0.25);
    \node[font=\scriptsize, anchor=west] at (10.85,0.25) {Thm. \ref{thm:three-day-unanimity} (i): $3$ updates};

    \draw[polylog] (0.0,-0.45) -- (0.7,-0.45);
    \node[font=\scriptsize, anchor=west] at (0.85,-0.45) {Thm. \ref{thm:super-polylog-unanimity}: subpolynomial};

    \draw[poly] (4.7,-0.45) -- (5.4,-0.45);
    \node[font=\scriptsize, anchor=west] at (5.55,-0.45) {Thms. \ref{thm:polynomial-minimal-bias}--\ref{thm:polynomial-lower-bound}: polynomial};

    \draw[unsure] (10.0,-0.45) -- (10.7,-0.45);
    \node[font=\scriptsize, anchor=west] at (10.85,-0.45) {unexplored};

    \end{scope}

\end{tikzpicture}
\caption{Schematic phase diagram summarizing the main results. The lengths of the regimes are illustrative only. The initial camp ratio \(\rho_0\) is defined in \eqref{eqn:initial-camp-ratio}. The advantages \(\Delta_0\) and \(\widetilde{\Delta}_0\) are defined in \eqref{eqn:unweighted-advantage} and \eqref{eqn:weighted-advantage}, respectively.}
\label{fig:phase-diagram}
\end{figure}
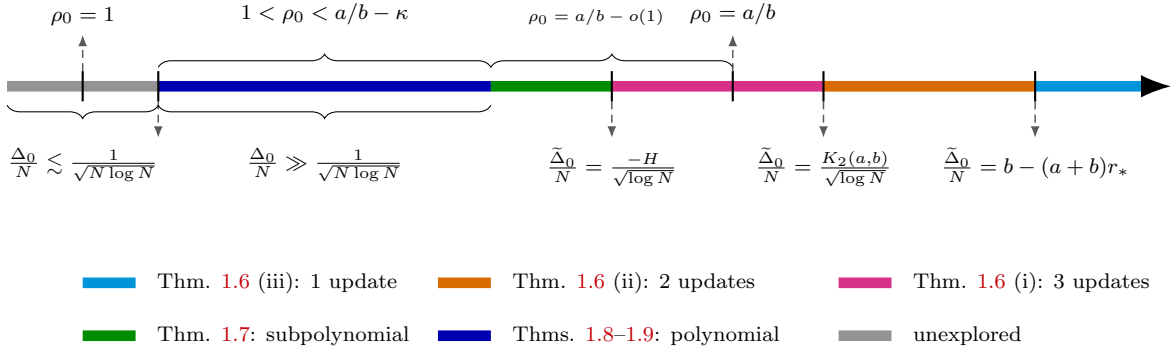
\begin{remark}\label{rem:community-detection-exponent}
    If a symmetric community-size restriction is further imposed, the exponent $I_0$ in
\eqref{eqn:It-exponent} reduces to $I = I(a,b)=\frac12(\sqrt a-\sqrt b)^2$. In the context of community detection, it is well known that exact recovery\footnote{Every vertex is correctly classified into one of the two communities.} can be achieved with high probability if and only if $I>1$; see \cite{Abbe2016ExactRI, Mossel2016ConsistencyTF}. When exact recovery is impossible, i.e., $I < 1$, the \emph{information-theoretic} lower bound on the expected mismatch ratio
is $N^{-I+ o(1)}$; see \cite{Zhang2016MinimaxRO, Abbe2020EntrywiseEA}.
\end{remark}

\subsection{Numerical Experiments}

We conduct numerical experiments to validate the theoretical predictions of the main results.

\paragraph{Constant-Time Transition.}
Theoretical predictions in \Cref{thm:three-day-unanimity} and \Cref{thm:super-polylog-unanimity} are illustrated in Figure~\ref{fig:experiment1}. The first two panels show the predicted sharp rise in constant-time success as \(H_{\mathrm{exp}}\) increases. The three-update transition occurs earlier than the two-update
transition because it permits an additional amplification step, while the
median blue-unanimity time
\(\tau_B\coloneqq\inf\{t\ge0:\gR_t=\emptyset\}\) increases smoothly below the three-update threshold,
consistently with the subpolynomial-time regime.

\begin{figure}[H]
    \centering
    \begin{subfigure}[t]{0.32\textwidth}
        \centering
        \includegraphics[width=\textwidth]{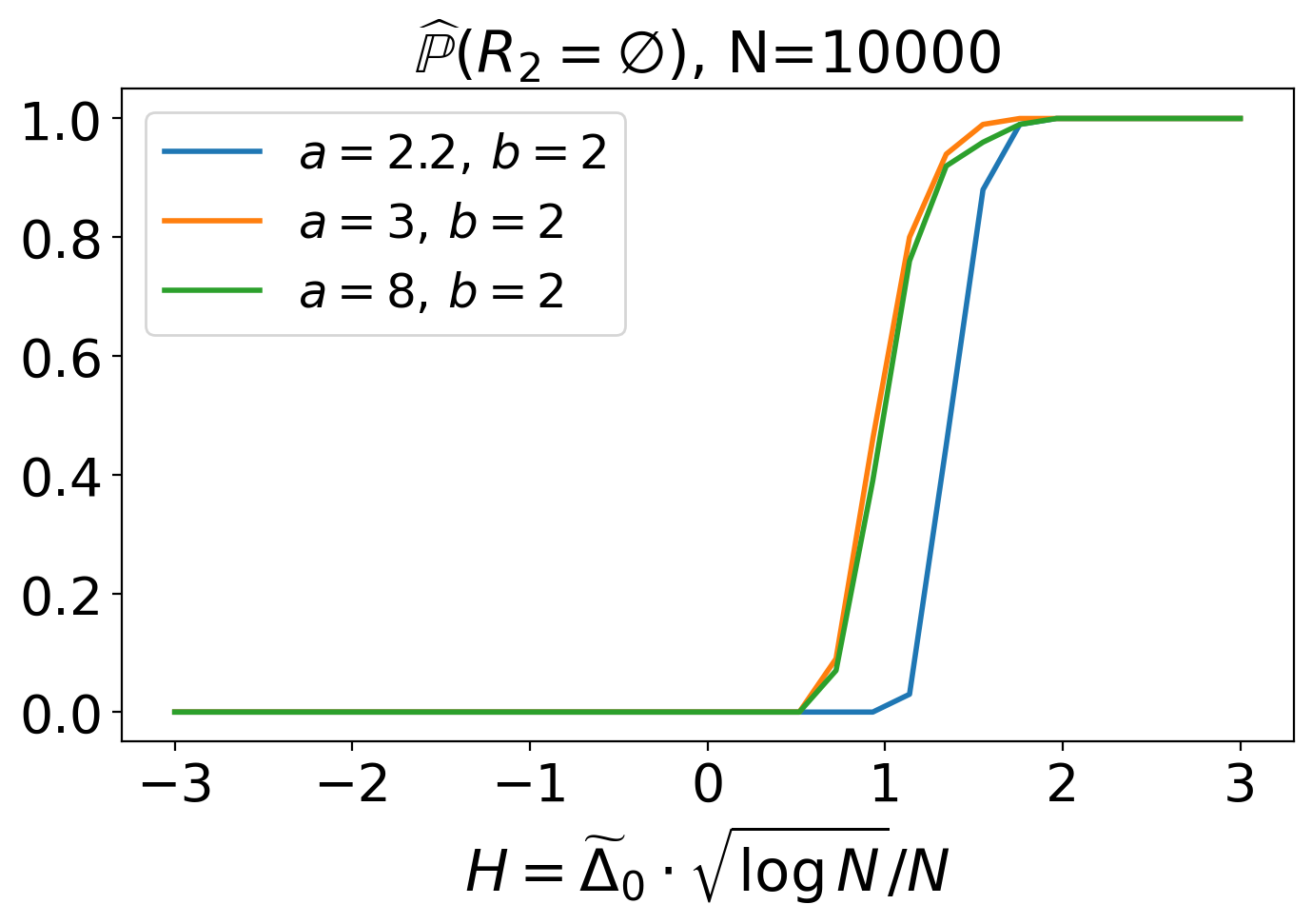}
    \end{subfigure}
    \hfill
    \begin{subfigure}[t]{0.32\textwidth}
        \centering
        \includegraphics[width=\textwidth]{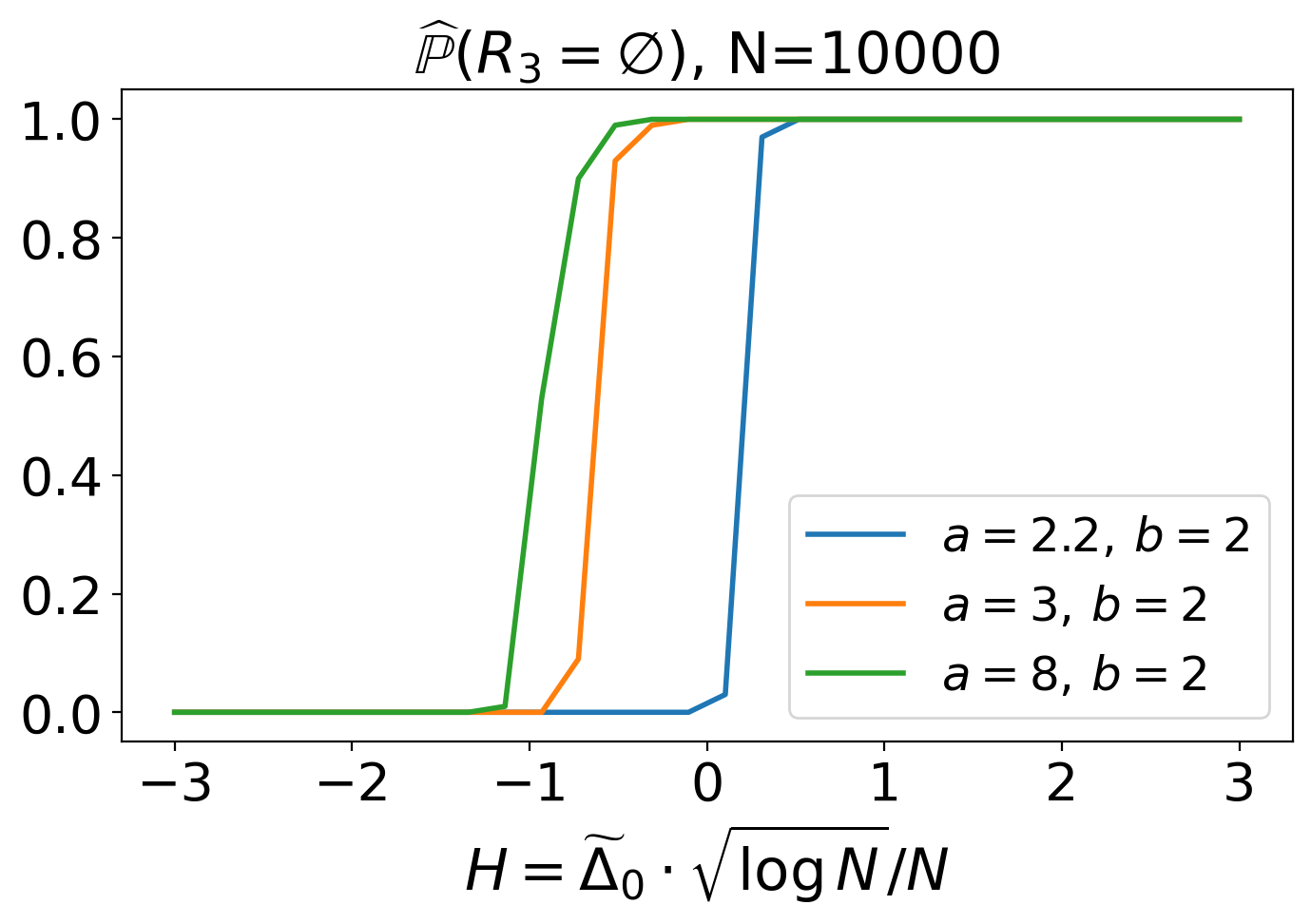}
    \end{subfigure}
    \hfill
    \begin{subfigure}[t]{0.32\textwidth}
        \centering
        \includegraphics[width=\textwidth]{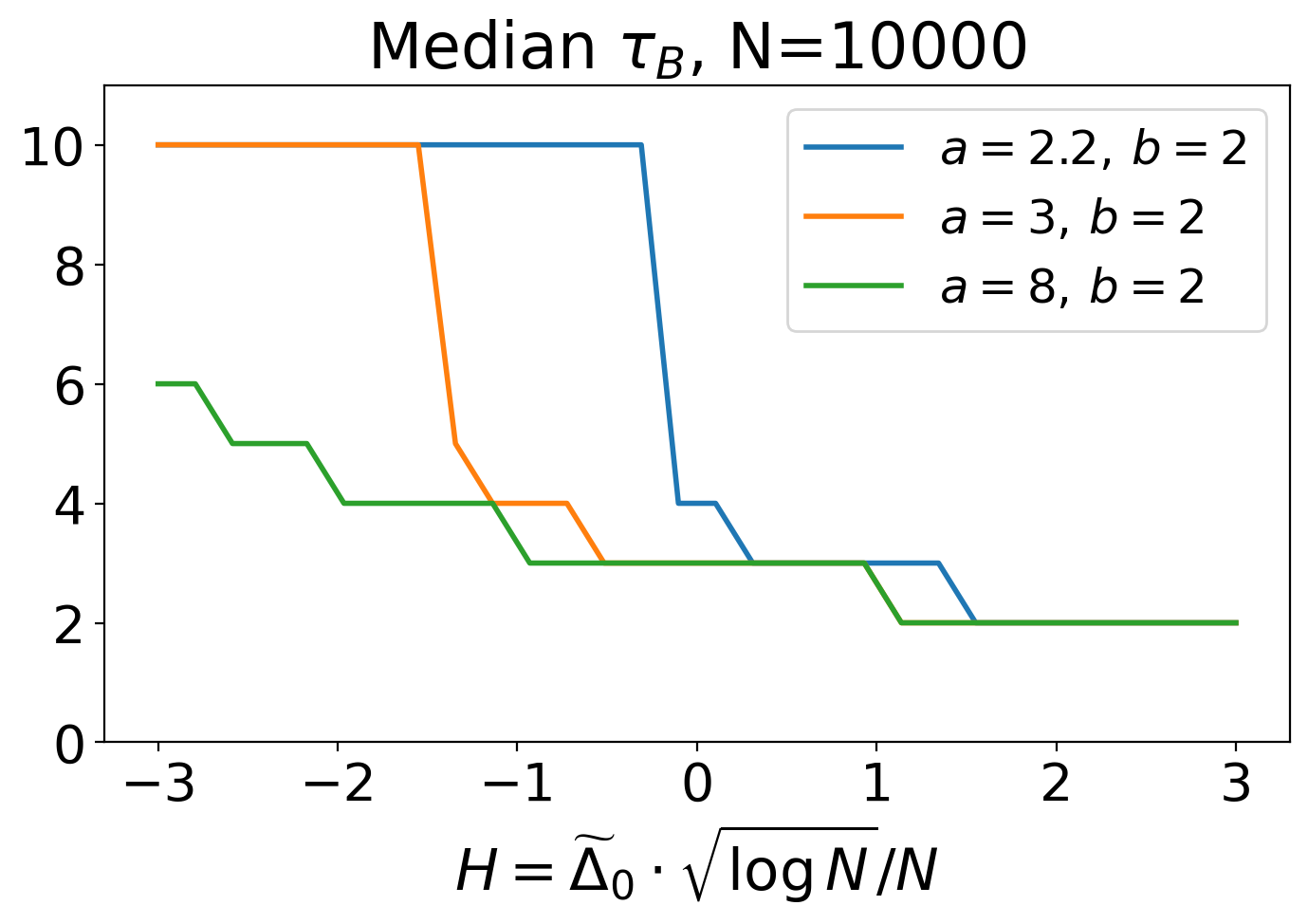}
    \end{subfigure}
    \caption{We take \(N=10^4\), \(b=2\), and vary \(a\in\{2.2,3,8\}\). We parametrize the initial condition by \(H_{\mathrm{exp}}\coloneqq\widetilde{\Delta}_0\sqrt{\log N}/N\), labeled \(H\) in the panels. Thus \(H_{\mathrm{exp}}=-H\) on the boundary \(\widetilde{\Delta}_0=-HN/\sqrt{\log N}\) in Theorem~\ref{thm:three-day-unanimity}\textup{(i)}. The first two panels report the
    estimated probabilities of blue unanimity by \(t=2\) and \(t=3\),
    respectively; the third reports the median \(\tau_B\).}
    \label{fig:experiment1}
\end{figure}

\paragraph{Polynomial-Time Exponent.}
The polynomial-time convergence exponent predicted by
\Cref{thm:polynomial-minimal-bias} and
\Cref{thm:polynomial-lower-bound} is tested in
Figure~\ref{fig:experiment2}. Since
\(\rho_0=|\gB_0|/|\gR_0|\) and
\(|\gB_0|+|\gR_0|=N\), we have
\(|\gR_0|/N=(1+\rho_0)^{-1}\) and
\(|\gB_0|/N=\rho_0(1+\rho_0)^{-1}\). Substitution into
\eqref{eqn:It-exponent} gives
\[
I_0
=\frac{\bigl(\ReLU(\sqrt a-\sqrt{b\rho_0})\bigr)^2}{1+\rho_0}
=\frac{(\sqrt a-\sqrt{b\rho_0})^2}{1+\rho_0},
\]
where the second equality uses \(1<\rho_0<a/b\). The empirical exponent
follows \(I_0\) across this regime, with visible finite-\(N\) deviations near
the transition points. The trajectories give the complementary picture: the
red camp typically decays slowly before the final extinction step, matching
the polynomial-time behavior predicted by the upper and lower bounds.

\begin{figure}[H]
    \centering
    \begin{subfigure}[t]{0.48\textwidth}
        \centering
        \includegraphics[width=\textwidth]{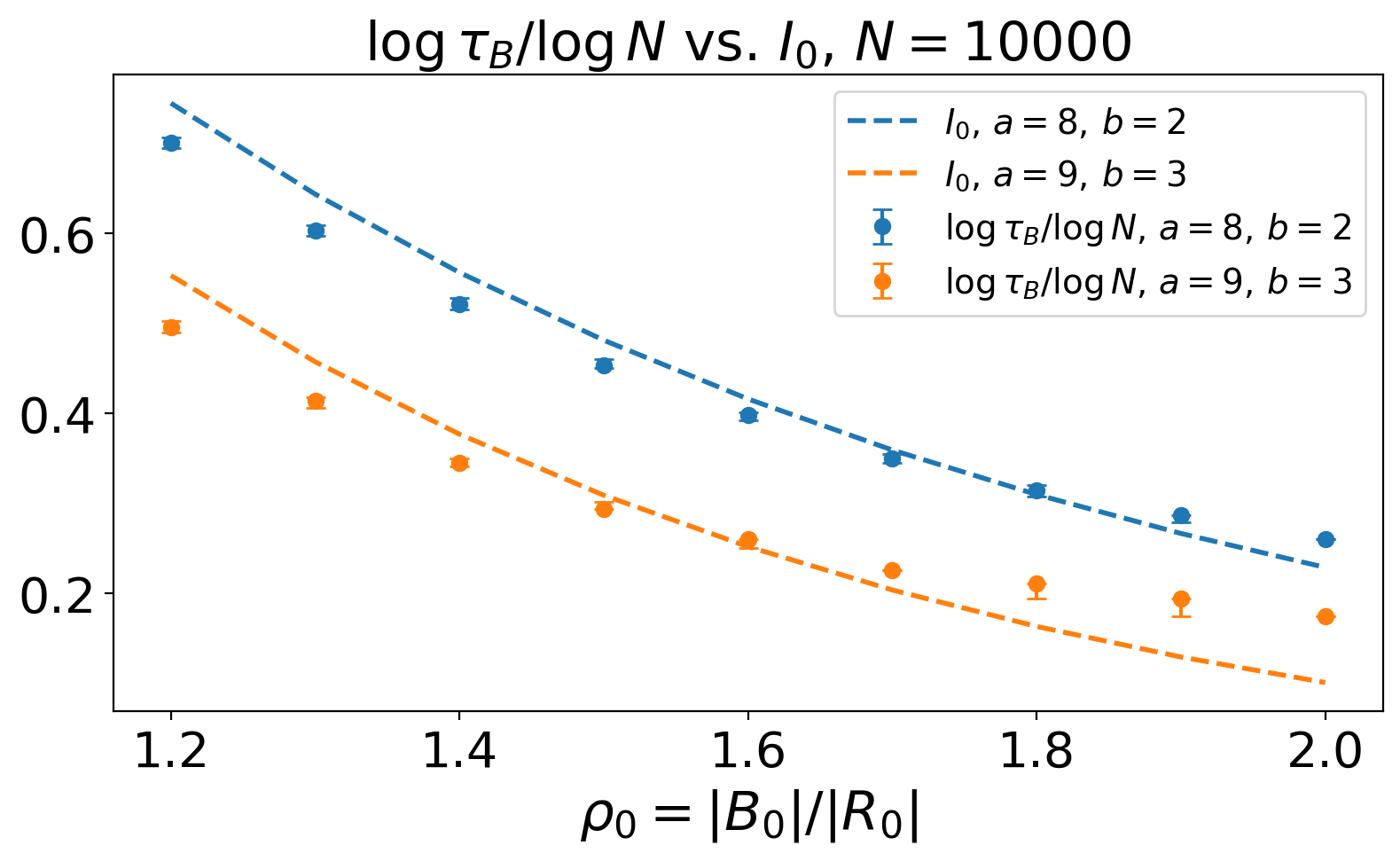}
    \end{subfigure}
    \hfill
    \begin{subfigure}[t]{0.48\textwidth}
        \centering
        \includegraphics[width=\textwidth]{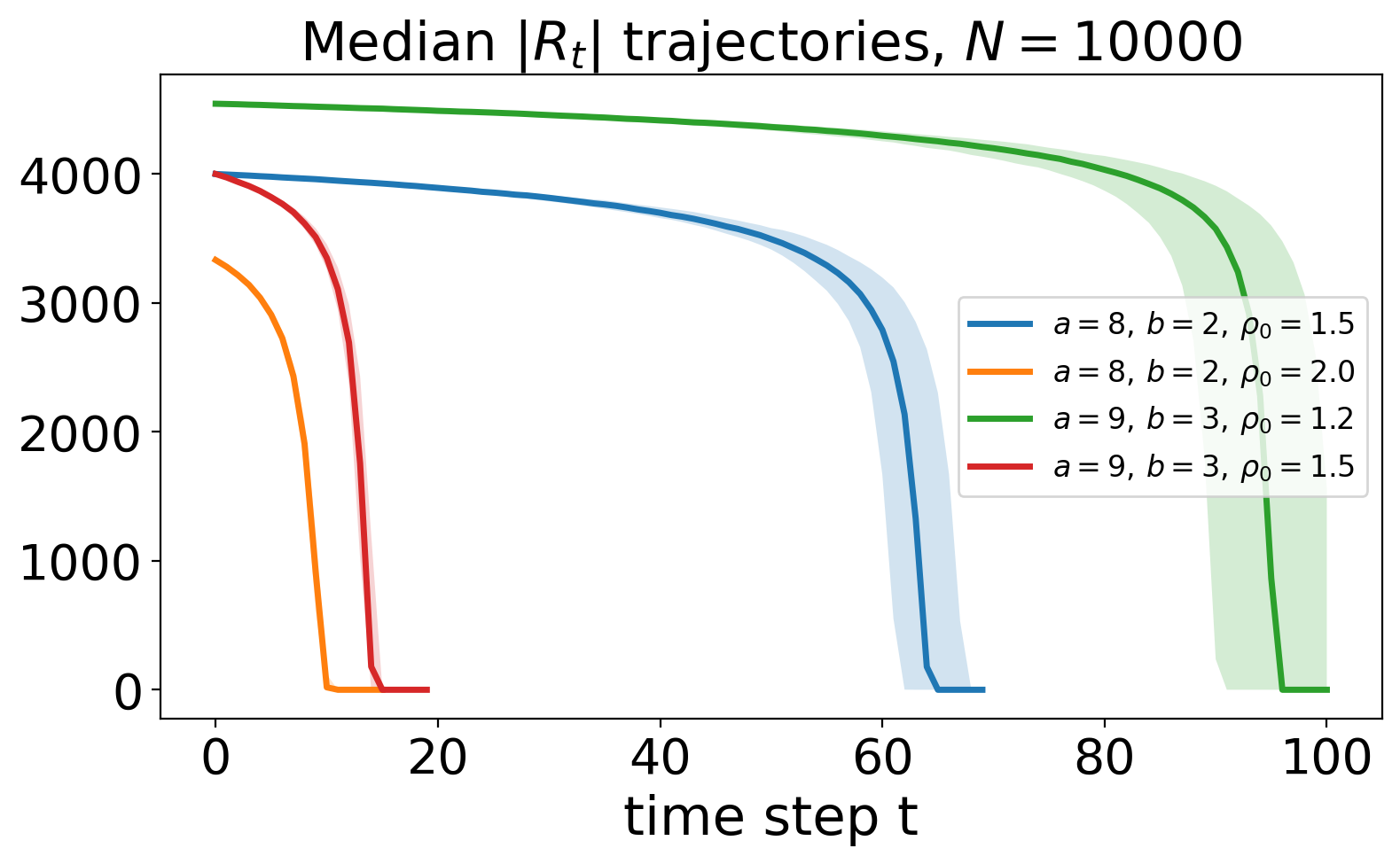}
    \end{subfigure}
    \caption{We take \(N=10^4\),
    \((a,b)\in\{(8,2),(9,3)\}\), and vary the initial camp ratio \(\rho_0\).
    The left panel compares the median empirical exponent
    \(\log\tau_B/\log N\) with
    \(I_0\) and includes the
    \(10\%\text{--}90\%\) quantile interval. The right panel plots the median
    trajectories of \(|\gR_t|\).}
    \label{fig:experiment2}
\end{figure}

\subsection{Notation}
All asymptotic statements are taken as \(N\to\infty\), and all logarithms are
natural. We write \([N]=\{1,\ldots,N\}\), use \(\indi{E}\) for the indicator of
an event \(E\), and denote probability, expectation, and variance by
\(\P\), \(\E\), and \(\Var\), respectively. The constants \(c,C>0\) may change from line to line and may depend on the fixed parameters under consideration, but never on \(N\). For real
sequences \(u_N,v_N\), we write \(u_N\lesssim v_N\) if
\(u_N\le Cv_N\) for all sufficiently large \(N\), and write
\(u_N\gtrsim v_N\) for the reverse inequality. We write
\(u_N\asymp v_N\) when both comparisons hold. For eventually nonnegative
sequences \(u_N,v_N\), the relations \(u_N\ll v_N\) and \(u_N\gg v_N\)
mean \(u_N/v_N\to0\) and \(u_N/v_N\to\infty\), respectively. When \(u_N\)
may be signed and only its magnitude is relevant, we write
\(|u_N|\ll v_N\) or \(|u_N|\gg v_N\) explicitly.
We use \(O,\Omega,\Theta,o,\omega\) with their standard meanings.

\subsection{Organization}
The rest of the paper is organized as follows.
Section~\ref{sec:one-step-evolution-advantages} collects the one-step
estimates for the weighted and unweighted advantages used throughout the
paper. The proofs of
Theorems~\ref{thm:three-day-unanimity},
\ref{thm:super-polylog-unanimity},
\ref{thm:polynomial-minimal-bias}, and
\ref{thm:polynomial-lower-bound} are contained in
Sections~\ref{sec:constant-days-to-unanimity},~\ref{sec:polylogarithmic-days-to-unanimity},~\ref{sec:polynomial-days-to-unanimity}, and \ref{sec:necessary-conditions}, respectively. The proofs of supporting lemmas and technical results are collected in the appendix.

\section{One-Step Evolution of the Advantages}
\label{sec:one-step-evolution-advantages}

The proofs of our main results rely heavily on a detailed analysis of the
one-step evolution of the advantages. We first introduce the necessary notation
before stating the main lemmas. The proofs of all lemmas in this section are
deferred to Appendix~\ref{app:deferred-one-step-evolution-advantages}.

For a blue vertex \(u\in\gB_t\) and a red vertex \(v\in\gR_t\), we define the
degree differences at \(u\) and \(v\), respectively, by
\begin{align}
    \rD_{t}^{\gB}(u) \coloneqq N^{\gR}_{t}(u) - N^{\gB}_{t}(u), \qquad
    \rD_{t}^{\gR}(v) \coloneqq N^{\gB}_{t}(v) - N^{\gR}_{t}(v).
    \label{eqn:degree-differences}
\end{align}
Recall the update rule in \eqref{eqn:opinion-update}. Conditioned on \(\rvy_t\), \(u\) flips to red at update \(t+1\) if and only if \(\rD_t^\gB(u)>0\), while \(v\) flips to blue if and only if \(\rD_t^\gR(v)>0\). Since these conditional laws depend only on the camp sizes, we omit the vertex label and write
\begin{subequations}
\begin{align}
    \rD_{t}^{\gB}\stackrel d=&\, \Bin(|\gR_{t}|,\beta)-\Bin(|\gB_{t}|-1,\alpha),
    \label{eqn:D-Bt-distribution}\\
    \rD_{t}^{\gR} \stackrel d=&\, \Bin(|\gB_{t}|,\beta)-\Bin(|\gR_{t}|-1,\alpha).
    \label{eqn:D-Rt-distribution}
\end{align}
\end{subequations}
The corresponding one-vertex flip and non-flip probabilities at update $t+1$ are
\begin{subequations}
\begin{align}
    p_{t}^{\gB} \coloneqq &\, \P(\rD_{t}^{\gB}>0\mid\rvy_{t}), \qquad
    q_{t}^{\gB} \coloneqq \P(\rD_{t}^{\gB}\le0\mid\rvy_{t}),
    \label{eqn:p-q-B}\\
    p_{t}^{\gR} \coloneqq &\, \P(\rD_{t}^{\gR}>0\mid\rvy_{t}), \qquad
    q_{t}^{\gR} \coloneqq \P(\rD_{t}^{\gR}\le0\mid\rvy_{t}),
    \label{eqn:p-q-R}
\end{align}
\end{subequations}
where $p_{t}^{\gB}$ and $p_{t}^{\gR}$ are the probabilities that a blue vertex flips to red and a red vertex flips to blue, respectively, while $q_{t}^{\gB}$ and $q_{t}^{\gR}$ are the corresponding non-flip probabilities.
Furthermore, we denote the numbers of blue-to-red and red-to-blue flips at update \(t+1\) by
\begin{align}
\gBR_t \coloneqq &\, \left|\gB_t\cap\gR_{t+1}\right|,
\qquad
\gRB_t \coloneqq \left|\gR_t\cap\gB_{t+1}\right|.
\label{eqn:gBR-gRB-definition}
\end{align}

Each red-to-blue flip increases \(\Delta_t\) by \(2\) and
\(\widetilde{\Delta}_t\) by \(a+b\), while each blue-to-red flip produces the
opposite changes. Thus the balance between \(\gRB_t\) and \(\gBR_t\)
determines the one-step drift of both advantages, motivating the aggregate
flip estimates below.

\subsection{Bounds on the Expected Number of Flips}

We begin by comparing the expected flip counts in the two directions.
Lemma~\ref{lem:flip-ratio-comparison} shows that, when \(a>b\), a blue
majority makes the expected number of red-to-blue flips quantitatively larger
than the expected number of blue-to-red flips, thereby producing positive
drift toward blue unanimity.
\begin{lemma}[Flip ratio comparison]
    \label{lem:flip-ratio-comparison}
    Under Assumptions~\ref{ass:sparse-assortative-sbm} and
    \ref{ass:connectivity}, fix a deterministic time \(t\ge0\), condition on
    \(\rvy_t\), and suppose that \(|\gB_t|\ge|\gR_t|\ge1\). Define
    \(\theta_t\) by
    \begin{align}
    \theta_{t} \coloneqq \min\{1,\Delta_{t} \log (N)/N\}. \label{eqn:theta-definition}
    \end{align}
    Then there exists a constant \(c_0=c_0(a,b)>0\) such that
    \begin{align*}
    \frac{|\gB_{t}|p_{t}^{\gB}}{|\gR_{t}|p_{t}^{\gR}}
    &\le
    1-c_0\theta_t,
    \end{align*}
    for all sufficiently large \(N\).
\end{lemma}
    
Conditioning on \(\rvy_t\), we denote the expected total number of vertices that flip at update \(t+1\) by
\begin{align}
\nu_t
&\coloneqq
\E\left[\gRB_t+\gBR_t\,\middle|\,\rvy_t\right].
\label{eqn:nu-definition}
\end{align}
We denote the conditional expected increments of the unweighted and weighted advantages by
\begin{align}
\mu_t
&\coloneqq
\E[\Delta_{t+1}-\Delta_t\mid\rvy_t],
&
\widetilde{\mu}_t
&\coloneqq
\E[\widetilde{\Delta}_{t+1}-\widetilde{\Delta}_t\mid\rvy_t].
\label{eqn:mu-tilde-mu-definition}
\end{align}
By linearity of expectation and \eqref{eqn:affine-delta-tilde-delta}, these
quantities satisfy
\begin{align}
    \widetilde{\mu}_t=\frac{a+b}{2}\mu_t.
    \label{eqn:mu-tilde-mu-linearity}
\end{align}

Lemma~\ref{lem:one-step-difference-bound} establishes the tail-free
comparisons between the expected advantage increments and the expected total
number of flips through the one-vertex flip probabilities.

\begin{lemma}[Tail-free one-step comparison]
\label{lem:one-step-difference-bound}
Under Assumptions~\ref{ass:sparse-assortative-sbm} and
\ref{ass:connectivity}, fix a deterministic time \(t\ge0\), condition on
\(\rvy_t\), and suppose that \(|\gB_t|\ge|\gR_t|\asymp N\). There are
constants \(c,C>0\), depending only on \(a,b\), such that
\begin{align}
c|\gR_t|p_t^\gR\theta_t
\le \mu_t\le C|\gR_t|p_t^\gR,
\qquad
c|\gR_t|p_t^\gR
\le \nu_t\le C|\gR_t|p_t^\gR,
\label{eqn:mu-nu-tail-free-bound}
\end{align}
where \(\theta_t\) is defined in \eqref{eqn:theta-definition}. Consequently,
\begin{align}
\mu_t/\nu_t\ge c\theta_t.
\label{eqn:mu-nu-comparison-bound}
\end{align}
\end{lemma}

The tail-free comparison reduces the aggregate one-step analysis to estimating
the one-vertex flip probability \(p_t^{\gR}\). Its magnitude varies with the
current configuration and determines whether the resulting advantage increment
is linear, almost linear, or sublinear in \(N\).

\subsection{Scale of the Advantage Increments}

We distinguish these scales through the signal-to-noise ratios of the
one-vertex degree differences. For this purpose, define their conditional
means and variances by
\begin{subequations}
\begin{align}
    m_{t}^{\gB} \coloneqq &\, \E\left[\rD_{t}^{\gB}\mid \rvy_{t}\right],
    \qquad
    v_{t}^{\gB} \coloneqq \Var\left(\rD_{t}^{\gB}\mid \rvy_{t}\right),
    \label{eqn:m-v-B}\\
    m_{t}^{\gR} \coloneqq &\,\E\left[\rD_{t}^{\gR}\mid \rvy_{t}\right],
    \qquad
    v_{t}^{\gR} \coloneqq \Var\left(\rD_{t}^{\gR}\mid \rvy_{t}\right).
    \label{eqn:m-v-R}
\end{align}
\end{subequations}
For each color \(\gC\in\{\gB,\gR\}\), the magnitude of
\(\lvert m_t^{\gC}\rvert/\sqrt{v_t^{\gC}}\) selects the appropriate Gaussian,
moderate-deviation, or large-deviation estimate from
Appendix~\ref{sec:flip-probability-analysis}. Substituting that estimate into
Lemma~\ref{lem:one-step-difference-bound} yields bounds for the conditional
mean, and Lemma~\ref{lem:R1-read-two-concentration} converts them into
high-probability bounds for the realized increment. 

First, when \(\lvert m_t^{\gC}\rvert/\sqrt{v_t^{\gC}}\) is bounded, the advantage increments grow linearly.
\begin{lemma}[Linear-jump]
\label{lem:linear-jump-one-step-bound}
Under Assumptions~\ref{ass:sparse-assortative-sbm} and
\ref{ass:connectivity}, fix a deterministic time \(t\ge0\), condition on
\(\rvy_t\), and suppose that \(\Delta_t\ge0\). For every fixed
\(H<\infty\), if
\[
\widetilde{\Delta}_t\ge-H\frac{N}{\sqrt{\log N}},
\]
then there exists a constant \(c_H=c_H(a,b,H)>0\) such that
\(0\le\nu_t\le N\), and
\begin{align}
c_HN-\widetilde{\Delta}_t
&\le \widetilde{\mu}_t\le bN-\widetilde{\Delta}_t,
&
\frac{2}{a+b}\left(c_HN-\widetilde{\Delta}_t\right)
&\le \mu_t
\le \frac{2}{a+b}\left(bN-\widetilde{\Delta}_t\right).
\label{eqn:linear-jump-expectation-bound}
\end{align}
Moreover, with probability at least \(1-2\exp(-c_HN)\),
\begin{align}
c_HN\le\widetilde{\Delta}_{t+1}\le bN,
\qquad
c_HN\le\Delta_{t+1}\le N,
\label{eq:linear-jump-bound}
\end{align}
for all sufficiently large \(N\).
\end{lemma}

Next, when \(\lvert m_t^{\gC}\rvert/\sqrt{v_t^{\gC}}\) lies in the
moderate-deviation range, i.e.,
\(1\lesssim\lvert m_t^{\gC}\rvert/\sqrt{v_t^{\gC}}\ll\sqrt{\log N}\), the
advantage increments grow almost linearly.
\begin{lemma}[Almost-linear-jump]
\label{lem:almost-linear-jump-one-step-bound}
Under Assumptions~\ref{ass:sparse-assortative-sbm} and
\ref{ass:connectivity}, fix a deterministic time \(t\ge0\), condition on
\(\rvy_t\), and suppose that \(\Delta_t\ge0\). Let \(h_N\to\infty\) satisfy
\(h_N=o(\sqrt{\log N})\), and define
\[
u_t\coloneqq-\widetilde{\Delta}_t\sqrt{\log N}/N.
\]
For fixed \(0<\ell\le L<\infty\), suppose that
\(\ell\le u_t\le Lh_N\). Then \(|\gB_t|,|\gR_t|\asymp N\),
\(\Delta_t\asymp N\), \(\theta_t=1\), and there are constants \(c,C>0\),
depending only on \(a,b,\ell,L\), such that
\begin{align}
cN\frac{\exp(-Cu_t^2)}{1+u_t}
&\le\mu_t,\widetilde{\mu}_t,\nu_t
\le CN\frac{\exp(-cu_t^2)}{1+u_t}.
\label{eqn:mu-nu-moderate-bound}
\end{align}
Moreover, conditionally on \(\rvy_t\), with probability at least
\(
1-2\exp\left(-cN\exp(-Cu_t^2)(1+u_t)^{-2}\right)
\),
\begin{align}
cN\frac{\exp(-Cu_t^2)}{1+u_t}
&\le\widetilde{\Delta}_{t+1}-\widetilde{\Delta}_t
\le CN\frac{\exp(-cu_t^2)}{1+u_t},
\notag\\
cN\frac{\exp(-Cu_t^2)}{1+u_t}
&\le\Delta_{t+1}-\Delta_t
\le CN\frac{\exp(-cu_t^2)}{1+u_t}.
\label{eq:almost-linear-jump-bound}
\end{align}
In particular, if \(u_t\asymp h_N\), the one-step increase has scale
\(N\exp\{-\Theta(h_N^2)\}\), up to the displayed Mills-ratio factor.
\end{lemma}

Finally, consider the large-deviation regime
\(
-m_t^{\gR}/\sqrt{v_t^{\gR}}
=
\Theta(\sqrt{\log N})
\),
where the advantage increments grow sublinearly.

\begin{lemma}[Sublinear-jump]
\label{lem:sublinear-jump-one-step-bound}
Under Assumptions~\ref{ass:sparse-assortative-sbm} and
\ref{ass:connectivity}, fix a deterministic time \(t\ge0\), condition on
\(\rvy_t\), and suppose that \(\Delta_t\ge0\). For some fixed \(c_*>0\),
assume that
\begin{align}
\widetilde{\Delta}_t\le-c_*N.
\label{eqn:sublinear-jump-condition}
\end{align}
For every fixed \(\eta>0\), there are constants \(c,C>0\), depending only on
\(a,b,c_*,\eta\), such that
\begin{subequations}
\begin{align}
c\theta_t N^{1-I_t-\eta}
&\le\mu_t\le CN^{1-I_t+\eta},
\label{eqn:mu-large-deviation-bound}\\
c\theta_t N^{1-I_t-\eta}
&\le\widetilde{\mu}_t\le CN^{1-I_t+\eta},
\label{eqn:tilde-mu-large-deviation-bound}\\
cN^{1-I_t-\eta}
&\le\nu_t\le CN^{1-I_t+\eta},
\label{eqn:nu-large-deviation-bound}
\end{align}
\end{subequations}
where \(I_t\) is defined in \eqref{eqn:It-exponent}. Moreover, conditionally on \(\rvy_t\), with probability at least
\(
1
-2\exp\big(-c\theta_t^2N^{1-2I_t-2\eta}\big)
-2\exp\big(-cN^{1-2I_t+2\eta}\big)
\), we have
\begin{align}
c\theta_t N^{1-I_t-\eta}
&\le\widetilde{\Delta}_{t+1}-\widetilde{\Delta}_t
\le CN^{1-I_t+\eta},
\notag\\
c\theta_t N^{1-I_t-\eta}
&\le\Delta_{t+1}-\Delta_t
\le CN^{1-I_t+\eta}.
\label{eq:sublinear-jump-bound}
\end{align}
\end{lemma}
Below, we briefly explain why it is sufficient to identify the exponent of the
red-to-blue flip probability \(p_t^{\gR}\) with \(I_t\) in
\eqref{eqn:It-exponent}. The main reason is that the blue-to-red flip
probability \(p_t^{\gB}\) is not of larger polynomial order than
\(p_t^{\gR}\). In principle, \(p_t^{\gR}\) and \(p_t^{\gB}\) involve two
different large-deviation exponents. However,
Lemma~\ref{lem:one-step-difference-bound} gives
\[
\nu_t\asymp|\gR_t|p_t^{\gR},
\qquad
c|\gR_t|p_t^{\gR}\theta_t
\le\mu_t\le
C|\gR_t|p_t^{\gR}.
\]
Thus \(p_t^{\gR}\) determines the one-step scale, while \(p_t^{\gB}\) enters
only through the comparison already incorporated into the tail-free bound.
To verify that \(p_t^{\gB}\) is not of a larger polynomial order, define
\[
A_t^\gR\coloneqq\frac{a(|\gR_t|-1)}{N},
\qquad
B_t^\gB\coloneqq\frac{b|\gB_t|}{N},
\qquad
A_t^\gB\coloneqq\frac{a(|\gB_t|-1)}{N},
\qquad
B_t^\gR\coloneqq\frac{b|\gR_t|}{N},
\]
and let
\[
I_t^\gR\coloneqq I(A_t^\gR,B_t^\gB),
\qquad
I_t^\gB\coloneqq I(A_t^\gB,B_t^\gR).
\]
Corollary~\ref{cor:large-deviation-tail-flip-probabilities} gives
\[
p_t^\gR=N^{-I_t^\gR+o(1)},
\qquad
p_t^\gB=N^{-I_t^\gB+o(1)}.
\]
Replacing \(|\gR_t|-1\) by \(|\gR_t|\) changes the first rate by
\(O(N^{-1/2})\), and hence
\[
I_t^\gR=I_t+O(N^{-1/2}).
\]

It remains to compare \(I_t^\gB\) with \(I_t\). Write
\(x=|\gB_t|/N\) and \(y=|\gR_t|/N\), and first omit the \(-1\) correction
in \(A_t^\gB\). Since \(\Delta_t\ge0\), we have \(x\ge y\), so
\[
\bigl(\sqrt{ax}-\sqrt{by}\bigr)
-\bigl(\sqrt{ay}-\sqrt{bx}\bigr)
=
(\sqrt a+\sqrt b)(\sqrt x-\sqrt y)
\ge0.
\]
Because \(z\mapsto (z_+)^2\) is nondecreasing,
\[
I(ax,by)\ge I(ay,bx)=I_t.
\]
Restoring the \(-1\) correction in \(A_t^\gB\) changes the rate by at most
\(O(N^{-1/2})\), and therefore
\[
I_t^\gB\ge I_t-O(N^{-1/2}).
\]
Consequently, the rate function \(I_t\) from \eqref{eqn:It-exponent} captures the red-to-blue flip scale, while the
blue-to-red probability cannot introduce a larger polynomial scale. This
justifies stating the preceding bounds solely in terms of \(I_t\).

\section{Constant-Time Blue Unanimity}\label{sec:constant-days-to-unanimity}

We present the proof of Theorem~\ref{thm:three-day-unanimity} in this section.
The proofs of the supporting lemmas are deferred to
Appendix~\ref{app:deferred-constant-days}.

The proof proceeds backward from unanimity. We first identify a red-camp size
below which all remaining red vertices disappear in one update. We then show
that a sufficiently positive weighted advantage reaches this extinction
region in one update. The linear-jump estimate supplies one additional step
under the weakest hypothesis of
Theorem~\ref{thm:three-day-unanimity}. To proceed, we introduce the following
lemmas.

\begin{lemma}[A small red camp disappears in one update]
\label{lem:third-day-extinction}
Fix a deterministic time \(t\ge0\). Suppose that, for some \(0<r<r_*\) with \(r_*\) defined in \eqref{eqn:r-star}, the following holds:
\begin{align}
|\gR_t|\le rN.
\label{eqn:third-day-extinction-condition}
\end{align}
Then there exists \(\xi=\xi(a,b,r)>0\) such that, for all sufficiently large
\(N\),
\begin{align}
\P\left(
\gR_{t+1}\neq\emptyset
\,\middle|\,
\rvy_{t}
\right)
\le
2N^{-\xi}.
\label{eq:third-day-extinction-prob}
\end{align}
\end{lemma}
For \(K>K_2(a,b)\) with \(K_2(a,b)\) defined in \eqref{eqn:K2-def}, we define
\begin{align}
q_K \coloneqq \Phi\left(-K/\sqrt{2ab/(a+b)}\right).
\label{eqn:qK-def}
\end{align}
The definition of \(K_2(a,b)\) ensures that \(bq_K/(a+b)<r_{*}\) with $r_{*}$ defined in \eqref{eqn:r-star}, so one can choose an extinction density strictly between these two quantities.

\begin{lemma}[One-step reduction below the extinction threshold]
\label{lem:second-day-reduction}
Fix a deterministic time \(t\ge0\), and condition on \(\rvy_t\). Suppose that,
for some \(K> K_2(a,b)\), we have
\begin{align}
\widetilde{\Delta}_t
\ge
K\frac{N}{\sqrt{\log N}}.
\label{eqn:second-day-reduction-condition}
\end{align}
Then, for any \(r\) satisfying
\begin{align}
\frac{b}{a+b}q_K<r<r_{*},
\label{eqn:second-day-extinction-density-condition}
\end{align}
we have, for all sufficiently large \(N\),
\begin{align}
\P\left(
|\gR_{t+1}|\le rN
\,\middle|\,
\rvy_t
\right)
\ge
1-2\exp\big(-\big(\log N\big)^2\big).
\label{eqn:second-day-reduction-inequality}
\end{align}
\end{lemma}

The preceding lemmas form a nested constant-time mechanism.
Lemma~\ref{lem:third-day-extinction} eliminates a sufficiently small red camp
in one update, while Lemma~\ref{lem:second-day-reduction} reaches that
extinction region in one update from a sufficiently positive weighted
advantage. Under the weakest hypothesis, the linear-jump estimate
\eqref{eq:linear-jump-bound} first creates such an advantage, adding one
further update. We now combine these three steps to prove
Theorem~\ref{thm:three-day-unanimity}.

\begin{proof}[Proof of Theorem~\ref{thm:three-day-unanimity}]
We prove the three claims in reverse order, beginning with the strongest
initial condition.

For part~\textup{(iii)}, assume
\eqref{eqn:one-day-unanimity-condition}. By
\eqref{eqn:red-size-from-advantage}, condition
\eqref{eqn:third-day-extinction-condition} holds at time \(t=0\).
Lemma~\ref{lem:third-day-extinction} therefore yields
\[
\P(\gR_1\neq\emptyset)\le 2N^{-\xi}
\]
for some \(\xi=\xi(a,b,r)>0\). This proves part~\textup{(iii)}.

For part~\textup{(ii)}, assume \eqref{eqn:two-day-unanimity-condition} and
choose \(r\) satisfying
\eqref{eqn:second-day-extinction-density-condition}. Let \(\mathcal E_1\) be the event that
\eqref{eqn:third-day-extinction-condition} holds at time \(t=1\).
Condition \eqref{eqn:two-day-unanimity-condition} is precisely
\eqref{eqn:second-day-reduction-condition} at time \(t=0\), so
Lemma~\ref{lem:second-day-reduction}, applied at time \(t=0\), gives
\[
\P(\mathcal E_1^c)
\le 2\exp\big(-\big(\log N\big)^2\big).
\]
On \(\mathcal E_1\), Lemma~\ref{lem:third-day-extinction}, applied at time
\(t=1\), gives a constant \(\xi=\xi(a,b,K)>0\) such that
\[
\P(\gR_2\neq\emptyset\mid\rvy_1)\le2N^{-\xi}.
\]
Taking conditional expectations and splitting according to
\(\mathcal E_1\) completes the proof of part~\textup{(ii)}:
\begin{align*}
\P(\gR_2\neq\emptyset)
&\le \P(\mathcal E_1^c)
+\E\left[
\indi{\mathcal E_1}
\P(\gR_2\neq\emptyset\mid\rvy_1)
\right]\\
&\le
2\exp\big(-\big(\log N\big)^2\big)+2N^{-\xi}.
\end{align*}

Finally, we prove part~\textup{(i)}. Fix \(H<\infty\) and assume
\eqref{eqn:three-day-unanimity-condition}. Fix \(K>K_2(a,b)\) and choose
\(r\) satisfying \eqref{eqn:second-day-extinction-density-condition}. Let
\(\mathcal E_1\) be the event
that \eqref{eqn:second-day-reduction-condition} holds at time \(t=1\), and
let \(\mathcal E_2\) be the event that
\eqref{eqn:third-day-extinction-condition} holds at time \(t=2\).
By \eqref{eqn:affine-delta-tilde-delta} and
\eqref{eqn:three-day-unanimity-condition}, \(\Delta_0>0\) for all sufficiently
large \(N\). Hence the linear-jump bound \eqref{eq:linear-jump-bound} applies at
time \(t=0\), and there is a constant \(c_H=c_H(a,b,H)>0\) such that
\[
\P(\mathcal E_1^c)\le2\exp(-c_HN).
\]
Here we used that the right-hand side of
\eqref{eqn:second-day-reduction-condition} is \(o(N)\), so the lower bound in
\eqref{eq:linear-jump-bound} implies \(\mathcal E_1\) for all sufficiently
large \(N\).
Conditional on \(\mathcal E_1\), Lemma~\ref{lem:second-day-reduction} applies
at time \(t=1\). By the tower property,
\begin{align*}
\P(\mathcal E_1\cap\mathcal E_2^c)
&=
\E\left[
\indi{\mathcal E_1}
\P(\mathcal E_2^c\mid\rvy_1)
\right] \le 2\exp\big(-\big(\log N\big)^2\big).
\end{align*}
Conditional on \(\mathcal E_2\), Lemma~\ref{lem:third-day-extinction} applies
at time \(t=2\), so for some \(\xi=\xi(a,b)>0\), the tower property gives
\begin{align*}
\P(\mathcal E_2\cap\{\gR_3\neq\emptyset\})
&=
\E\left[
\indi{\mathcal E_2}
\P(\gR_3\neq\emptyset\mid\rvy_2)
\right] \le 2N^{-\xi}.
\end{align*}
The event \(\{\gR_3\neq\emptyset\}\) is contained in the union of
\(\mathcal E_1^c\), \(\mathcal E_1\cap\mathcal E_2^c\), and
\(\mathcal E_2\cap\{\gR_3\neq\emptyset\}\). Consequently,
\[
\P(\gR_3\neq\emptyset)
\le
2\exp(-c_H N)
+2\exp\big(-\big(\log N\big)^2\big)
+2N^{-\xi}.
\]
This proves part~\textup{(i)} and completes the proof.
\end{proof}

\section{Subpolynomial-Time Blue Unanimity}
\label{sec:polylogarithmic-days-to-unanimity}

Throughout this section, work under the assumptions of
Theorem~\ref{thm:super-polylog-unanimity}. The proofs of the supporting lemmas
are deferred to Appendix~\ref{app:deferred-super-polylog}.

With the filtration defined in \eqref{eqn:natural-filtration}, define the
stopping time by
\[
\tau\coloneqq\inf\{t\ge0:u_t\le1\}, \qquad u_t\coloneqq-\widetilde{\Delta}_t\frac{\sqrt{\log N}}{N}.
\]
Thus \(\tau\) is the entrance time into the constant-time window.

The first step is to control how long the accumulated almost-linear gains take
to reach this window. Lemma~\ref{lem:constant-window-entrance} shows that the moderate-disadvantage phase ends by the time horizon \(T_N\)
defined in \eqref{eqn:super-polylog-time-horizon}, with an exponentially decaying failure probability.

\begin{lemma}[Entrance into the constant-time window]
\label{lem:constant-window-entrance}
There exist constants \(A,c,C>0\), depending only on \(a,b\), such that, for
all sufficiently large \(N\), the time horizon \(T_N\) defined in
\eqref{eqn:super-polylog-time-horizon} obeys
\[
\P(\tau>T_N)
\le
2\exp(Ah_N^2)
\exp\left(
-cN\frac{\exp(-C h_N^2)}{(1+h_N)^2}
\right).
\]
\end{lemma}

Once the process has entered the window, it remains to convert that entrance
into blue unanimity within a constant number of updates.

\begin{lemma}[Completion after entrance into the constant-time window]
\label{lem:constant-window-completion}
There exists a constant \(\xi=\xi(a,b)>0\) such that, on the event
\(\{\tau<\infty\}\), for all sufficiently large \(N\),
\[
\P\left(
\gR_{\tau+3}\neq\emptyset
\,\middle|\,
\mathcal F_{\tau}
\right)
\le
4\exp\left(-(\log N)^2\right)+2N^{-\xi}.
\]
\end{lemma}

The two Lemmas~\ref{lem:constant-window-entrance} and \ref{lem:constant-window-completion} supply the two stages of the argument: the first reaches the
constant-time window by time \(T_N\), and the second gives blue
unanimity within three further updates. We now combine them to prove Theorem~\ref{thm:super-polylog-unanimity}.

\begin{proof}[Proof of Theorem~\ref{thm:super-polylog-unanimity}]
    Since the all-blue configuration is absorbing,
    \[
    \{\gR_{T_N+3}\neq\emptyset\}
    \subseteq
    \{\tau>T_N\}
    \cup
    \{\tau\le T_N,\ \gR_{\tau+3}\neq\emptyset\}.
    \]
    Lemma~\ref{lem:constant-window-completion} gives
    \[
    \P\left(
    \tau\le T_N,\,
    \gR_{\tau+3}\neq\emptyset
    \right)
    \le
    4\exp\left(-(\log N)^2\right)+2N^{-\xi}.
    \]
    Combining this estimate with
    Lemma~\ref{lem:constant-window-entrance}, we get
    \[
    \P\left(\gR_{T_N+3}=\emptyset\right)
    \ge
    1
    -2\exp(Ah_N^2)
    \exp\left(
    -cN\frac{\exp(-C h_N^2)}{(1+h_N)^2}
    \right)
    -4\exp\left(-(\log N)^2\right)
    -2N^{-\xi}.
    \]
    Finally, because \(h_N=o(\sqrt{\log N})\),
    \[
    \exp(A h_N^2)
    \exp\left(
    -cN\frac{\exp(-C h_N^2)}{(1+h_N)^2}
    \right)
    =
    o(N^{-\xi})
    \]
    after decreasing \(\xi\), and \(\exp(-(\log N)^2)=o(N^{-\xi})\). Hence the
    success probability is \(1-O(N^{-\xi+o(1)})\).
    \end{proof}

\section{Polynomial-Time Blue Unanimity}
\label{sec:polynomial-days-to-unanimity}

We present the proof of Theorem~\ref{thm:polynomial-minimal-bias} in this
section. The proofs of the supporting lemmas are deferred to
Appendix~\ref{app:deferred-polynomial-days}.

Following the notation in Theorem~\ref{thm:super-polylog-unanimity}, we define the stopping time by
\begin{align}
\tau
&\coloneqq
\inf\left\{
t\ge0:
\widetilde{\Delta}_t
\ge
-\frac{N}{\sqrt{\log\log N}}
\right\}, \qquad h_N\coloneqq\sqrt{\frac{\log N}{\log\log N}}.
\label{eq:tau-def}
\end{align}
By \eqref{eqn:affine-delta-tilde-delta}, the stopping condition is equivalent
to $\Delta_t \ge \Delta_N^{\mathrm{sp}}$, where
\begin{align}
\Delta_N^{\mathrm{sp}}
\coloneqq
\frac{a-b}{a+b}N
-
\frac{2N}{(a+b)\sqrt{\log\log N}}.
\label{eqn:super-polylog-unweighted-threshold}
\end{align}
Thus \(\tau\) is the entrance time into the subpolynomial-time completion
window. Conditional on \(\mathcal F_\tau\), blue unanimity then follows from
Theorem~\ref{thm:super-polylog-unanimity} within
\(T_N+3=N^{o(1)}\) additional updates.

\subsection{Block Amplification}

We begin by quantifying the positive conditional drift and the expected
flip scale available before the process reaches the completion window.
\begin{lemma}[Rate-function bounds before the completion window]
    \label{lem:pre-super-polylog-one-step-bound}
    Fix \(\eta>0\). There exist constants \(c,C>0\), depending only on
    \(a,b,\eta\), such that, for every \(t<\tau\) satisfying \(\Delta_t>0\),
    \begin{align}
    \mu_t
    &\ge
    c\theta_t N^{1-I_t-\eta},
    &
    \nu_t
    &\le
    CN^{1-I_t+\eta},
    &
    \frac{\mu_t}{\nu_t}
    &\ge
    c\theta_t.
    \label{eqn:pre-super-polylog-rate-bounds}
    \end{align}
\end{lemma}

A natural first attempt is to apply a one-step concentration estimate at the
lower-drift scale supplied by
Lemma~\ref{lem:pre-super-polylog-one-step-bound}. By
\eqref{eqn:affine-delta-tilde-delta}, applying
Lemma~\ref{lem:R1-read-two-concentration} with deviation equal to one half of
that scale produces, up to a constant factor, the exponent
\[
\frac{\left(\theta_t N^{1-I_t-\eta}\right)^2}{N}
=
\theta_t^2N^{1-2I_t-2\eta}.
\]
This one-step argument becomes ineffective near the lower boundary of the
polynomial-time theorem. At the boundary scale
\(\Delta_t\asymp\sqrt{N/\log N}\), we have
\(\theta_t\asymp\sqrt{\log N/N}\), and the last expression becomes
\(N^{-2I_t-2\eta}\log N=o(1)\) whenever \(I_t\) is bounded away from zero.
Consequently, the resulting one-step failure bound does not vanish, so this
estimate cannot certify positive progress at every update or be iterated over
the polynomial-time horizon.

An alternative to the one-step argument is to use block amplification, which avoids requiring visible progress after each individual
update. Consider a time interval beginning at time \(s\) and ending at time \(s+L-1\) on which
\(I_{s+m}\le\overline I_s\), with \(\overline I_s\) specified by
\eqref{eqn:local-rate-cap-condition} below, and
\(\theta_{s+m}\ge c\theta_s\). The one-step drift can then be summed uniformly by choosing
\(L\asymp\Delta_sN^{\overline I_s+\eta-1}/\theta_s\), which gives
\[
\sum_{m=0}^{L-1}\mu_{s+m}
\gtrsim
L\theta_s N^{1-\overline I_s-\eta}
\asymp
\Delta_s.
\]
After taking the proportionality constant in \(L\) sufficiently large, the
accumulated conditional drift is therefore a sufficiently large multiple of
\(\Delta_s\), the scale required to double the advantage. The conditional MGF
and stopped supermartingale argument below control the cumulative centered
fluctuation on this block and show that doubling occurs before a substantial
downward excursion, except with exponentially small probability.

\subsection{Estimate of Block Length}

Fix a block starting at time \(s<\tau\). Our goal is to show that its upper
exit time \(\tau_s^+\), defined in \eqref{eqn:tau-s-plus}, corresponding to
either doubling the advantage or
entering the completion window, is at most the natural block length \(L\) with
high probability. We begin by controlling the cumulative random fluctuation.
Write the centered one-step increment as
\begin{align}
\xi_t\coloneqq \Delta_{t+1}-\Delta_t-\mu_t.
\label{eqn:xi-definition}
\end{align}

\begin{lemma}[Conditional Bernstein MGF bound]
    \label{lem:conditional-bernstein-mgf}
    There exist constants \(K>0\) and \(\lambda_0>0\), depending only on
    \(a,b\), such that for every \(0\le \lambda\le \lambda_0\),
    \begin{align}
    \E[\exp(-\lambda\xi_{t})\mid\mathcal F_{t}]
    \le
    \exp(K\lambda^2\nu_{t}).
    \end{align}
\end{lemma}

The preceding MGF estimate can be iterated after compensating for the
conditional flip scale \(\nu_t\). Let \(K>0\) and \(\lambda_0>0\) be the
constants supplied by Lemma~\ref{lem:conditional-bernstein-mgf}. For any time
\(s\ge0\), any \(0\le\lambda\le\lambda_0\), and any \(\ell\ge0\), define
\begin{align}
    M_{s + \ell}
    &\coloneqq
    \sum_{m=0}^{\ell-1}\xi_{s+m},
    &
    W_{s + \ell}
    &\coloneqq
    \exp\bigg(-\lambda M_{s + \ell}
    -K\lambda^2\sum_{m=0}^{\ell-1}\nu_{s+m}\bigg). \label{eqn:Wsl-Msl}
\end{align}

The compensating term in \(W_{s+\ell}\) is chosen so that the conditional MGF
bound applies at every update. This gives the supermartingale used to control
the fluctuation accumulated over the block.
\begin{lemma}\label{lem:supermartingale}
Fix a time \(s\ge 0\). For every \(0\le \lambda\le \lambda_0\), the
process \((W_{s+\ell})_{\ell\ge0}\) defined in \eqref{eqn:Wsl-Msl} is a
nonnegative supermartingale with respect to the filtration
\((\mathcal F_{s+\ell})_{\ell\ge 0}\).
\end{lemma}

We stop the block as soon as the advantage either reaches its upper target or
makes a fixed proportional downward excursion.
\begin{definition}\label{def:upperlowerexit}
Fix \(\chi\in(0,1)\). For a time \(s\ge0\), define
\begin{subequations}
\begin{align}
    \tau_{s}^{+}
    &\coloneqq
    \inf\left\{
        \ell\ge 0:
        \Delta_{s+\ell}\ge 2\Delta_{s}
        \textnormal{ or } s+\ell\ge\tau
    \right\}, \label{eqn:tau-s-plus} \\
    \tau_{s}^{-}
    &\coloneqq
    \inf\left\{
        \ell\ge 0:
        \Delta_{s+\ell}\le \chi\Delta_{s}
    \right\}. \label{eqn:tau-s-minus}
\end{align}
\end{subequations}
\end{definition}

If \(\tau_s^+>L\), then either neither boundary has been reached by time \(L\),
or the lower boundary was reached before the upper boundary. Thus
\[
\{\tau_s^+>L\}
\subseteq
\{\tau_s^+>L,\ \tau_s^->L\}
\cup
\{\tau_s^-\le L,\ \tau_s^-<\tau_s^+\}.
\]
The next two lemmas show that both alternatives have exponentially small
conditional probability. The first rules out remaining inside the block
corridor beyond \(L\).
\begin{lemma}\label{lem:no-doubling-no-lower-exit}
Fix \(\eta>0\), \(\chi\in(0,1)\), and a time \(0\le s<\tau\) with
\(\Delta_s>0\). Let \(\overline{I}_{s}\) be an
\(\mathcal F_s\)-measurable number such that
\begin{align}
I_{s+\ell}
&\le \overline{I}_{s}
\quad\text{whenever}\quad
\chi\Delta_s\le \Delta_{s+\ell}\le2\Delta_s
\ \text{and}\ s+\ell<\tau.
\label{eqn:local-rate-cap-condition}
\end{align}
Define \(L\) by
\begin{align}
L
\coloneqq
\left\lceil
A_0\Delta_s N^{\overline{I}_{s}+\eta-1}/\theta_s
\right\rceil,
\label{eqn:L-definition}
\end{align}
where \(A_0>0\) is some sufficiently large constant. Then
\[
\P(\tau_s^+>L,\,\tau_s^->L\mid\mathcal F_{s})\le \exp(-c\Delta_s\theta_{s}).
\]
\end{lemma}

The second rules out reaching the lower boundary before the upper target.
\begin{lemma}[Lower exit estimate]
\label{lem:lower-exit-estimate}
Fix \(\chi\in(0,1)\) and a time \(0\le s<\tau\) with \(\Delta_s>0\).
Let \(L\) be a nonnegative \(\mathcal F_s\)-measurable integer.
Then
\[
\P(\tau_s^-\le L,\,\tau_s^-<\tau_s^+\mid\mathcal F_{s})
\le \exp(-c\Delta_s\theta_{s}).
\]

\end{lemma}

Combining these two estimates with the event decomposition above yields the
desired high-probability upper bound on \(\tau_s^+\).
\begin{lemma}\label{lem:block-doubling}
Fix \(\eta>0\), \(\chi\in(0,1)\), and a time \(0\le s<\tau\) with
\(\Delta_s>0\). Let \(\overline{I}_{s}\) be an
\(\mathcal F_s\)-measurable number such that
\(I_{s+\ell}\le\overline{I}_{s}\) whenever
\(\chi\Delta_s\le\Delta_{s+\ell}\le2\Delta_s\) and \(s+\ell<\tau\). There
exists \(A_*=A_*(a,b,\eta,\chi)>0\) such that, for every \(A_0\ge A_*\),
\[
\P(\tau_s^+>L\mid\mathcal F_{s})\le \exp(-c\Delta_s\theta_{s}),
\]
where \(L\) is defined by \eqref{eqn:L-definition} using this choice of
\(A_0\).
\end{lemma}

Hence the advantage doubles, or the process enters the completion window,
within \(L\) updates except with conditional probability
\(\exp(-c\Delta_s\theta_s)\).

\subsection{Concluding the Proof of the Polynomial-Time Theorem}
\label{subsec:concluding-polynomial-proof}

\begin{proof}[Proof of Theorem~\ref{thm:polynomial-minimal-bias}]
Fix \(\varepsilon>0\), and choose \(\eta>0\) so that
\(10\eta<\varepsilon\). The hypotheses give
\(\Delta_0>0\) and \(\Delta_0^2\log N/N\to\infty\). By the camp-size identities and
\eqref{eqn:It-exponent}, \(I_t\) is a continuous nonincreasing function of
\(\Delta_t/N\) before \(\tau\). It is uniformly Lipschitz on the relevant
compact interval. We may therefore fix \(\chi\in(0,1)\), sufficiently close
to \(1\), such that whenever \(0\le d\le N\),
\[
\sup_{\chi d\le d'\le d}
I\left(\frac{a(N-d')}{2N},\frac{b(N+d')}{2N}\right)
\le
I\left(\frac{a(N-d)}{2N},\frac{b(N+d)}{2N}\right)
+\eta.
\]

We first bound the time required to reach \(\tau\). If \(\tau=0\), this stage
is empty. Otherwise, define \(\Delta_j\coloneqq2^j\Delta_0\), and let \(J\)
be the smallest integer such that
\(\Delta_J\ge\Delta_N^{\mathrm{sp}}\). Since
\(\Delta_N^{\mathrm{sp}}\le N\), we have \(J\le C\log N\). For
\(0\le j\le J\), let
\[
t_j
\coloneqq
\inf\left\{
t\ge0:
\Delta_t\ge\Delta_j
\text{ or }
t\ge\tau
\right\}.
\]
Then \(t_0=0\) and \(t_J=\tau\).

By conditional resampling, Lemma~\ref{lem:block-doubling} applies at each
stopping time \(t_j\), conditional on \(\mathcal F_{t_j}\).
Fix \(0\le j<J\), suppose that \(t_j<\tau\), and put \(s=t_j\). If
\(\Delta_s\ge\Delta_{j+1}\), then \(t_{j+1}=t_j\). Otherwise,
\(\Delta_j\le\Delta_s<\Delta_{j+1}\). Choose \(\overline I_s\) as the
largest rate over the interval in Lemma~\ref{lem:block-doubling}. The choice
of \(\chi\), the monotonicity of the rate, and
\(\Delta_s\ge\Delta_j\ge\Delta_0\) give
\[
\overline I_s
\le
I_s+\eta
\le
I_0+\eta.
\]
Consequently, after substituting \(s=t_j\) into \eqref{eqn:L-definition}, we
find that the block length used in Lemma~\ref{lem:block-doubling} satisfies
\[
L_j
\coloneqq
\left\lceil
A_0\Delta_s
\frac{N^{\overline I_s+\eta-1}}{\theta_s}
\right\rceil
\le
C\Delta_s
\frac{N^{I_0+2\eta-1}}{\theta_s}.
\]
Meanwhile, Lemma~\ref{lem:block-doubling} yields
\[
\P\left(
t_{j+1}>t_j+L_j
\,\middle|\,
\mathcal F_{t_j}
\right)
\le
\exp(-c\Delta_s\theta_s).
\]
Since \(\Delta_s\ge\Delta_j\), the definition of \(\theta_s\) gives
\[
\Delta_s\theta_s
\ge
c\min\left\{
\Delta_j,
\frac{\Delta_j^2\log N}{N}
\right\}.
\]

For levels at which \(t_j=\tau\) or \(t_{j+1}=t_j\), set \(L_j=0\), since
no block is needed.
We next sum the block lengths. There are two cases.
\begin{enumerate}[label=\textup{Case \arabic*:}, leftmargin=*]
\item If \(\Delta_s<N/\log N\), then
\(\theta_s=\Delta_s\log N/N\), and therefore
\[
L_j
\le
C\Delta_s
\frac{N^{I_0+2\eta-1}}{\theta_s}
=
C\frac{N^{I_0+2\eta}}{\log N}.
\]
There are at most \(C\log N\) levels of this type.

\item If \(\Delta_s\ge N/\log N\), then \(\theta_s=1\). Since the present
branch also satisfies \(\Delta_s<\Delta_{j+1}\), we have
\[
L_j
\le
C\Delta_sN^{I_0+2\eta-1}
\le
C\Delta_{j+1}N^{I_0+2\eta-1}.
\]
The geometric sum of the corresponding \(\Delta_{j+1}\)'s is at most \(CN\).
\end{enumerate}
Combining the two cases above, for all sufficiently large \(N\), we have
\[
\sum_{j=0}^{J-1}L_j
\le
CN^{I_0+2\eta}
\le
N^{I_0+\varepsilon/2}.
\]

We now convert the block estimates into a bound for the entrance time. For
each \(0\le j<J\), let \(\mathcal E_j\) be the event that \(t_j<\tau\) and
the process fails to reach the next level within its allotted block, namely,
\[
\mathcal E_j
\coloneqq
\left\{t_j<\tau,\ t_{j+1}>t_j+L_j\right\}.
\]
If none of the events \(\mathcal E_j\) occurs, then, for every \(j\), either
\(t_j=\tau\), in which case \(t_{j+1}=t_j\), or
\(t_{j+1}-t_j\le L_j\). Consequently,
\[
\tau=t_J
=
\sum_{j=0}^{J-1}(t_{j+1}-t_j)
\le
\sum_{j=0}^{J-1}L_j
\le
N^{I_0+\varepsilon/2}.
\]
It follows that
\[
\left\{\tau>N^{I_0+\varepsilon/2}\right\}
\subseteq
\bigcup_{j=0}^{J-1}\mathcal E_j.
\]
Moreover, the conditional block estimate and the tower property give
\[
\begin{aligned}
\P(\mathcal E_j)
&=
\E\left[
\indi{t_j<\tau}
\P\left(
t_{j+1}>t_j+L_j
\,\middle|\,
\mathcal F_{t_j}
\right)
\right] \le
\exp\left(
-c\min\left\{
\Delta_j,
\frac{\Delta_j^2\log N}{N}
\right\}
\right).
\end{aligned}
\]
Therefore, a union bound over the dyadic levels yields
\[
\P\left(\tau>N^{I_0+\varepsilon/2}\right)
\le
\sum_{j=0}^{J-1}
\exp\left(
-c\min\left\{
\Delta_j,
\frac{\Delta_j^2\log N}{N}
\right\}
\right).
\]
To estimate this sum, let \(j_*\) be the first index such that
\(\Delta_{j_*}\ge N/\log N\). This index exists for all sufficiently large
\(N\), because
\(\Delta_{J-1}=\Delta_J/2\ge\Delta_N^{\mathrm{sp}}/2\asymp N\).
For \(j<j_*\), the quadratic term is smaller and
\[
\min\left\{
\Delta_j,
\frac{\Delta_j^2\log N}{N}
\right\}
=
4^j\frac{\Delta_0^2\log N}{N}.
\]
Since \(\Delta_0^2\log N/N\to\infty\), summing over these levels gives,
after changing \(c>0\),
\[
\sum_{j<j_*}
\exp\left(-c4^j\frac{\Delta_0^2\log N}{N}\right)
\le
\exp\left(-c\frac{\Delta_0^2\log N}{N}\right).
\]
For \(j\ge j_*\), the linear term is smaller and
\[
\sum_{j=j_*}^{J-1}e^{-c\Delta_j}
\le
Ce^{-c\Delta_{j_*}}
\le
Ce^{-cN/\log N}
\le
N^{-\xi}.
\]
If \(\Delta_0>N/\log N\), then \(j_*=0\), the first sum is empty, and the
same linear estimate applies from the initial level. Thus, in both cases,
\[
\P\left(\tau>N^{I_0+\varepsilon/2}\right)
\le
\exp\left(-c\frac{\Delta_0^2\log N}{N}\right)
+N^{-\xi}.
\]

It remains to complete the process after \(\tau\). For the choice in
\eqref{eq:tau-def},
\[
h_N^2=\frac{\log N}{\log\log N},
\qquad
T_N
=
\left\lceil
\exp(Ah_N^2)
\right\rceil
=
N^{o(1)}.
\]
The failure bound in Theorem~\ref{thm:super-polylog-unanimity} is
\(N^{-\xi+o(1)}\), because its first term is
\(\exp\{-N^{1-o(1)}\}\). By time homogeneity and the independence of future
update graphs, the theorem applies conditionally on \(\mathcal F_\tau\).
Moreover,
\[
N^{I_0+\varepsilon/2}+T_N+3
\le
N^{I_0+\varepsilon}
\]
for all sufficiently large \(N\). Since the all-blue configuration is
absorbing, after decreasing \(\xi>0\) if necessary,
\[
\begin{aligned}
\P\left(
\gR_{\lceil N^{I_0+\varepsilon}\rceil}\neq\emptyset
\right)
&\le
\P\left(
\tau>N^{I_0+\varepsilon/2}
\right)
+3N^{-\xi}\\
&\le
\exp\left(
-c\frac{\Delta_0^2\log N}{N}
\right)
+4N^{-\xi}.
\end{aligned}
\]
Taking complements proves the theorem.
\end{proof}

\begin{remark}[Refining the \(\varepsilon\)-slack]
Suppose we formally set
\(
\widetilde{\Delta}_0=-h_NN/\sqrt{\log N}
\)
in the scaling of Theorem~\ref{thm:polynomial-minimal-bias}. Then
\[
I_0
=
\left(
\sqrt{\frac{a|\gR_0|}{N}}
-
\sqrt{\frac{b|\gB_0|}{N}}
\right)^2
=
\frac{\widetilde{\Delta}_0^2/N^2}
{\left(
\sqrt{a|\gR_0|/N}
+
\sqrt{b|\gB_0|/N}
\right)^2}
=
O\!\left(\frac{h_N^2}{\log N}\right).
\]
Consequently,
\[
N^{I_0+\varepsilon}
=
N^\varepsilon N^{I_0}
=
N^\varepsilon\exp\bigl(O(h_N^2)\bigr).
\]
Thus the factor \(\exp(O(h_N^2))\) matches the time scale in
Theorem~\ref{thm:super-polylog-unanimity}, while \(N^\varepsilon\) is the
additional slack in Theorem~\ref{thm:polynomial-minimal-bias}. The proof of Theorem~\ref{thm:polynomial-minimal-bias} makes this transition explicit;
removing the remaining \(\varepsilon\)-slack from the block-amplification
argument is an open problem.
\end{remark}

\section{Polynomial-Time Lower Bound for Blue Unanimity}
\label{sec:necessary-conditions}

We present the proof of Theorem~\ref{thm:polynomial-lower-bound} in this
section. The proofs of the supporting lemmas are deferred to
Appendix~\ref{app:deferred-polynomial-lower-bound}.

The argument stops the process when the unweighted advantage has increased by
a fixed linear amount. Before that time, the large-deviation exponent remains
close to its initial value \(I_0\), so the expected number of red-to-blue flips
up to time \(N^{I_0-\varepsilon}\) is sublinear in \(N\). Reaching blue
unanimity, however, requires a linear increase in the unweighted advantage.

\subsection{Rate Control Below an Intermediate Level}

Let \(x_0\coloneqq\Delta_{0}/N\), and define
\(x_*\coloneqq(a-b)/(a+b)\). The condition
\(1<\rho_0\le a/b-\kappa\) implies
\(0<x_0\le x_\kappa\), where \(x_\kappa\) is defined in
\eqref{eqn:x-kappa-definition} below:
\begin{align}
x_\kappa\coloneqq(a/b-\kappa-1)/(a/b-\kappa+1).
\label{eqn:x-kappa-definition}
\end{align}
Since \(\kappa>0\), \(x_\kappa<x_*\). For \(x<x_*\), define
\[
\mathcal I(x)
\coloneqq
\left(\sqrt{a(1-x)/2}-\sqrt{b(1+x)/2}\right)^2.
\]
Thus \(I_{0}=\mathcal I(x_0)\).

\begin{lemma}\label{lem:local-exponent-control}
Fix \(\varepsilon>0\), and set \(\eta\coloneqq\varepsilon/2\). There exists a
constant \(0<\sigma=\sigma(a,b,\kappa,\varepsilon)<1-x_0\) such that whenever
\(\Delta_{t}\le \Delta_{0}+\sigma N\), we have \(I_{t}\ge I_{0}-\eta\).
\end{lemma}

Lemma~\ref{lem:local-exponent-control} allows us to replace the evolving exponent by the fixed exponent \(I_0\) until the process crosses the intermediate level used below.

\subsection{Counting Red-to-Blue Flips}

Fix \(\varepsilon>0\), let \(\sigma\) be supplied by
Lemma~\ref{lem:local-exponent-control}, and define
\[
\zeta
\coloneqq
\inf\{t\ge0:\Delta_t\ge\Delta_0+\sigma N\},
\qquad
T
\coloneqq
\left\lfloor N^{I_0-\varepsilon}\right\rfloor.
\]
The first estimate bounds the total number of red-to-blue flips accumulated
before the stopping time \(\zeta\).

\begin{lemma}\label{lem:red-to-blue-flip-count-before-zeta}
There exists \(C=C(a,b,\kappa,\varepsilon)>0\) such that
\[
\E\left[
\sum_{t=0}^{T-1}\gRB_t\indi{t<\zeta}
\right]
\le
CN^{1-\varepsilon/2}.
\]
\end{lemma}

Since crossing the intermediate level requires a linear net increase in
\(\Delta_t\), the preceding expectation estimate gives a polynomially small
upper bound on the probability of an early crossing.
\begin{lemma}\label{lem:zeta-is-late}
There exists \(C=C(a,b,\kappa,\varepsilon)>0\) such that
\(\P(\zeta\le T)\le C N^{-\varepsilon/2}\).
\end{lemma}

\subsection{Concluding the Proof of the Polynomial-Time Lower Bound}

The rate-control and flip-count estimates now reduce the theorem to a direct
comparison between the intermediate level and blue unanimity.
\begin{proof}[Proof of Theorem~\ref{thm:polynomial-lower-bound}]
Fix \(\varepsilon>0\), and use the corresponding \(\sigma\), \(\zeta\), and
\(T\) defined above. Lemma~\ref{lem:local-exponent-control} gives
\(\Delta_0+\sigma N<N\). On the event \(\{\gR_T=\emptyset\}\), blue unanimity
implies \(\Delta_T=N\), so the definition of \(\zeta\) gives
\[
\{\gR_T=\emptyset\}
\subseteq
\{\zeta\le T\}.
\]
Consequently, Lemma~\ref{lem:zeta-is-late} yields
\[
\P(\gR_T=\emptyset)
\le
\P(\zeta\le T)
\le
CN^{-\varepsilon/2}.
\]
Taking complements and recalling the definition of \(T\), we obtain
\[
\P\left(
\gR_{\lfloor N^{I_0-\varepsilon}\rfloor}\neq\emptyset
\right)
\ge
1-CN^{-\varepsilon/2},
\]
which proves the theorem.
\end{proof}

\section{Future Directions}
We conclude this paper with some possible future directions.

\paragraph{Effect of Spatial Locality.}
In the SBM, vertices in the same camp have identical
neighborhood laws. However, many interactions in real-world networks are
instead spatially local. Thus a vertex samples opinions primarily from nearby vertices and its update depends
on the local, rather than only the global, opinion imbalance. \emph{Random geometric graphs} (RGGs) provide a natural model for this constraint. Recent work
\cite{gaudio2024exact, gaudio2025exact} studies community detection in the
\emph{Geometric Stochastic Block Model} (GSBM) and derives the \emph{information-theoretic} threshold for exact recovery. In view of Remark~\ref{rem:community-detection-exponent},
this raises a parallel question for majority dynamics: which features of the
local geometry, such as the connection radius, local density fluctuations, and
spatial bottlenecks, determine whether an initial advantage can be amplified to reach unanimity? A central question is whether spatially coherent minority regions can survive even when the global weighted advantage favors the majority, thereby changing both the threshold for unanimity and the time required to reach it.

\paragraph{Effect of Degree Heterogeneity.}
The Erd\H{o}s--R\'enyi model and the standard SBM treat
vertices within the same community as statistically interchangeable, and
therefore do not capture the substantial degree heterogeneity commonly
observed in real networks. The degree-corrected stochastic block model
\cite{Dasgupta2004Spectral} addresses this limitation by assigning each vertex
an individual connectivity parameter while retaining the underlying community
structure. This extension is particularly natural for majority dynamics:
high-degree vertices base their updates on more observations and can influence
more neighbors, so degree heterogeneity may change both the direction and the
speed of consensus. Community detection under degree correction has been
studied in \cite{Gao2018CommunityDI}, where the minimax rate for community detection has been derived. A corresponding question for majority
dynamics is whether unanimity still occurs, and how its threshold and
convergence time depend on the distribution of vertex degrees.



\paragraph{The Disassortative Regime.}
The assortative condition \(a>b\) creates positive feedback: vertices interact
more often with those sharing their current opinion, and a sufficiently strong
advantage can therefore reinforce itself. It is natural to ask what replaces
this mechanism in the disassortative regime \(b>a\), which models
heterophilous networks whose interactions occur more frequently across the two
opinion camps. Such cross-camp exposure creates negative feedback, because a
vertex is more likely to observe the opinion opposite to its own. In the
idealized extreme \(a=0\), every non-isolated vertex has only opposite-colored
neighbors and hence flips at the next update, suggesting an alternating
two-cycle rather than convergence. Understanding when this tendency produces
oscillation, persistent disagreement, or eventual unanimity is a natural
counterpart to the assortative theory developed here.

\paragraph{Relaxing the Connectivity Condition.}
Our main results assume \(b>1\), a convenient condition that guarantees
connectivity with high probability uniformly over the possible imbalance
between the two opinion camps. This condition is stronger than what is needed
for a fixed sparse stochastic block model. As preparation for a future
extension, Lemma~\ref{lem:connectivity-sparse-sbm} derives a sharper
connectivity criterion that depends explicitly on the two camp sizes and the
parameters \(a,b\). The remaining challenge is to incorporate this
state-dependent criterion into the dynamics, because the opinion partition,
and hence the connectivity threshold of the newly sampled graph, changes from
one update to the next. This raises the question of whether unanimity can still
be reached when the process temporarily passes through a disconnected regime,
or whether persistent disconnection allows different components to retain
opposing opinions.

\paragraph{Weighted Self-Opinions.}
The present update rule gives a vertex's current opinion no additional weight
beyond the votes supplied by its neighbors. In many opinion-formation models,
however, agents exhibit inertia or stubbornness and require a sufficiently
strong opposing majority before changing their state. This effect can be
modeled by the modified rule
\[
\ervy_{t+1}(v)
=
\operatorname{sign}\!\left(
N_t^{\gB}(v)-N_t^{\gR}(v)+s\,\ervy_t(v)
\right),
\]
where \(s\in\mathbb R\) is the weight assigned to the current opinion. A
positive value of \(s\) favors persistence, whereas a negative value favors
switching; equivalently, the modification introduces a weighted self-loop at
each vertex. It is natural to ask how this local inertia changes the threshold
for unanimity, the convergence time, and the possibility that a minority
opinion persists indefinitely.

\newpage
\appendix
\phantomsection
\addcontentsline{toc}{section}{Appendices}

\section{Flip Probability Estimates for Sparse Binomial Differences}\label{sec:flip-probability-analysis}
Let \(\rD=\rX-\rY\) be the difference of two independent Binomial random
variables, with
    \begin{align}
    \rX\sim\Bin(s,\beta),\qquad \rY\sim\Bin(r,\alpha), \label{eqn:X-Y-distribution}
\end{align}
where \(r=r(N)\le N\) and \(s=s(N)\le N\). Define the ratios
\begin{align}
    A_{N} \coloneqq \frac{ar}{N}, \qquad B_{N} \coloneqq \frac{bs}{N}. \label{eqn:A-B-ratios}
\end{align}
The probabilities \(\alpha\) and \(\beta\) lie in the critical regime: both
are of order \(\log N/N\). Thus there are constants \(a,b>0\) such that
\begin{align}
    \alpha=a \cdot \frac{\log N}{N}, \qquad \beta=b \cdot \frac{\log N}{N}. \label{eqn:alpha-beta-critical}
\end{align}
Denote the mean and variance of the random variable \(\rD\) by
\begin{align}
    \mu_{N} \coloneqq \E\rD, \qquad \sigma^2_{N} \coloneqq \Var(\rD). \label{eqn:mu-sigma-squared}
\end{align}
Below, we approximate \(\P(\rD>0)\) in three different regimes according to
the magnitude of \(\mu_N/\sigma_N\).

\subsection{Gaussian Window: Berry--Esseen Approximation}
\noindent
The estimates in this subsection are useful in the bounded Gaussian window, where
the standardized mean \(\mu_N/\sigma_N\) is bounded. In this regime, the
Berry--Esseen error is small enough to approximate the tail probability directly
using the standard normal distribution function; in particular, when
\(|\mu_N|/\sigma_N=O(1)\), it gives a bounded-window approximation for
\(\P(\rD>0)\).

\begin{lemma}\label{lem:gaussian-sparse-binomial-difference}
    Consider the problem described in
    \eqref{eqn:X-Y-distribution}--\eqref{eqn:mu-sigma-squared}. Suppose that
    \(A_N+B_N\ge c_0\) for some constant \(c_0>0\). Then, uniformly over such \(r,s\), we have
    \[
    \sup_{x\in\mathbb R}\left|
    \P\left(\frac{\rD-\mu_N}{\sigma_N}\le x\right)-\Phi(x)\right|
    \le C/\sqrt{\log N},
    \]
    where \(C=C(a,b,c_0)>0\) is some constant.
    \end{lemma}

    \begin{proof}[Proof of Lemma~\ref{lem:gaussian-sparse-binomial-difference}]
    Since \(A_N+B_N\ge c_0\) and \(\alpha,\beta=O(\log N/N)\), for all sufficiently large \(N\),
    \[
    \sigma_N^2
    =
    \beta(1-\beta)s+\alpha(1-\alpha)r
    \ge
    \frac{c_0}{2}\log N.
    \]
    We realize the two binomial variables on a common probability space. Let
    \(\xi_1,\ldots,\xi_s\) be i.i.d.\ \(\Ber(\beta)\), let
    \(\zeta_1,\ldots,\zeta_r\) be i.i.d.\ \(\Ber(\alpha)\), and assume that
    the two families are independent. Then
    \[
    \rD-\mu_N
    =
    \sum_{i=1}^{s}(\xi_i-\beta)
    -
    \sum_{j=1}^{r}(\zeta_j-\alpha).
    \]
    The summands on the right-hand side are independent and have mean zero.
    Their total variance is exactly
    \[
    \sum_{i=1}^{s}\beta(1-\beta)
    +
    \sum_{j=1}^{r}\alpha(1-\alpha)
    =
    \sigma_N^2 \leq
    (a+b)\log N.
    \]
    We next bound the total absolute third moment. For
    \(\xi\sim\Ber(\beta)\), we have
    \[
    \E\abs{\xi-\beta}^{3}
    =
    \beta(1-\beta)^3
    +
    (1-\beta)\beta^3
    =
    \beta(1-\beta)\bigl((1-\beta)^2+\beta^2\bigr)
    \le \beta,
    \]
    while the same estimate applies to \(-(\zeta-\alpha)\). Hence
    \[
    \sum_{i=1}^{s}\E\abs{\xi_i-\beta}^{3}
    +
    \sum_{j=1}^{r}\E\abs{\zeta_j-\alpha}^{3}
    \le
    \beta s+\alpha r \leq (a+b)\log N,
    \]
    where the last inequality follows from \(r,s\le N\) and
    \eqref{eqn:alpha-beta-critical}. Applying
    Lemma~\ref{lem:berry-esseen} to the centered summands above gives
    \[
    \sup_{x\in\mathbb R}
    \left|
    \P\left(
    \frac{\rD-\mu_N}{\sigma_N}
    \le x
    \right)
    -
    \Phi(x)
    \right|
    \le
    \gC_{\mathrm{BE}}
    \frac{\beta s+\alpha r}{\sigma_N^3} \leq \frac{C}{\sqrt{\log N}},
    \]
    where the last inequality uses \(\beta s+\alpha r\le(a+b)\log N\) and the
    lower bound on \(\sigma_N^2\).
    \end{proof}

    \begin{corollary}[Gaussian estimates for one-vertex flip probabilities]
    \label{cor:gaussian-flip-probabilities}
Recall \(p_{t}^{\gB}\), \(q_{t}^{\gB}\) in \eqref{eqn:p-q-B}, \(p_{t}^{\gR}\), \(q_{t}^{\gR}\) in \eqref{eqn:p-q-R}, \(m_{t}^{\gB}\), \(v_{t}^{\gB}\) in \eqref{eqn:m-v-B} and \(m_{t}^{\gR}\), \(v_{t}^{\gR}\) in \eqref{eqn:m-v-R}. Define the following quantities:
\begin{align}
x_{t}^{\gR}\coloneqq m_{t}^{\gR}/\sqrt{v_{t}^{\gR}}, \qquad x_{t}^{\gB}\coloneqq m_{t}^{\gB}/\sqrt{v_{t}^{\gB}}.
\end{align}
Condition on \(\rvy_{t}\) satisfying \(|\gR_{t}|\ge c_0N\) and \(|\gB_{t}|\ge c_0N\) for some constant \(c_0>0\). Then, uniformly over such configurations,
    \begin{subequations}
    \begin{align}
    p_{t}^{\gR}
    &=
    \Phi(x_{t}^{\gR})+O(1/\sqrt{\log N}), \label{eq:gaussian-pR}\\
    q_{t}^{\gR}
    &=
    \Phi(-x_{t}^{\gR})+O(1/\sqrt{\log N}), \label{eq:gaussian-qR}\\
    p_{t}^{\gB}
    &=
    \Phi(x_{t}^{\gB})+O(1/\sqrt{\log N}), \label{eq:gaussian-pB}\\
    q_{t}^{\gB}
    &=
    \Phi(-x_{t}^{\gB})+O(1/\sqrt{\log N}). \label{eq:gaussian-qB}
    \end{align}
    \end{subequations}
\end{corollary}

    \begin{proof}[Proof of Corollary~\ref{cor:gaussian-flip-probabilities}]
    Throughout the proof, we condition on \(\rvy_t\). All probabilities and
    expectations below are conditional on this configuration.

    We first consider a fixed red vertex. Since \(\alpha,\beta=O(\log N/N)\), for all sufficiently large \(N\) we have \(1-\alpha\ge1/2\), \(1-\beta\ge1/2\), and \(|\gR_t|-1\ge c_0N/2\). Hence the conditional variance \(v_t^{\gR}\) is at least \(\frac{ac_0}{4}\log N\). Applying Lemma~\ref{lem:gaussian-sparse-binomial-difference} with
    \(s=|\gB_t|\), \(r=|\gR_t|-1\), \(\mu_N=m_t^{\gR}\), and \(\sigma_N^2=v_t^{\gR}\) gives
    \[
    \sup_{x\in\R}
    \left|
    \P\left(
    \frac{\rD_t^{\gR}-m_t^{\gR}}{\sqrt{v_t^{\gR}}}
    \le x
    \,\middle|\,
    \rvy_t
    \right)
    -
    \Phi(x)
    \right|
    \le
    \frac{C}{\sqrt{\log N}},
    \]
    where \(C=C(a,b,c_0)\). Taking \(x=-x_t^{\gR}\) gives
    \[
    q_t^{\gR}
    =
    \Phi(-x_t^{\gR})+O\left(\frac1{\sqrt{\log N}}\right),
    \qquad
    p_t^{\gR}
    =
    \Phi(x_t^{\gR})+O\left(\frac1{\sqrt{\log N}}\right).
    \]
    The same argument applied to \(\rD_t^{\gB}\), with
    \(s=|\gR_t|\) and \(r=|\gB_t|-1\), proves \eqref{eq:gaussian-pB} and \eqref{eq:gaussian-qB}.
\end{proof}

\subsection{Large-Deviation Tail: Rate-Function Exponents}
\noindent
The estimates in this subsection are useful on the logarithmic tail scale,
typically when \(|\mu_N|/\sigma_N\) is of order \(\sqrt{\log N}\). In this
range the polynomial prefactor is not tracked; the relevant information is the
exponent given by the rate function \(I(\cdot,\cdot)\) defined in
\eqref{eqn:LDP-rate-function}.

\begin{lemma}[Large-deviation tail estimates for a difference of sparse binomials] \label{lem:binomial-difference-large-deviation-tail}
Consider the problem described in
\eqref{eqn:X-Y-distribution}--\eqref{eqn:mu-sigma-squared}. Suppose that
\(A_N\ge c_0\) and \(B_N\ge c_0\) for some constant \(c_0>0\). Then,
uniformly over all such \(r,s\), we have
\begin{subequations}
\begin{align}
\P(\rD>0)
&\,= N^{-I(A_N, B_N) + o(1)} \label{eqn:binomial-difference-upper-tail}\\
\P(\rD\le 0)
&\,= N^{-I(B_N, A_N) + o(1)}.
\label{eqn:binomial-difference-lower-tail}
\end{align}
\end{subequations}
    \end{lemma}

    \begin{proof}[Proof of Lemma~\ref{lem:binomial-difference-large-deviation-tail}]
    We first prove the upper bounds in \eqref{eqn:binomial-difference-upper-tail} and \eqref{eqn:binomial-difference-lower-tail}. For every \(\theta\ge0\), using independence, Markov's inequality, and
    \(\log(1+x)\le x\), we get
    \begin{align*}
    \P(\rX-\rY>0)
    &\le
    \E e^{\theta(\rX-\rY)}\\
    &=
    \left(1+\beta(e^\theta-1)\right)^s
    \left(1+\alpha(e^{-\theta}-1)\right)^r\\
    &\le
    \exp\left\{
    \log N\left[B_N(e^\theta-1)+A_N(e^{-\theta}-1)\right]
    \right\}.
    \end{align*}
    The minimum of \(B_N(e^\theta-1)+A_N(e^{-\theta}-1)\) over
    \(\theta\ge0\) is \(-(\sqrt{A_N}-\sqrt{B_N})^2\) if \(B_N<A_N\), and is \(0\) if \(B_N\ge A_N\). Consequently, we have
    \begin{align}
    \P(\rX-\rY>0)
    \le N^{-I(A_N,B_N)},
    \label{eq:binomial-difference-chernoff-upper}
    \end{align}
    which proves the upper bound in \eqref{eqn:binomial-difference-upper-tail}. Applying the same argument to \(\rY-\rX\) yields the upper bound in \eqref{eqn:binomial-difference-lower-tail}, i.e.,
    \begin{align}
    \P(\rX-\rY\le0)
    =
    \P(\rY-\rX\ge0)
    \le
    N^{-I(B_N,A_N)}.
    \label{eq:binomial-difference-chernoff-lower}
    \end{align}

    It remains to prove the matching lower bounds, uniformly on
    \(A_N,B_N\ge c_0\). Since \(A_N\in[c_0,a]\) and \(B_N\in[c_0,b]\), it suffices to argue along an arbitrary subsequence on which \(A_N\to A\) and \(B_N\to B\), with \(A,B>0\). We first prove the lower bound in \eqref{eqn:binomial-difference-upper-tail}. Denote \(\rS_N\coloneqq \rY-\rX\). For every fixed \(\theta\in\R\), we have
    \begin{align*}
    \frac{1}{\log N}\log \E e^{\theta \rS_N}
    &=
    \frac{r}{\log N}\log\left(1+\alpha(e^\theta-1)\right)
    +
    \frac{s}{\log N}\log\left(1+\beta(e^{-\theta}-1)\right)\\
    &\longrightarrow
    A(e^\theta-1)+B(e^{-\theta}-1)
    \eqqcolon \Lambda(\theta).
    \end{align*}
    The quadratic error from expanding the two logarithms is
    \[
    O\left(\frac{r\alpha^2+s\beta^2}{\log N}\right)
    =
    O\left(\frac{\log N}{N}\right)
    =
    o(1).
    \]
    Thus Lemma~\ref{lem:gartner-ellis-left-tail} applies to \(\rS_N\) with
    rate \(\log N\).
\begin{itemize}
    \item If \(B<A\), then \(\Lambda'(0)=A-B>0\). For every fixed \(\delta>0\),
    Lemma~\ref{lem:gartner-ellis-left-tail} gives
    \begin{align*}
    \lim_{N\to\infty}
    \frac{1}{\log N}\log \P(\rS_N\le -\delta\log N)
    =
    -J(-\delta),
    \qquad
    J(x)\coloneqq \sup_{\theta\in\R}\{x\theta-\Lambda(\theta)\}.
    \end{align*}
    Since \(\{\rS_N\le -\delta\log N\}\subseteq\{\rX-\rY>0\}\) for all
    sufficiently large \(N\), letting \(\delta\downarrow0\) gives
    \[
    \liminf_{N\to\infty}
    \frac{1}{\log N}\log \P(\rX-\rY>0)
    \ge -J(0).
    \]
    The value at zero is
    \[
    J(0)
    =
    -\inf_{\theta\in\R}\Lambda(\theta)
    =
    (\sqrt A-\sqrt B)^2
    =
    I(A,B),
    \]
    where the minimizing point satisfies \(e^\theta=\sqrt{B/A}\). Combining
    this with \eqref{eq:binomial-difference-chernoff-upper} gives the desired
    logarithmic asymptotic when \(B<A\).

    \item If \(A=B\), then for every \(\delta>0\), the same argument as in the previous case applies with
    \(-\delta<0=\Lambda'(0)\), and \(J(-\delta)\to J(0)=0\) as
    \(\delta\downarrow0\). The Chernoff upper bound gives the matching upper tail exponent \(0\).

    \item If \(B>A\), then \(I(A,B)=0\), and
    \eqref{eq:binomial-difference-chernoff-lower} gives
    \(\P(\rX-\rY\le0)=o(1)\). Hence \(\P(\rX-\rY>0)=1-o(1)\), which is the desired exponent \(0\).
\end{itemize}
Therefore, along the chosen subsequence,
    \[
    \P(\rX-\rY>0)
    =
    N^{-I(A,B)+o(1)}.
    \]
    Since \(I(A_N,B_N)\to I(A,B)\), this proves \eqref{eqn:binomial-difference-upper-tail} uniformly.

    The proof of \eqref{eqn:binomial-difference-lower-tail} is the same with the
    roles of \(\rX\) and \(\rY\) interchanged. Equivalently, apply the preceding
    argument to \(\rT_N\coloneqq \rX-\rY\). Then
    \[
    \frac{1}{\log N}\log \E e^{\theta \rT_N}
    \longrightarrow
    B(e^\theta-1)+A(e^{-\theta}-1).
    \]
    The same three-case analysis as in the preceding argument implies
    \[
    \P(\rX-\rY\le0)
    =
    N^{-I(B,A)+o(1)}
    =
    N^{-I(B_N,A_N)+o(1)}.
    \]
    This proves the second estimate uniformly along every convergent
    subsequence and completes the proof.
    \end{proof}

    \begin{corollary}[Large-deviation tail bounds for one-vertex flip probabilities] \label{cor:large-deviation-tail-flip-probabilities}
    Fix \(t\ge 0\), and condition on \(\rvy_{t}\). Recall the distributions of
    \(\rD_{t}^{\gR}\) and \(\rD_{t}^{\gB}\) from
    \eqref{eqn:D-Rt-distribution} and \eqref{eqn:D-Bt-distribution}, and the
    flip probabilities \(p_t^{\gB},q_t^{\gB},p_t^{\gR},q_t^{\gR}\) from
    \eqref{eqn:p-q-B} and \eqref{eqn:p-q-R}. Recall
    \(A_t^\gR,B_t^\gB,A_t^\gB,B_t^\gR\) from
    Lemma~\ref{lem:sublinear-jump-one-step-bound} and the function \(I(x,y)\)
    from \eqref{eqn:LDP-rate-function}. Then uniformly over all configurations satisfying \(A_{t}^{\gB}\ge c_0\), \(B_{t}^{\gR}\ge c_0\), we have
\begin{subequations}
\begin{align}
p_t^{\gB}
&=
N^{-I(A_{t}^{\gB},B_{t}^{\gR})+o(1)}
\label{eq:blue-to-red-flip-rate}\\
q_t^{\gB}
&=
N^{-I(B_{t}^{\gR},A_{t}^{\gB})+o(1)}.
\label{eq:blue-stays-blue-rate}
\end{align}
\end{subequations}
Similarly, uniformly over all configurations satisfying \(A_{t}^{\gR}\geq c_0\), \(B_{t}^{\gB}\geq c_0\),
    \begin{subequations}
    \begin{align}
    p_t^{\gR}
    &=
    N^{-I(A_{t}^{\gR},B_{t}^{\gB})+o(1)}\,,
    \label{eq:red-to-blue-flip-rate}\\
    q_t^{\gR}
    &=
    N^{-I(B_{t}^{\gB},A_{t}^{\gR})+o(1)}.
    \label{eq:red-stays-red-rate}
    \end{align}
    \end{subequations}
    \end{corollary}

    \begin{proof}[Proof of Corollary~\ref{cor:large-deviation-tail-flip-probabilities}]
    Condition on \(\rvy_{t}\). By \eqref{eqn:D-Bt-distribution}, applying
    Lemma~\ref{lem:binomial-difference-large-deviation-tail} with
    \(s=|\gR_{t}|\) and \(r=|\gB_{t}|-1\) identifies the general parameters as
    \(B_N=B_t^{\gR}\) and \(A_N=A_t^{\gB}\). This gives
    \eqref{eq:blue-to-red-flip-rate} and \eqref{eq:blue-stays-blue-rate}.

    Similarly, by \eqref{eqn:D-Rt-distribution}, applying
    Lemma~\ref{lem:binomial-difference-large-deviation-tail} with
    \(s=|\gB_{t}|\) and \(r=|\gR_{t}|-1\) identifies
    \(B_N=B_t^{\gB}\) and \(A_N=A_t^{\gR}\). This gives
    \eqref{eq:red-to-blue-flip-rate} and \eqref{eq:red-stays-red-rate}.
    \end{proof}

    \subsection{Moderate-Deviation Tail: Tilt and Mills Bounds}
    \noindent
    The estimates in this subsection are useful when the relevant flip probability is
    an intermediate Gaussian tail. More precisely, they cover the range
    \(1\lesssim |\mu_N|/\sigma_N\ll\sqrt{\log N}\), retaining both the
    rate-function exponent \(N^{-I(\cdot,\cdot)}\) and the Mills-ratio prefactor
    \((1+|\mu_N|/\sigma_N)^{-1}\).
    
        \begin{lemma}[Moderate-deviation tail approximation]
            \label{lem:moderate-deviation-tail-approximation}
        Consider the problem described in
        \eqref{eqn:X-Y-distribution}--\eqref{eqn:mu-sigma-squared}.
        Fix \(\ell>0\), and let \(h_N\to\infty\) satisfy
        \(h_N=o(\sqrt{\log N})\). Suppose \(A_N,B_N\ge c_0\) and
        \[
        \ell\le x_N\coloneqq -\frac{\mu_N}{\sigma_N}\le h_N.
        \]
        Then, uniformly over all such \(r,s\),
        \begin{align}
        c\,
        N^{-I(A_N,B_N)}
        \frac{x_N}{1+x_N^2}
        \le
        \P(\rD>0)
        \le
        C\,
        N^{-I(A_N,B_N)}
        \frac1{x_N}.
        \label{eq:moderate-deviation-tail-general}
        \end{align}
        Here the rate function \(I(x,y)\) is defined in
        \eqref{eqn:LDP-rate-function}, and the constants \(c,C>0\) depend only on
        \(a,b,c_0,\ell\).
        \end{lemma}
        \begin{proof}[Proof of Lemma~\ref{lem:moderate-deviation-tail-approximation}]
        Since \(\mu_N=(B_N-A_N)\log N\), the condition
        \(-\mu_N/\sigma_N\ge\ell>0\) implies \(A_N>B_N\). Moreover,
        \[
        \sigma_N^2
        =
        \{A_N+B_N+o(1)\}\log N,
        \]
        so \(-\mu_N/\sigma_N\le h_N\) gives
        \[
        0<A_N-B_N
        =
        \frac{-\mu_N}{\log N}
        \le
        C\frac{h_N}{\sqrt{\log N}}
        =
        o(1).
        \]
        Consequently, a direct Berry--Esseen approximation under \(\P\) has absolute error
        \(O((\log N)^{-1/2})\).

        However, in the present moderate-deviation range, this
        error need not be small relative to \(\P(\rD>0)\). We therefore use an
        exponential change of measure: the normalization of the tilted law
        extracts the large-deviation factor, while under the tilted law the
        threshold \(0\) lies in the central window and can be analyzed by
        Berry--Esseen and Mills bounds.

        To introduce the tilted measure $\P_{*}$, we choose the saddle point
        \[
        \lambda_*
        \coloneqq
        \frac12\log\frac{A_N}{B_N}>0.
        \]
        This choice balances the leading sparse intensities:
        \[
        B_Ne^{\lambda_*}
        =
        A_Ne^{-\lambda_*}
        =
        \sqrt{A_NB_N}.
        \]
        Since \(A_N-B_N=o(1)\) and \(A_N,B_N\ge c_0\), we have
        \(\lambda_*=O(A_N-B_N)=o(1)\).
        Define the tilted probability measure \(\P_*\) by
        \[
        \frac{\de\P_*}{\de\P}
        =
        \frac{e^{\lambda_*\rD}}{\E e^{\lambda_*\rD}}.
        \]
        Then the desired probability can be factored as
        \begin{align}
        \P(\rD>0)
        &=
        \E e^{\lambda_*\rD}\,
        \E_*\left[e^{-\lambda_*\rD}\indi{\rD>0}\right].
        \label{eq:moderate-change-of-measure-factorization}
        \end{align}
        Define the tilted mean, variance, and scaled tilting parameter by
        \[
        \mu_*\coloneqq\E_*\rD,
        \qquad
        \sigma_*^2\coloneqq\Var_*(\rD),
        \qquad
        a_N\coloneqq\lambda_*\sigma_*,
        \]
        and let \(\phi(y)=(2\pi)^{-1/2}e^{-y^2/2}\) denote the standard normal
        density.
        We first state the two estimates that drive the argument and verify
        them afterward:
        \begin{subequations}
        \begin{align}
        \E e^{\lambda_*\rD}
        =&\,
        N^{-I(A_N,B_N)}e^{o(1)},
        \label{eq:moderate-saddle-mgf}\\
        \E_*\left[e^{-\lambda_*\rD}\indi{\rD>0}\right]
        =&\,
        \int_0^\infty e^{-a_Ny}\phi(y)\,\de y
        +
        o\left(\frac{a_N}{1+a_N^2}\right),
        \label{eq:moderate-tilted-expectation}
        \end{align}
        \end{subequations}
        Moreover,
        \[
        a_N=(1+o(1))x_N.
        \]

        Recall that the standard Mills bounds are
        \[
        \frac{x}{1+x^2}\phi(x)
        \le
        \Phi(-x)
        \le
        \frac{\phi(x)}{x},
        \qquad x>0.
        \]
        Completing the square and applying these bounds gives
        \begin{align}
        \int_0^\infty e^{-a_Ny}\phi(y)\,\de y
        &=
        e^{a_N^2/2}\Phi(-a_N), \notag\\
        \frac{1}{\sqrt{2\pi}}\frac{a_N}{1+a_N^2}
        &\le
        \int_0^\infty e^{-a_Ny}\phi(y)\,\de y
        \le
        \frac{1}{\sqrt{2\pi}}\frac1{a_N}.
        \label{eq:moderate-mills-integral}
        \end{align}
        Substituting
        \eqref{eq:moderate-saddle-mgf}--\eqref{eq:moderate-tilted-expectation}
        into \eqref{eq:moderate-change-of-measure-factorization}, and using
        \(a_N\asymp x_N\), yields
        \[
        c\,
        N^{-I(A_N,B_N)}
        \frac{x_N}{1+x_N^2}
        \le
        \P(\rD>0)
        \le
        C\,
        N^{-I(A_N,B_N)}
        \frac1{x_N},
        \]
        which proves \eqref{eq:moderate-deviation-tail-general}. It remains to
        verify the two estimates used above.

        \medskip
            \noindent\emph{Proof of \eqref{eq:moderate-saddle-mgf}.}
        For any \(\lambda=o(1)\), independence gives
        \begin{align*}
        \log\E e^{\lambda\rD}
        &=
        s\log\{1+\beta(e^\lambda-1)\}
        +
        r\log\{1+\alpha(e^{-\lambda}-1)\}\\
        &=
        \log N\left[
        B_N(e^\lambda-1)+A_N(e^{-\lambda}-1)
        \right]\\
        &\quad+
        O\left(
        s\beta^2(e^\lambda-1)^2
        +
        r\alpha^2(e^{-\lambda}-1)^2
        \right).
        \end{align*}
        At \(\lambda=\lambda_*\), the error is
        \(O(\log^2N\,\lambda_*^2/N)=o(1)\). The definition of \(\lambda_*\)
        therefore gives
        \begin{align*}
        \log\E e^{\lambda_*\rD}
        &=
        \log N\left[
        B_N\left(\sqrt{A_N/B_N}-1\right)
        +
        A_N\left(\sqrt{B_N/A_N}-1\right)
        \right]+o(1)\\
        &=
        -(\sqrt{A_N}-\sqrt{B_N})^2\log N+o(1)\\
        &=
        -I(A_N,B_N)\log N+o(1),
        \end{align*}
        where the last equality uses \(A_N>B_N\). Exponentiating proves
        \eqref{eq:moderate-saddle-mgf}.

        \medskip
        \noindent\emph{Proof of \eqref{eq:moderate-tilted-expectation}.}
        The density defining \(\P_*\) factorizes over the independent Bernoulli
        summands. Thus the summands remain independent under \(\P_*\). If
        \(\xi\sim\Ber(\beta)\) is a summand of \(\rX\), then
        \[
        \beta_*
        \coloneqq
        \P_*(\xi=1)
        =
        \frac{\beta e^{\lambda_*}}
        {1+\beta(e^{\lambda_*}-1)}.
        \]
        If \(\zeta\sim\Ber(\alpha)\) is a summand of \(\rY\), then its sign in
        \(\rD=\rX-\rY\) is negative, and hence
        \[
        \alpha_*
        \coloneqq
        \P_*(\zeta=1)
        =
        \frac{\alpha e^{-\lambda_*}}
        {1+\alpha(e^{-\lambda_*}-1)}.
        \]
        Consequently, under \(\P_*\),
        \[
        \rD
        \stackrel d=
        \Bin(s,\beta_*)-\Bin(r,\alpha_*).
        \]

        Since \(\lambda_*=o(1)\),
        \[
        \beta_*=\beta e^{\lambda_*}+O(\beta^2|\lambda_*|),
        \qquad
        \alpha_*=\alpha e^{-\lambda_*}+O(\alpha^2|\lambda_*|).
        \]
        It follows that
        \begin{align*}
        \mu_*
        &=
        s\beta_*-r\alpha_*\\
        &=
        \log N\left(
        B_Ne^{\lambda_*}-A_Ne^{-\lambda_*}
        \right)
        +
        O\left(
        (s\beta^2+r\alpha^2)|\lambda_*|
        \right)\\
        &=
        O(\log^2N\,|\lambda_*|/N)
        =
        o(1),
        \end{align*}
        because the leading term vanishes by the choice of \(\lambda_*\).
        Similarly,
        \begin{align*}
        \sigma_*^2
        &=
        s\beta_*(1-\beta_*)+r\alpha_*(1-\alpha_*)\\
        &=
        \{B_Ne^{\lambda_*}+A_Ne^{-\lambda_*}+o(1)\}\log N\\
        &=
        \{2\sqrt{A_NB_N}+o(1)\}\log N
        =
        (1+o(1))\sigma_N^2.
        \end{align*}
        Moreover,
        \[
        \lambda_*
        =
        \frac{A_N-B_N}{2B_N}\{1+o(1)\},
        \qquad
        x_N
        =
        \frac{(A_N-B_N)\log N}{\sigma_N}.
        \]
        Together with the two variance asymptotics, these relations give
        \(a_N=\lambda_*\sigma_*=(1+o(1))x_N\).

        Since \(\lambda_*=o(1)\), the tilted probabilities satisfy
        \(c\log N/N\le\alpha_*,\beta_*\le C\log N/N\). Their total centered
        absolute third moment is at most
        \(C(s\beta_*+r\alpha_*)\le C\log N\), whereas
        \(\sigma_*^2\ge c\log N\). Lemma~\ref{lem:berry-esseen} therefore
        gives
        \[
        \sup_{x\in\R}
        \left|
        \P_*\left(
        \frac{\rD-\mu_*}{\sigma_*}\le x
        \right)
        -
        \Phi(x)
        \right|
        \le
        \frac{C}{\sqrt{\log N}}.
        \]

        Define
        \[
        Z_N\coloneqq\frac{\rD-\mu_*}{\sigma_*},
        \qquad
        b_N\coloneqq\frac{1-\mu_*}{\sigma_*},
        \]
        and let \(F_N\) be the distribution function of \(Z_N\) under
        \(\P_*\). Since \(\rD\) is integer-valued,
        \[
        \{\rD>0\}
        =
        \{\rD\ge1\}
        =
        \{Z_N\ge b_N\}.
        \]
        The estimates above imply
        \[
        b_N=O\left(\frac1{\sqrt{\log N}}\right),
        \qquad
        a_Nb_N=\lambda_*(1-\mu_*)=o(1),
        \qquad
        \lambda_*\mu_*=o(1).
        \]
        Therefore,
        \begin{align}
        \E_*\left[e^{-\lambda_*\rD}\indi{\rD>0}\right]
        &=
        e^{-\lambda_*\mu_*}
        \int_{[b_N,\infty)}e^{-a_Ny}\,\de F_N(y).
        \label{eq:moderate-tilted-stieltjes}
        \end{align}

        The Stieltjes integral in
        \eqref{eq:moderate-tilted-stieltjes} differs from the Gaussian integral
        in \eqref{eq:moderate-tilted-expectation} in two ways. First, its
        integrator is the distribution function \(F_N\) of the tilted lattice
        variable, rather than the standard normal distribution function
        \(\Phi\). Second, its lower endpoint is the lattice threshold \(b_N\),
        rather than \(0\). We control these two errors separately.

        We first replace \(F_N\) by \(\Phi\) while retaining the exact threshold
        \(b_N\). The Berry--Esseen estimate controls the Kolmogorov distance
        between these distribution functions, whereas the integrand has a jump
        at \(b_N\). It is therefore natural to regard the expectation as a
        Stieltjes integral against the bounded-variation function
        \(f_N(y)\coloneqq e^{-a_Ny}\indi{y\ge b_N}\). For all sufficiently
        large \(N\), \(b_N>0\), and
        \(f_N\) jumps from \(0\) to \(e^{-a_Nb_N}\) at \(b_N\) and then
        decreases monotonically to \(0\). Hence its total variation is
        \[
        \operatorname{TV}(f_N)
        =
        e^{-a_Nb_N}+e^{-a_Nb_N}
        =
        2e^{-a_Nb_N}
        =
        O(1).
        \]
        Since \(\Phi\) is
        continuous, the possible lattice atom at the jump point satisfies
        \[
        \P_*(Z_N=b_N)
        \le
        2\sup_{y\in\R}|F_N(y)-\Phi(y)|.
        \]
        Bounded-variation Stieltjes integration by parts, including this
        endpoint term, yields
        \begin{align}
        \left|
        \int_{\R}f_N(y)\,\de F_N(y)
        -
        \int_{\R}f_N(y)\,\de\Phi(y)
        \right|
        &\le
        C\sup_{y\in\R}|F_N(y)-\Phi(y)|
        \operatorname{TV}(f_N) \notag\\
        &=
        O\left(\frac1{\sqrt{\log N}}\right).
        \label{eq:moderate-stieltjes-error}
        \end{align}

        Thus \eqref{eq:moderate-stieltjes-error} is the error incurred by
        replacing the tilted lattice law with its Gaussian approximation,
        without yet changing the event \(\rD>0\).

        After this distributional replacement, the Gaussian integral still
        begins at \(b_N=(1-\mu_*)/\sigma_*\), because the integer-valued event
        \(\rD>0\) is exactly \(\rD\ge1\). The target integral begins at \(0\),
        which is the limiting standardized threshold because
        \(\mu_*=o(1)\) and \(\sigma_*\asymp\sqrt{\log N}\). The cost of removing
        this lattice continuity correction is the integral over the interval
        between \(0\) and \(b_N\):
        \begin{align}
        \left|
        \int_{b_N}^{\infty}e^{-a_Ny}\phi(y)\,\de y
        -
        \int_0^{\infty}e^{-a_Ny}\phi(y)\,\de y
        \right|
        &\le
        \frac{|b_N|e^{a_N|b_N|}}{\sqrt{2\pi}}
        =
        O\left(\frac1{\sqrt{\log N}}\right).
        \label{eq:moderate-lattice-threshold-error}
        \end{align}
        Hence \eqref{eq:moderate-lattice-threshold-error} controls only the
        displacement of the threshold; it is separate from the normal
        approximation error in \eqref{eq:moderate-stieltjes-error}.

        Finally, \(a_N = (1 + o(1)) x_N\ge c\ell\) and
        \(a_N=O(h_N)=o(\sqrt{\log N})\), so
        \[
        \frac{(\log N)^{-1/2}}{a_N/(1+a_N^2)}
        =
        \frac1{a_N\sqrt{\log N}}
        +
        \frac{a_N}{\sqrt{\log N}}
        =
        o(1).
        \]
        Therefore, the sum of the two errors is
        \(o(a_N/(1+a_N^2))\), which is negligible relative to the lower Mills
        scale in \eqref{eq:moderate-mills-integral}.
        Combining this estimate with
        \eqref{eq:moderate-mills-integral} and
        \eqref{eq:moderate-tilted-stieltjes}--\eqref{eq:moderate-lattice-threshold-error},
        and using \(e^{-\lambda_*\mu_*}=1+o(1)\), proves
        \eqref{eq:moderate-tilted-expectation}. This completes the proof.
        \end{proof}
    
        \begin{corollary}[Moderate-deviation tail estimates for one-vertex flip probabilities]
        \label{cor:moderate-deviation-tail-flip-probabilities}
        Recall \(x_t^{\gR}\) and \(x_t^{\gB}\) from
        Corollary~\ref{cor:gaussian-flip-probabilities}. Fix \(\ell>0\), and let \(h_N\to\infty\) satisfy \(h_N=o(\sqrt{\log N})\). Condition on
        \(\rvy_t\) satisfying \(|\gR_t|\ge c_0N\) and \(|\gB_t|\ge c_0N\) for some constant \(c_0>0\). Then the following estimates hold uniformly over all such configurations. If
        \(\ell\le -x_t^{\gR}\le h_N\), then
        \begin{align}
        p_t^{\gR}
        \asymp
        \frac{N^{-I(A_t^{\gR},B_t^{\gB})}}{1+|x_t^{\gR}|}.
        \label{eq:moderate-deviation-pR-rate-function}
        \end{align}
        If \(\ell\le -x_t^{\gB}\le h_N\), then
        \begin{align}
        p_t^{\gB}
        \asymp
        \frac{N^{-I(A_t^{\gB},B_t^{\gR})}}{1+|x_t^{\gB}|}.
        \label{eq:moderate-deviation-pB-rate-function}
        \end{align}
        Here \(A_t^{\gR}\), \(B_t^{\gB}\), \(A_t^{\gB}\), and
        \(B_t^{\gR}\) are defined in
        Lemma~\ref{lem:sublinear-jump-one-step-bound}, and the rate function
        \(I(x,y)\) is defined in \eqref{eqn:LDP-rate-function}. The constants
        in \(\asymp\) depend only on \(a,b,c_0,\ell\).
        \end{corollary}
    
        \begin{proof}[Proof of Corollary~\ref{cor:moderate-deviation-tail-flip-probabilities}]
        We first show explicitly how the sharper Mills-ratio estimate
        \eqref{eq:moderate-deviation-tail-general} implies the coarser form in
        the corollary. For every \(x\ge\ell\),
        \begin{align}
        \frac{x}{1+x^2}
        &=
        \frac{1}{1+x}\frac{x(1+x)}{1+x^2}
        \ge
        \frac{c_\ell}{1+x},
        &
        \frac1x
        &=
        \frac{1}{1+x}\frac{1+x}{x}
        \le
        \frac{C_\ell}{1+x},
        \label{eq:moderate-mills-to-coarse}
        \end{align}
        where one may take
        \[
        c_\ell\coloneqq\min\{\ell/2,1\},
        \qquad
        C_\ell\coloneqq1+\ell^{-1}.
        \]
        Indeed, if \(\ell\le x\le1\), then
        \(x(1+x)/(1+x^2)\ge\ell/2\); if \(x\ge1\), then
        \(x(1+x)/(1+x^2)\ge1\). The upper comparison follows directly from
        \((1+x)/x=1+1/x\le1+1/\ell\). Consequently,
        \eqref{eq:moderate-deviation-tail-general} gives, uniformly for
        \(\ell\le x_N\le h_N\),
        \begin{align}
        c'
        \frac{N^{-I(A_N,B_N)}}{1+x_N}
        \le
        \P(\rD>0)
        \le
        C'
        \frac{N^{-I(A_N,B_N)}}{1+x_N},
        \label{eq:moderate-deviation-tail-coarse}
        \end{align}
        where \(c',C'>0\) depend only on the parameters listed in
        Lemma~\ref{lem:moderate-deviation-tail-approximation}.

        We now apply \eqref{eq:moderate-deviation-tail-coarse} to the two
        one-vertex degree differences. Suppose first that
        \(\ell\le-x_t^{\gR}\le h_N\). By
        \eqref{eqn:D-Rt-distribution}, the parameters in the general problem
        are
        \[
        s=|\gB_t|,
        \qquad
        r=|\gR_t|-1,
        \qquad
        A_N=A_t^{\gR},
        \qquad
        B_N=B_t^{\gB},
        \]
        and \(\mu_N=m_t^{\gR}\), \(\sigma_N^2=v_t^{\gR}\). Thus
        \[
        x_N=-\frac{\mu_N}{\sigma_N}
        =-x_t^{\gR}
        =|x_t^{\gR}|.
        \]
        Since both camps have size at least \(c_0N\), for all sufficiently
        large \(N\),
        \[
        A_t^{\gR}=\frac{a(|\gR_t|-1)}{N}\ge\frac{ac_0}{2},
        \qquad
        B_t^{\gB}=\frac{b|\gB_t|}{N}\ge bc_0.
        \]
        Hence the hypotheses of
        Lemma~\ref{lem:moderate-deviation-tail-approximation} hold, and
        \eqref{eq:moderate-deviation-tail-coarse} gives
        \[
        p_t^{\gR}
        =
        \P(\rD_t^{\gR}>0\mid\rvy_t)
        \asymp
        \frac{N^{-I(A_t^{\gR},B_t^{\gB})}}
        {1+|x_t^{\gR}|}.
        \]
        This proves \eqref{eq:moderate-deviation-pR-rate-function}.

        Suppose next that \(\ell\le-x_t^{\gB}\le h_N\). By
        \eqref{eqn:D-Bt-distribution}, now
        \[
        s=|\gR_t|,
        \qquad
        r=|\gB_t|-1,
        \qquad
        A_N=A_t^{\gB},
        \qquad
        B_N=B_t^{\gR},
        \]
        while \(\mu_N=m_t^{\gB}\), \(\sigma_N^2=v_t^{\gB}\), and
        \(x_N=-x_t^{\gB}=|x_t^{\gB}|\). The linear camp-size assumptions
        similarly imply
        \[
        A_t^{\gB}\ge\frac{ac_0}{2},
        \qquad
        B_t^{\gR}\ge bc_0
        \]
        for all sufficiently large \(N\). Applying
        \eqref{eq:moderate-deviation-tail-coarse} therefore gives
        \[
        p_t^{\gB}
        =
        \P(\rD_t^{\gB}>0\mid\rvy_t)
        \asymp
        \frac{N^{-I(A_t^{\gB},B_t^{\gR})}}
        {1+|x_t^{\gB}|},
        \]
        which proves \eqref{eq:moderate-deviation-pB-rate-function}.
        \end{proof}    

\section{Read-\texorpdfstring{\(k\)}{k} Chernoff Bound}
\label{sec:read-k-chernoff-bound}
We first introduce the concept of a read-\(k\) family.
\begin{definition}[Read-\(k\) family]
    Let \(\rX_1,\ldots,\rX_m\) be independent random variables. For each \(j\in[r]\), let \(\varnothing\ne P_j\subseteq[m]\), and let \(f_j\) be a Boolean function of $\{\rX_i\}_{i\in P_j}$. Suppose that each \(\rX_i\) is read by at most \(k\) of the functions \(f_j\), i.e., \(\bigl|\{j\in[r]: i\in P_j\}\bigr|\le k\) for every \(i\in[m]\). Then, the random variables $\{\rY_j \coloneqq f_j(\{\rX_i\}_{i\in P_j})\}_{j\in[r]}$ form a read-\(k\) family.
\end{definition}
We next state the read-\(k\) Chernoff bound.
\begin{lemma}[Read-\(k\) Chernoff bound {\cite[Theorem~1.1]{gavinsky2015tail}}]
    \label{lem:read-k}
     Let \(\{\rY_j\}_{j\in[r]}\) be a read-\(k\) family of indicator random variables with \(p_j \coloneqq\P(\rY_j=1)\), and denote their average by \(p \coloneqq r^{-1}\sum_{j=1}^r p_j\).
    Then, for every \(\varepsilon>0\),
    \begin{subequations}
    \begin{align}
    \P\bigg(
    \sum_{j=1}^r \rY_j\ge (p+\varepsilon)r
    \bigg)
    &\le
    \exp\bigg(
    -\frac{r}{k}\D_{\mathrm{KL}}(p+\varepsilon\|p)
    \bigg),\\
    \P\bigg(
    \sum_{j=1}^r \rY_j\le (p-\varepsilon)r
    \bigg)
    &\le
    \exp\bigg(
    -\frac{r}{k}\D_{\mathrm{KL}}(p-\varepsilon\|p)
    \bigg),
    \end{align}
    \end{subequations}
    where the Kullback-Leibler divergence is defined as
    \[
    \D_{\mathrm{KL}}(q\|p)\coloneqq q\log(q/p)+(1-q)\log((1-q)/(1-p)).
    \]
    \end{lemma}

    \begin{corollary}[Centered concentration for a read-\(k\) sum]
    \label{cor:read-k-chernoff-bound-sum}
    Let \(\{\rY_j\}_{j\in[r]}\) be a read-\(k\) family of indicator random
    variables, and set \(\rS\coloneqq\sum_{j=1}^r\rY_j\). Then, for every
    \(s>0\),
    \begin{align}
    \P\left(\left|\rS-\E \rS\right|\ge s\right)
    \le
    2\exp\left(-\frac{2s^2}{kr}\right).
    \label{eq:read-k-centered-concentration}
    \end{align}
    \end{corollary}

    \begin{proof}[Proof of Corollary~\ref{cor:read-k-chernoff-bound-sum}]
    Put \(p\coloneqq r^{-1}\E \rS\) and \(\varepsilon\coloneqq s/r\). If
    \(p\in\{0,1\}\), then every \(\rY_j\) is deterministic and the conclusion
    is immediate. Hence assume \(0<p<1\).

    We first record the relative-entropy estimate used for both tails. For every
    \(q\in(0,1)\),
    \begin{align*}
    \D_{\mathrm{KL}}(q\|p)
    &=
    q\log\frac{q}{p}
    +(1-q)\log\frac{1-q}{1-p}\\
    &=
    \int_{p\wedge q}^{p\vee q}
    \frac{|q-x|}{x(1-x)}\,\de x\\
    &\ge
    4\int_{p\wedge q}^{p\vee q}|q-x|\,\de x
    =
    2(q-p)^2.
    \end{align*}
    By continuity, the same bound holds for \(q\in\{0,1\}\).
    If \(p+\varepsilon>1\), then the event \(\{\rS-\E\rS\ge s\}\) is empty.
    Otherwise, Lemma~\ref{lem:read-k} and the displayed estimate with
    \(q=p+\varepsilon\) give
    \[
    \P(\rS-\E \rS\ge s)
    \le
    \exp\left(-\frac{2r\varepsilon^2}{k}\right)
    =
    \exp\left(-\frac{2s^2}{kr}\right).
    \]
    Similarly, if \(p-\varepsilon<0\), then
    \(\{\rS-\E \rS\le-s\}\) is empty; otherwise, Lemma~\ref{lem:read-k} with
    \(q=p-\varepsilon\) yields the same bound. Adding the two one-sided
    estimates proves \eqref{eq:read-k-centered-concentration}.
    \end{proof}

    With the read-\(k\) Chernoff bound, we can prove read-\(2\) concentration
    for the one-step red-camp size and weighted advantage.
    \begin{lemma}[Read-\(2\) concentration for one-step red camp and weighted advantage]
    \label{lem:R1-read-two-concentration}
    Fix \(t\ge0\) and condition on \(\rvy_t\). Then, for every \(u>0\),
    \begin{subequations}
    \begin{align}
    \P\left(
    \left||\gR_{t+1}|-\E\left[|\gR_{t+1}|\,\middle|\,\rvy_t\right]\right|
    \ge u
    \,\middle|\,
    \rvy_t
    \right)
    &\le
    2\exp(-u^2/N), \label{eq:Rtplus-one-read-two-concentration}\\
    \P\left(
    \left|
    \widetilde{\Delta}_{t+1}
    -
    \E\left[\widetilde{\Delta}_{t+1}\,\middle|\,\rvy_t\right]
    \right|
    \ge u
    \,\middle|\,
    \rvy_t
    \right)
    &\le
    2\exp\left(-\frac{u^2}{(a+b)^2N}\right).
    \label{eq:tilde-delta-read-two-concentration}
    \end{align}
    \end{subequations}
    \end{lemma}
    \begin{proof}[Proof of Lemma~\ref{lem:R1-read-two-concentration}]
        Throughout the proof, the configuration \(\rvy_t\) is fixed. We write
        \(\P_t(\cdot)\coloneqq\P(\cdot\mid\rvy_t)\) and
        \(\E_t[\cdot]\coloneqq\E[\cdot\mid\rvy_t]\).
    
        For every unordered pair \(\{u,v\}\subset\gV\), let
        \(\rX_{uv}\coloneqq\ermA_{t+1}(u,v)\) be the edge indicator for update \(t+1\). Under
        \(\P_t\), the family \(\{\rX_{uv}:u<v\}\) is independent. The marginal law of
        \(\rX_{uv}\) is Bernoulli with parameter \(\alpha\) if \(u\) and \(v\) have the
        same color in \(\rvy_t\), and Bernoulli with parameter \(\beta\) otherwise.
    
        For each vertex \(v\in\gV\), set
        \(\rY_v\coloneqq\indi{v\in\gR_{t+1}}\). Once \(\rvy_t\) is fixed, the value of
        \(\rY_v\) is a deterministic Boolean function of the incident edge indicators
        \(\{\rX_{uv}:u\in\gV\setminus\{v\}\}\): these indicators determine the number
        of red and blue neighbors of \(v\) at update \(t+1\), and hence determine the updated
        color of \(v\) by the majority rule, with ties resolved by keeping the old color.
        Therefore \(\{\rY_v:v\in\gV\}\) is a read-\(2\) family under \(\P_t\), since an
        edge variable \(\rX_{uv}\) is used only in the two vertex functions defining
        \(\rY_u\) and \(\rY_v\).
    
        Let \(\rS\coloneqq\sum_{v\in\gV}\rY_v=|\gR_{t+1}|\). Applying
        Corollary~\ref{cor:read-k-chernoff-bound-sum} under \(\P_t\), with
        \(r=N\), \(k=2\), and \(s=u\), gives
        \[
        \P_t\left(
        \left|\rS-\E_t \rS\right|\ge u
        \right)
        \le
        2\exp(-u^2/N).
        \]
        This proves \eqref{eq:Rtplus-one-read-two-concentration}. Moreover,
        \eqref{eqn:red-size-from-advantage} gives
        $
        \widetilde{\Delta}_{t+1}=bN-(a+b)|\gR_{t+1}|
        $,
        which implies
        \[
        \left|
        \widetilde{\Delta}_{t+1}
        -
        \E_t[\widetilde{\Delta}_{t+1}]
        \right|
        =
        (a+b)\left||\gR_{t+1}|-\E_t[|\gR_{t+1}|]\right|.
        \]
        Applying \eqref{eq:Rtplus-one-read-two-concentration} with
        \(u/(a+b)\) proves \eqref{eq:tilde-delta-read-two-concentration}.
        \end{proof}
        \begin{lemma}[MGF form of the read-\(k\) bound]
            \label{lem:read-k-mgf}
            Let \(Y_1,\ldots,Y_r\) be a read-\(k\) family of indicator random variables, and
            let \(S\coloneqq\sum_{j=1}^r Y_j\) and
            \(\mu\coloneqq\E S\).
            Then there exist constants \(C>0\) and \(\theta_0>0\), depending only on \(k\),
            such that for every \(|\theta|\le \theta_0\),
            \(\E\exp(\theta(S-\mu))\le\exp(C\theta^2\mu)\).
            \end{lemma}
        
            \begin{proof}[Proof of Lemma~\ref{lem:read-k-mgf}]
            By the definition of a read-\(k\) family, there are independent random
            variables \(X_1,\ldots,X_m\), nonempty subsets \(P_j\subseteq[m]\), and Boolean
            functions \(f_j\) such that
            \[
            Y_j=f_j\bigl((X_i)_{i\in P_j}\bigr),
            \qquad 1\le j\le r,
            \]
            and each \(i\in[m]\) belongs to at most \(k\) of the sets \(P_j\). Fix
            \(\theta\in\R\), and for each \(j\in[r]\) define the nonnegative
            function
            \[
            g_j\bigl((X_i)_{i\in P_j}\bigr)
            \coloneqq
            \exp\left(
            \theta f_j\bigl((X_i)_{i\in P_j}\bigr)
            \right)
            =
            e^{\theta Y_j}.
            \]
            We now specify the objects used in
            Theorem~\ref{thm:finner-generalized-holder}. For each original
            function \(g_j\), set
            \[
            S_j\coloneqq P_j,
            \qquad
            p_j\coloneqq k,
            \qquad j\in[r].
            \]
            Thus \(S_j\) is the set of coordinates on which \(g_j\) depends,
            and \(p_j=k\) is the corresponding H\"older exponent. For every
            coordinate \(i\in[m]\), let
            \[
            M_i\coloneqq\{j\in[r]:i\in S_j\},
            \qquad
            d_i\coloneqq|M_i|\le k.
            \]
            The original functions contribute
            \[
            \sum_{j\in M_i}\frac1{p_j}=\frac{d_i}{k}
            \]
            to the coordinate sum in Finner's hypothesis. This already equals
            \(1\) when \(d_i=k\). If \(d_i<k\), introduce a new function index
            \(i^\star\) and define
            \[
            g_{i^\star}(X_i)\equiv1,
            \qquad
            S_{i^\star}\coloneqq\{i\},
            \qquad
            p_{i^\star}\coloneqq\frac{k}{k-d_i}.
            \]
            Since \(S_{i^\star}\) contains only coordinate \(i\), this
            auxiliary function changes no other coordinate sum, while at
            coordinate \(i\) it gives
            \[
            \sum_{j\in M_i}\frac1{p_j}
            +\frac1{p_{i^\star}}
            =
            \frac{d_i}{k}+\frac{k-d_i}{k}
            =
            1.
            \]
            Repeating this construction for every \(i\) with \(d_i<k\)
            produces an augmented family satisfying
            \(\sum_{j:i\in S_j}p_j^{-1}=1\) at every coordinate. Moreover,
            since each auxiliary function is identically one and the law of
            \(X_i\) is a probability measure,
            \[
            \left(\E
            g_{i^\star}^{p_{i^\star}}\right)^{1/p_{i^\star}}=1.
            \]

            Since \(S=\sum_{j=1}^rY_j\), applying
            Theorem~\ref{thm:finner-generalized-holder} to the augmented family
            under the product law of \(X_1,\ldots,X_m\) yields
            \begin{align}
            \E e^{\theta S}
            &=
            \E\left[
            \prod_{j=1}^r g_j
            \prod_{\substack{i\in[m]\\d_i<k}}g_{i^\star}
            \right]
            \le
            \prod_{j=1}^r
            \left(\E g_j^k\right)^{1/k}
            \prod_{\substack{i\in[m]\\d_i<k}}
            \left(
            \E g_{i^\star}^{p_{i^\star}}
            \right)^{1/p_{i^\star}} \notag\\
            &=
            \prod_{j=1}^r
            \left(\E g_j^k\right)^{1/k}
            =
            \prod_{j=1}^r
            \left(\E e^{k\theta Y_j}\right)^{1/k}.
            \label{eq:read-k-mgf-finner-step}
            \end{align}
        
            We now evaluate the terms in the product. Denote \(\pi_j\coloneqq\P(Y_j=1)\). Since \(Y_j\) is an indicator variable, we have
            \[
            \E e^{k\theta Y_j}
            =
            1-\pi_j+\pi_j e^{k\theta}
            =
            1+\pi_j(e^{k\theta}-1).
            \]
            The quantity \(\pi_j(e^{k\theta}-1)\) is strictly larger than \(-1\).
            Hence \(\log(1+u)\le u\), valid for every \(u>-1\), and
            \eqref{eq:read-k-mgf-finner-step} give
            \begin{align}
            \log\E e^{\theta S}
            &\le
            \frac1k
            \sum_{j=1}^r
            \log\left(1+\pi_j(e^{k\theta}-1)\right)
            \le
            \frac{e^{k\theta}-1}{k}
            \sum_{j=1}^r\pi_j
            =
            \frac{e^{k\theta}-1}{k}\mu,
            \label{eq:read-k-uncentered-mgf}
            \end{align}
            where the last equality follows from
            \(\mu=\E S=\sum_{j=1}^r\pi_j\).
        
            Centering \(S\) and using \eqref{eq:read-k-uncentered-mgf}, we obtain
            \begin{align}
            \log\E e^{\theta(S-\mu)}
            &=
            -\theta\mu+\log\E e^{\theta S}
            \le
            \frac{e^{k\theta}-1-k\theta}{k}\mu.
            \label{eq:read-k-centered-mgf-preliminary}
            \end{align}
            For every \(x\in\R\), Taylor's formula with integral remainder gives
            \[
            e^x-1-x
            =
            x^2\int_0^1(1-u)e^{ux}\,\de u
            \le
            \frac{x^2}{2}e^{|x|}.
            \]
            Applying this estimate with \(x=k\theta\) to
            \eqref{eq:read-k-centered-mgf-preliminary} yields
            \[
            \log\E e^{\theta(S-\mu)}
            \le
            \frac{k}{2}e^{k|\theta|}\theta^2\mu.
            \]
            We may therefore take
            \(\theta_0=1/k\) and \(C=ek/2\). For every
            \(|\theta|\le\theta_0\),
            \[
            \E e^{\theta(S-\mu)}
            \le
            \exp(C\theta^2\mu),
            \]
            which proves the lemma.
            \end{proof}            

\section{Deferred Proofs in Section~\ref{sec:one-step-evolution-advantages}}
\label{app:deferred-one-step-evolution-advantages}

\begin{lemma}[Expectations and variances of the degree differences]\label{lem:expected-variance-degree-differences}
    Under Assumption~\ref{ass:sparse-assortative-sbm}, the expected degree
    differences at time \(t\) can be expressed as
    \begin{align}
        m_{t}^{\gR} = (\widetilde{\Delta}_{t}+a)\log (N)/N, \qquad m_{t}^{\gB} = -\big((a-b)N + \widetilde{\Delta}_{t} - a\big)\log (N)/N. \label{eqn:expected-degree-differences-R-B}
    \end{align}
Furthermore, the variances of the degree differences at time \(t\) satisfy
\begin{align}
    v_{t}^{\gR} \asymp \log(N), \qquad v_{t}^{\gB} \asymp \log(N). \label{eqn:variance-degree-differences-R-B}
\end{align}
\end{lemma}
\begin{proof}[Proof of Lemma \ref{lem:expected-variance-degree-differences}]
    The proof follows directly from linearity of expectation and additivity of
    variance for independent summands.
    Substituting the parameters from
    Assumption~\ref{ass:sparse-assortative-sbm} gives
\begin{align*}
m_{t}^{\gR}
\coloneqq &\,
\E\left[\rD_{t}^{\gR}\mid \rvy_{t}\right]
=
\big[ b|\gB_{t}|-a(|\gR_{t}|-1) \Big]\frac{\log(N)}{N}
=
(\widetilde{\Delta}_{t}+a)\frac{\log (N)}{N},\\
v_{t}^{\gR}
\coloneqq &\,
\Var\left(\rD_{t}^{\gR}\mid \rvy_{t}\right)
=
|\gB_{t}|\beta(1-\beta)
+
\left(|\gR_{t}|-1\right)\alpha(1-\alpha) \notag\\
=&\,
\left[
b|\gB_{t}|(1-\beta)
+
a\left(|\gR_{t}|-1\right)(1-\alpha)
\right]\frac{\log(N)}{N},
\end{align*}
where $v_{t}^{\gR} \asymp \log(N)$ since at least one of $|\gB_{t}|$ or $|\gR_{t}|$ is linear in $N$. Similarly, we have
\begin{align*}
m_{t}^{\gB}
\coloneqq &\,
\E\left[\rD_{t}^{\gB}\mid \rvy_{t}\right]
=
\frac{
b|\gR_{t}|-a(|\gB_{t}|-1)
}{N}\log (N)
=
-(a-b)\log (N)
-
\frac{\widetilde{\Delta}_{t}}{N}\log (N)
+
\frac{a\log (N)}{N},\\
v_{t}^{\gB}
\coloneqq &\,
\Var\left(\rD_{t}^{\gB}\mid \rvy_{t}\right)
=
|\gR_{t}|\beta(1-\beta)
+
\left(|\gB_{t}|-1\right)\alpha(1-\alpha) \notag\\
=&\,
\left[b|\gR_{t}|(1-\beta)
+
a\left(|\gB_{t}|-1\right)(1-\alpha)
\right] \frac{\log (N)}{N},
\end{align*}
where $v_{t}^{\gB} \asymp \log(N)$ follows from the same argument.
\end{proof}

\begin{lemma}[Conditional expectations of camp sizes and advantages after one update]
    \label{lem:conditional-expectations-next-day}
    For all \(t\ge0\), the conditional expectations of the camp sizes at time
    \(t+1\) can be expressed as
    \begin{subequations}
    \begin{align}
        \E\left[ |\gB_{t+1}|\,\middle|\, \rvy_{t} \right]
&\, = |\gR_{t}|p_t^{\gR} + |\gB_{t}|q_t^{\gB}, \label{eq:B-next-expectation}\\
        \E\left[ |\gR_{t+1}|\,\middle|\, \rvy_{t} \right]
&\, = |\gR_{t}|q_t^{\gR} + |\gB_{t}|p_t^{\gB}. \label{eq:R-next-expectation}
    \end{align}
    \end{subequations}
    Indeed, every red-to-blue flip increases \(\Delta_{t}\) by \(2\), while every
    blue-to-red flip decreases it by \(2\).
    \begin{align}
    \E[\Delta_{t+1}\mid\rvy_{t}]
    =\Delta_{t}+2(|\gR_{t}|p_{t}^{\gR}-|\gB_{t}|p_{t}^{\gB}). \label{eqn:Delta-one-step-drift}
    \end{align}
    Furthermore, the conditional expectation of the weighted advantage at
    time \(t+1\) can be expressed as
    \begin{subequations}
        \begin{align}
            \E\left[\widetilde{\Delta}_{t+1} \,\middle|\, \rvy_{t}\right]
            &\, = \widetilde{\Delta}_{t} + (a+b) \left( |\gR_{t}|p_t^{\gR} - |\gB_{t}|p_t^{\gB} \right) \label{eq:gap-next-expectation}\\
            &\, = bN\bigl(p_t^{\gR}-p_t^{\gB}\bigr)
            +
            \widetilde{\Delta}_{t}
            \bigl(1-p_t^{\gR}-p_t^{\gB}\bigr)
            -
            (a-b)N p_t^{\gB}.
            \label{eq:gap-expectation-expanded}
        \end{align}
    \end{subequations}
\end{lemma}
\begin{proof}[Proof of Lemma~\ref{lem:conditional-expectations-next-day}]
Condition on \(\rvy_t\). By the update rule in \eqref{eqn:opinion-update} and
the degree differences in \eqref{eqn:degree-differences}, a red vertex
\(v\in\gR_t\) contributes to \(\gB_{t+1}\) if and only if
\(\rD_t^{\gR}(v)>0\), while it contributes to \(\gR_{t+1}\) if and only if
\(\rD_t^{\gR}(v)\le0\). Similarly, a blue vertex \(u\in\gB_t\) contributes to
\(\gR_{t+1}\) if and only if \(\rD_t^{\gB}(u)>0\), while it contributes to
\(\gB_{t+1}\) if and only if \(\rD_t^{\gB}(u)\le0\). Therefore,
\begin{align*}
|\gB_{t+1}|
&=
\sum_{v\in\gR_t}\indi{\rD_t^{\gR}(v)>0}
+
\sum_{u\in\gB_t}\indi{\rD_t^{\gB}(u)\le0},\\
|\gR_{t+1}|
&=
\sum_{v\in\gR_t}\indi{\rD_t^{\gR}(v)\le0}
+
\sum_{u\in\gB_t}\indi{\rD_t^{\gB}(u)>0}.
\end{align*}
Taking conditional expectations and using \eqref{eqn:p-q-R} and
\eqref{eqn:p-q-B}, we obtain \eqref{eq:B-next-expectation} and
\eqref{eq:R-next-expectation}.

We now compute the conditional expectation of the weighted advantage. By
\eqref{eqn:weighted-advantage}, \eqref{eq:B-next-expectation}, and
\eqref{eq:R-next-expectation},
\begin{align*}
\E\left[\widetilde{\Delta}_{t+1}\,\middle|\,\rvy_t\right]
&=
\E\left[ b|\gB_{t+1}|-a|\gR_{t+1}|\,\middle|\,\rvy_t\right]\\
&=
b\left(|\gR_t|p_t^{\gR}+|\gB_t|q_t^{\gB}\right)
-
a\left(|\gR_t|q_t^{\gR}+|\gB_t|p_t^{\gB}\right)\\
&=
b|\gB_t|-a|\gR_t|
+
(a+b)|\gR_t|p_t^{\gR}
-
(a+b)|\gB_t|p_t^{\gB} \\
&=
\widetilde{\Delta}_t
+
(a+b)\left(|\gR_t|p_t^{\gR}-|\gB_t|p_t^{\gB}\right),
\end{align*}
where the third equality uses \(q_t^{\gR}=1-p_t^{\gR}\) and
\(q_t^{\gB}=1-p_t^{\gB}\). This proves \eqref{eq:gap-next-expectation}. It remains only to rewrite the
same expression in terms of
\(\widetilde{\Delta}_t\) and \(N\). From
\eqref{eqn:blue-size-from-advantage} and
\eqref{eqn:red-size-from-advantage},
\((a+b)|\gR_t|=bN-\widetilde{\Delta}_t\) and
\((a+b)|\gB_t|=aN+\widetilde{\Delta}_t\). Substituting these identities into
\eqref{eq:gap-next-expectation} yields
\begin{align*}
\E\left[\widetilde{\Delta}_{t+1}\,\middle|\,\rvy_t\right]
&=
\widetilde{\Delta}_t
+
\left(bN-\widetilde{\Delta}_t\right)p_t^{\gR}
-
\left(aN+\widetilde{\Delta}_t\right)p_t^{\gB} \\
&=
bN\bigl(p_t^{\gR}-p_t^{\gB}\bigr)
+
\widetilde{\Delta}_t\bigl(1-p_t^{\gR}-p_t^{\gB}\bigr)
-
(a-b)Np_t^{\gB},
\end{align*}
which is \eqref{eq:gap-expectation-expanded}. This completes the proof.
\end{proof}

With the above lemmas in hand, we are now ready to prove the lemmas in Section~\ref{sec:one-step-evolution-advantages}.

\subsection{Proof of Lemma~\ref{lem:flip-ratio-comparison}}

\begin{proof}
Recall that \(p_{t}^{\gB}\) and \(p_{t}^{\gR}\) are defined in
\eqref{eqn:p-q-B} and \eqref{eqn:p-q-R}. By
\eqref{eqn:D-Bt-distribution} and \eqref{eqn:D-Rt-distribution},
\[
\rD_t^{\gB}\stackrel d=
\Bin(|\gR_t|,\beta)-\Bin(|\gB_t|-1,\alpha),
\qquad
\rD_t^{\gR}\stackrel d=
\Bin(|\gB_t|,\beta)-\Bin(|\gR_t|-1,\alpha).
\]
Thus the flip probabilities are
\begin{align*}
p_t^{\gB}
&=
\sum_{i>j}
\binom{|\gR_t|}{i}
\beta^i(1-\beta)^{|\gR_t|-i}
\binom{|\gB_t|-1}{j}
\alpha^j(1-\alpha)^{|\gB_t|-1-j},\\
p_t^{\gR}
&=
\sum_{i>j}
\binom{|\gB_t|}{i}
\beta^i(1-\beta)^{|\gB_t|-i}
\binom{|\gR_t|-1}{j}
\alpha^j(1-\alpha)^{|\gR_t|-1-j}.
\end{align*}
Here, as usual, binomial coefficients outside their natural range are
interpreted as zero.

We compare the summands term by term. If the corresponding red-to-blue summand
is zero, then the blue-to-red summand is also zero. Indeed, either
\(i>|\gB_t|\), which implies \(i>|\gR_t|\), or \(j>|\gR_t|-1\), which together
with \(i>j\) also implies \(i>|\gR_t|\). Thus it remains to consider pairs
\(i>j\) for which the denominator below is nonzero. For such pairs,
\begin{align*}
\frac{\binom{|\gR_t|}{i}\binom{|\gB_t|-1}{j}}
{\binom{|\gB_t|}{i}\binom{|\gR_t|-1}{j}}
&=
\frac{|\gR_t|}{|\gB_t|}
\prod_{\ell=1}^{i-1}\frac{|\gR_t|-\ell}{|\gB_t|-\ell}
\prod_{\ell=1}^{j}\frac{|\gB_t|-\ell}{|\gR_t|-\ell}\\
&=
\frac{|\gR_t|}{|\gB_t|}
\prod_{\ell=j+1}^{i-1}\frac{|\gR_t|-\ell}{|\gB_t|-\ell}
\le
\frac{|\gR_t|}{|\gB_t|},
\end{align*}
where the product is empty when \(i=j+1\). Therefore each blue-to-red summand is
at most
\[
\frac{|\gR_t|}{|\gB_t|}
\left(\frac{1-\alpha}{1-\beta}\right)^{\Delta_t}
\]
times the corresponding red-to-blue summand. Summing over \(i>j\), we obtain
\[
\frac{|\gB_t|p_t^{\gB}}{|\gR_t|p_t^{\gR}}
\le
\left(\frac{1-\alpha}{1-\beta}\right)^{\Delta_t}.
\]

It remains to express the last bound in terms of \(\theta_t\). Since
\(\alpha=a\log (N)/N\), \(\beta=b\log (N)/N\), and \(a>b\), for all sufficiently
large \(N\), we have
\begin{align*}
\log\frac{1-\alpha}{1-\beta}
&=
\log(1-\alpha)-\log(1-\beta)
\le
-(\alpha-\beta)
=
-(a-b)\frac{\log (N)}{N}.
\end{align*}
Set \(c_0\coloneqq1-e^{-(a-b)}\). If \(\Delta_t\log(N)/N\le1\), then
\(\theta_t=\Delta_t\log(N)/N\), and convexity of \(x\mapsto e^{-(a-b)x}\) on
\([0,1]\) gives
\[
    \left(\frac{1-\alpha}{1-\beta}\right)^{\Delta_t}
    \le
    \exp\left(-(a-b)\frac{\Delta_t\log(N)}{N}\right)
    \le
    1-c_0\theta_t.
\]
If \(\Delta_t\log(N)/N\ge1\), then \(\theta_t=1\), and
\[
    \left(\frac{1-\alpha}{1-\beta}\right)^{\Delta_t}
    \le
    \exp\left(
    -(a-b)\frac{\Delta_t\log(N)}{N}
    \right)
    \le
    e^{-(a-b)}
    =
    1-c_0\theta_t.
\]
The desired result follows by combining the last two cases.
\end{proof}

\subsection{Proof of Lemma~\ref{lem:one-step-difference-bound}}

\begin{proof}
By Lemma~\ref{lem:conditional-expectations-next-day},
\eqref{eqn:gBR-gRB-definition}, \eqref{eqn:p-q-B}, and \eqref{eqn:p-q-R},
\[
\mu_t
=2\left(|\gR_t|p_t^\gR-|\gB_t|p_t^\gB\right),
\qquad
\nu_t
=|\gR_t|p_t^\gR+|\gB_t|p_t^\gB.
\]
The same identities, \eqref{eqn:blue-size-from-advantage}, and
\eqref{eqn:red-size-from-advantage} give
\begin{align}
\widetilde{\mu}_t
&=(a+b)\left(|\gR_t|p_t^\gR-|\gB_t|p_t^\gB\right) \notag\\
&=\left(bN-\widetilde{\Delta}_t\right)p_t^\gR
-\left(aN+\widetilde{\Delta}_t\right)p_t^\gB.
\label{eq:first-day-weighted-drift}
\end{align}
Thus \(\widetilde{\mu}_t=(a+b)\mu_t/2\).

After decreasing the constant in Lemma~\ref{lem:flip-ratio-comparison}, if
necessary, that lemma gives
\[
|\gB_t|p_t^\gB
\le (1-c_0\theta_t)|\gR_t|p_t^\gR
\le |\gR_t|p_t^\gR
\]
for some \(0<c_0\le1\). Therefore
\[
|\gR_t|p_t^\gR-|\gB_t|p_t^\gB
\ge c_0|\gR_t|p_t^\gR\theta_t.
\]
Substitution into the identities for \(\mu_t\) and \(\nu_t\) proves
\eqref{eqn:mu-nu-tail-free-bound}; dividing the lower bound for \(\mu_t\) by
the upper bound for \(\nu_t\) proves \eqref{eqn:mu-nu-comparison-bound}.
\end{proof}

\subsection{Proof of Lemma~\ref{lem:linear-jump-one-step-bound}}

\begin{proof}
Choose a sufficiently small constant \(\eta_0=\eta_0(a,b)>0\). If
\(|\gR_t|\le\eta_0N\), then \(|\gB_t|\ge(1-\eta_0)N\), and
Lemma~\ref{lem:expected-variance-degree-differences} gives
\[
m_t^\gB
=-\big((a-b)N+\widetilde{\Delta}_t-a\big)\frac{\log N}{N}
\le-c\log N.
\]
Lemma~\ref{lem:gaussian-sparse-binomial-difference}, applied to
\(\rD_t^\gB\), therefore gives \(p_t^\gB=o(1)\). By
\eqref{eq:R-next-expectation},
\[
\E[|\gR_{t+1}|\mid\rvy_t]\le\eta_0N+o(N),
\]
and hence
\[
\E[\widetilde{\Delta}_{t+1}\mid\rvy_t]
=bN-(a+b)\E[|\gR_{t+1}|\mid\rvy_t]
\ge cN
\]
after decreasing \(\eta_0\).

Suppose instead that \(|\gR_t|>\eta_0N\). The assumed lower bound on
\(\widetilde{\Delta}_t\) and \eqref{eqn:blue-size-from-advantage} imply
\[
|\gB_t|\ge\frac{aN}{a+b}-o(N),
\qquad
|\gR_t|>\eta_0N.
\]
Thus both camps have linear size. Equation~\eqref{eqn:affine-delta-tilde-delta}
also gives \(\Delta_t\asymp N\), so \(\theta_t=1\). By
Lemma~\ref{lem:one-step-difference-bound},
\[
\widetilde{\mu}_t\ge c|\gR_t|p_t^\gR.
\]
Furthermore, Lemma~\ref{lem:expected-variance-degree-differences} gives
\[
m_t^\gR
=(\widetilde{\Delta}_t+a)\frac{\log N}{N}
\ge-H\sqrt{\log N}+o(\sqrt{\log N}),
\qquad
v_t^\gR\asymp\log N.
\]
Hence \(x_t^\gR\ge-C_H\), and
Corollary~\ref{cor:gaussian-flip-probabilities} yields
\(p_t^\gR\ge c_H\). Consequently,
\(\widetilde{\mu}_t\ge c_HN\). In both cases,
\(\E[\widetilde{\Delta}_{t+1}\mid\rvy_t]\ge c_HN\), which proves the
lower bound for \(\widetilde{\mu}_t\) in
\eqref{eqn:linear-jump-expectation-bound}. Its upper bound follows from the
deterministic inequality \(\widetilde{\Delta}_{t+1}\le bN\). The bounds for
\(\mu_t\) follow from \eqref{eqn:affine-delta-tilde-delta}, and
\(0\le\nu_t\le N\) follows because \(\nu_t\) is the expected number of
vertices that flip.

For the high-probability estimate, Lemma~\ref{lem:R1-read-two-concentration}
gives, for every \(s>0\),
\begin{align}
\P\left(
\left|\widetilde{\Delta}_{t+1}
-\E[\widetilde{\Delta}_{t+1}\mid\rvy_t]\right|\ge s
\,\middle|\,\rvy_t
\right)
\le2\exp\left(-\frac{s^2}{(a+b)^2N}\right).
\label{eq:first-day-gap-concentration}
\end{align}
Taking \(s=c_HN/2\), and decreasing \(c_H\), gives
\(\widetilde{\Delta}_{t+1}\ge c_HN\) except on an event of conditional
probability at most \(2\exp(-c_HN)\). The weighted upper bound is
deterministic, and the bounds for \(\Delta_{t+1}\) follow from
\eqref{eqn:affine-delta-tilde-delta}, after changing \(c_H\).
\end{proof}

\subsection{Proof of Lemma~\ref{lem:almost-linear-jump-one-step-bound}}

\begin{proof}
The definition of \(u_t\), \eqref{eqn:blue-size-from-advantage}, and
\eqref{eqn:red-size-from-advantage} give
\[
|\gB_t|
=\frac{aN}{a+b}-\frac{u_tN}{(a+b)\sqrt{\log N}},
\qquad
|\gR_t|
=\frac{bN}{a+b}+\frac{u_tN}{(a+b)\sqrt{\log N}}.
\]
Thus both camps have linear size. Equation~\eqref{eqn:affine-delta-tilde-delta}
gives \(\Delta_t\asymp N\), and hence \(\theta_t=1\). Moreover,
Lemma~\ref{lem:expected-variance-degree-differences} gives
\[
m_t^\gR=-u_t\sqrt{\log N}+O(\log N/N),
\qquad
v_t^\gR=\left(\frac{2ab}{a+b}+o(1)\right)\log N.
\]
Consequently, \(-x_t^\gR\asymp u_t\). By
Corollary~\ref{cor:moderate-deviation-tail-flip-probabilities},
\[
p_t^\gR
\asymp
\frac{N^{-I(A_t^\gR,B_t^\gB)}}{1+u_t}.
\]
Since
\(I(A_t^\gR,B_t^\gB)\log N=\Theta(u_t^2)\), this becomes
\[
p_t^\gR\asymp\frac{\exp\{-\Theta(u_t^2)\}}{1+u_t}.
\]
Substitution into Lemma~\ref{lem:one-step-difference-bound}, together with
\(|\gR_t|\asymp N\) and \(\theta_t=1\), proves
\eqref{eqn:mu-nu-moderate-bound}.

Apply \eqref{eq:first-day-gap-concentration} with
\[
s=cN\frac{\exp(-Cu_t^2)}{1+u_t},
\]
decreasing \(c\) and increasing \(C\) if necessary. On the resulting event, the
realized increment differs from its conditional mean by at most \(s\), so the
expectation bounds transfer with at most a fixed-factor loss. The complementary
event has conditional probability at most
\[
2\exp\left(-cN\frac{\exp(-Cu_t^2)}{(1+u_t)^2}\right).
\]
This proves the weighted bounds in \eqref{eq:almost-linear-jump-bound}. The
unweighted bounds follow from \eqref{eqn:affine-delta-tilde-delta}.
\end{proof}

\subsection{Proof of Lemma~\ref{lem:sublinear-jump-one-step-bound}}

\begin{proof}
The assumption \(\Delta_t\ge0\) gives \(|\gB_t|\ge|\gR_t|\), while
\eqref{eqn:red-size-from-advantage} and
\eqref{eqn:sublinear-jump-condition} give \(|\gR_t|\asymp N\). Hence both camps
have linear size. Corollary~\ref{cor:large-deviation-tail-flip-probabilities}
gives
\[
p_t^\gR
=
N^{-I(A_t^\gR,B_t^\gB)+o(1)}.
\]
The exponent differs from \(I_t\) only by replacing \(|\gR_t|-1\) with
\(|\gR_t|\). Since \(|\gR_t|\asymp N\),
\[
I(A_t^\gR,B_t^\gB)=I_t+O(N^{-1/2}).
\]
Consequently, for every fixed \(\eta>0\) and all sufficiently large \(N\),
\[
N^{-I_t-\eta}\le p_t^\gR\le N^{-I_t+\eta}.
\]
Substituting this estimate into
Lemma~\ref{lem:one-step-difference-bound}, and using
\(|\gR_t|\asymp N\), proves \eqref{eqn:mu-large-deviation-bound} and
\eqref{eqn:nu-large-deviation-bound}. The relation
\eqref{eqn:mu-tilde-mu-linearity} proves
\eqref{eqn:tilde-mu-large-deviation-bound}.

For the lower bound, apply
\eqref{eq:first-day-gap-concentration} with
\(s=c\theta_t N^{1-I_t-\eta}\), after decreasing \(c\). Its conditional
failure probability is at most
\[
2\exp\left(-c\theta_t^2N^{1-2I_t-2\eta}\right).
\]
For the upper bound, use
\(s=CN^{1-I_t+\eta}\), after increasing \(C\). The corresponding
failure probability is at most
\[
2\exp\left(-cN^{1-2I_t+2\eta}\right).
\]
Combining these events proves the weighted bounds in
\eqref{eq:sublinear-jump-bound}; the unweighted bounds follow from
\eqref{eqn:affine-delta-tilde-delta}.
\end{proof}

\section{Deferred Proofs in Section~\ref{sec:constant-days-to-unanimity}}\label{app:deferred-constant-days}

\subsection{Proof of Lemma~\ref{lem:third-day-extinction}}

\begin{proof}
Condition on \(\rvy_{t}\) and assume
\eqref{eqn:third-day-extinction-condition}. If
\(\gR_{t}=\emptyset\), then no blue vertex has a red neighbor, so
\(\gR_{t+1}=\emptyset\) deterministically. Hence we may assume
\(1\le |\gR_{t}|\le rN\). Then \(|\gB_{t}|\ge(1-r)N\).

Recall the quantities \(A_{t}^{\gR}\), \(B_{t}^{\gR}\),
\(A_{t}^{\gB}\), and \(B_{t}^{\gB}\) from
Corollary~\ref{cor:large-deviation-tail-flip-probabilities}. Under the current notation,
\[
A_{t}^{\gR}=\frac{a(|\gR_{t}|-1)}{N},\qquad
B_{t}^{\gR}=\frac{b|\gR_{t}|}{N},\qquad
A_{t}^{\gB}=\frac{a(|\gB_{t}|-1)}{N},\qquad
B_{t}^{\gB}=\frac{b|\gB_{t}|}{N}.
\]
By \eqref{eqn:D-Rt-distribution}, \eqref{eqn:D-Bt-distribution}, and the
Chernoff upper bounds in the proof of Lemma~\ref{lem:binomial-difference-large-deviation-tail},
\begin{align*}
q_t^{\gR}
&=
\P\left(\rD_t^{\gR}\le0\,\middle|\,\rvy_t\right)
\le
N^{-I(B_t^{\gB},A_t^{\gR})},\\
p_t^{\gB}
&=
\P\left(\rD_t^{\gB}>0\,\middle|\,\rvy_t\right)
\le
N^{-I(A_t^{\gB},B_t^{\gR})}.
\end{align*}
We next lower bound the two exponents. Since \(|\gR_{t}|\le rN\) and
\(|\gB_{t}|\ge(1-r)N\),
\[
B_t^{\gB}\ge b(1-r),
\qquad
A_t^{\gR}\le ar.
\]
Therefore, using the definition of \(I\) in \eqref{eqn:LDP-rate-function},
\[
I(B_t^{\gB},A_t^{\gR})
\ge
\left(\sqrt{b(1-r)}-\sqrt{ar}\right)^2.
\]
Because \(r<r_{*}\), define
\[
\eta\coloneqq
\left(\sqrt{b(1-r)}-\sqrt{ar}\right)^2-1>0.
\]
Thus \(q_t^{\gR}\le N^{-1-\eta}\).

For the blue-to-red probability, the same size bounds give
\[
A_t^{\gB}\ge a\left(1-r-\frac1N\right),
\qquad
B_t^{\gR}\le br.
\]
Since \(a\ge b\) and \(r<r_{*}<1/2\), we have
\begin{align*}
\left(\sqrt{a(1-r)}-\sqrt{br}\right)^2
-
\left(\sqrt{b(1-r)}-\sqrt{ar}\right)^2
=
(a-b)(1-2r)\ge 0.
\end{align*}
Consequently, for all sufficiently large \(N\),
\[
I(A_t^{\gB},B_t^{\gR})\ge 1+\frac{\eta}{2},
\]
and hence \(p_t^{\gB}\le N^{-1-\eta/2}\). Taking
\(\xi=\eta/2\), we have, for all sufficiently large \(N\),
\[
q_t^{\gR}\le N^{-1-\xi},
\qquad
p_t^{\gB}\le N^{-1-\xi}.
\]

Finally, a vertex is red after update \(t+1\) only if either it was red at time
\(t\) and stays red, or it was blue at time \(t\) and flips to red. By the union bound
and the definitions \eqref{eqn:p-q-R} and \eqref{eqn:p-q-B},
\begin{align*}
\P\left(\gR_{t+1}\neq\emptyset\,\middle|\,\rvy_{t}\right)
&\le
|\gR_{t}|q_t^{\gR}
+
|\gB_{t}|p_t^{\gB}\\
&\le
N\cdot N^{-1-\xi}
+
N\cdot N^{-1-\xi}
=
2N^{-\xi}.
\end{align*}
\end{proof}

\subsection{Proof of Lemma~\ref{lem:second-day-reduction}}

\begin{proof}
    Throughout the proof, we condition on \(\rvy_t\) satisfying
    \eqref{eqn:second-day-reduction-condition}, and take \(r\) as in
    \eqref{eqn:second-day-extinction-density-condition}. Define
    \[
    x_K\coloneqq K/\sqrt{2ab/(a+b)},
    \qquad
    \rho\coloneqq \frac{bq_K/(a+b)+r}{2}.
    \]
    By \eqref{eqn:second-day-extinction-density-condition},
    \begin{align*}
    \frac{b}{a+b}q_K<\rho<r<r_{*}<\frac12.
    \end{align*}
    We claim that, uniformly over all configurations satisfying
    \eqref{eqn:second-day-reduction-condition},
    \begin{align}
    \E\left[|\gR_{t+1}|\,\middle|\,\rvy_t\right]\le \rho N+o(N).
    \label{eq:second-day-mean-rho}
    \end{align}
    Assuming \eqref{eq:second-day-mean-rho}, the desired high-probability bound
    follows directly from Lemma~\ref{lem:R1-read-two-concentration}. Applying that
    lemma with \(u=\sqrt N\log(N)\) gives
    \begin{align*}
    \P\left(
    \left||\gR_{t+1}|-\E\left[|\gR_{t+1}|\,\middle|\,\rvy_t\right]\right|
    \ge \sqrt N\log (N)
    \,\middle|\,
    \rvy_t
    \right)
    \le
    2\exp\left(-(\log (N))^2\right).
    \end{align*}
    Since \(\rho<r\) and \(\sqrt N\log (N)=o(N)\),
    \eqref{eq:second-day-mean-rho} implies that \(|\gR_{t+1}|\le rN\) for all
    sufficiently large \(N\), except with conditional probability at most
    \(2\exp\left(-(\log (N))^2\right)\).
    
    It remains to prove \eqref{eq:second-day-mean-rho}. We divide the proof into two
    cases according to the size of \(\gR_t\).
    
    \noindent\textbf{Case 1: \(|\gR_t|\le \rho N\)}. In this case, the red camp is
    already below the intermediate density \(\rho\). Hence
    \(|\gB_t|\ge (1-\rho)N\). For a fixed blue vertex, conditional on \(\rvy_t\),
    \begin{align*}
    \rD_t^{\gB}
    \stackrel d=
    \Bin(|\gR_t|,\beta)-\Bin(|\gB_t|-1,\alpha).
    \end{align*}
    Following the notation of Corollary~\ref{cor:large-deviation-tail-flip-probabilities}, set
    \[
    A_t^{\gB}\coloneqq \frac{a(|\gB_t|-1)}{N},
    \qquad
    B_t^{\gR}\coloneqq \frac{b|\gR_t|}{N}.
    \]
    Applying the Chernoff calculation in the proof of
    Lemma~\ref{lem:binomial-difference-large-deviation-tail} to
    \(\rX\sim\Bin(|\gR_t|,\beta)\) and
    \(\rY\sim\Bin(|\gB_t|-1,\alpha)\), we obtain
    \begin{align*}
    p_t^{\gB}
    =
    \P\left(\rD_t^{\gB}>0\,\middle|\,\rvy_t\right)
    \le
    N^{-I(A_t^{\gB},B_t^{\gR})}.
    \end{align*}
    The exponent \(I(A_t^{\gB},B_t^{\gR})\) is bounded away from zero. Indeed,
    using \(|\gR_t|\le \rho N\), \(|\gB_t|\ge (1-\rho)N\), and
    \eqref{eqn:LDP-rate-function}, for all sufficiently large \(N\), we have
    \begin{align*}
    \sqrt{I(A_t^{\gB},B_t^{\gR})}
    &\ge
    \sqrt{a\left(1-\rho-N^{-1}\right)}
    -
    \sqrt{b\rho}
    \ge
    \sqrt b\left(\sqrt{1-\rho}-\sqrt\rho\right)+o(1) \geq c_1,
    \end{align*}
    where \(c_1>0\) is a positive constant because \(\rho<1/2\). Hence
    \(p_t^{\gB}=o(1)\). Therefore, \eqref{eq:second-day-mean-rho} follows from
    \eqref{eq:R-next-expectation} and the trivial bound \(q_t^{\gR}\le1\):
    \begin{align*}
    \E\left[|\gR_{t+1}|\,\middle|\,\rvy_t\right]
    =
    |\gR_t|q_t^{\gR}+|\gB_t|p_t^{\gB}
    \le
    |\gR_t|+Np_t^{\gB}
    \le
    \rho N+o(N).
    \end{align*}
    
    \noindent\textbf{Case 2: \(|\gR_t|>\rho N\)}. Set
    \(\gamma_N\coloneqq\widetilde{\Delta}_t/N\). The hypothesis
    \eqref{eqn:second-day-reduction-condition} gives
    \begin{align*}
    \gamma_N\sqrt{\log N}\ge K.
    \end{align*}
    In particular, \(\gamma_N\ge K/\sqrt{\log N}\). Together with
    \eqref{eqn:blue-size-from-advantage} and
    \eqref{eqn:red-size-from-advantage}, this gives
    \begin{align*}
    |\gR_t|
    =
    \frac{bN-\widetilde{\Delta}_t}{a+b}
    \le
    \left(
    \frac{b}{a+b}
    +o(1)
    \right)N,
    \qquad
    |\gB_t|
    =
    \frac{aN+\widetilde{\Delta}_t}{a+b}
    \ge
    \left(
    \frac{a}{a+b}
    +o(1)
    \right)N.
    \end{align*}
    Since \(a>b\) and \(|\gR_t|>\rho N\), both camps have linear size.
    Hence Corollary~\ref{cor:gaussian-flip-probabilities} applies uniformly in
    the present case.

    We first bound the probability \(q_t^{\gR}\) that a red vertex remains red. By
    \eqref{eqn:expected-degree-differences-R-B},
    \begin{align*}
    m_t^{\gR}
    =
    (\widetilde{\Delta}_t+a)\frac{\log N}{N}
    =
    \gamma_N\log N+\frac{a\log N}{N}.
    \end{align*}
    Moreover, the variance formula in the proof of
    Lemma~\ref{lem:expected-variance-degree-differences} gives
    \begin{align}
    \frac{v_t^{\gR}}{\log N}
    &=
    \frac{
    b|\gB_t|(1-\beta)
    +
    a(|\gR_t|-1)(1-\alpha)
    }{N} \notag\\
    &=
    \frac{2ab}{a+b}
    -
    \frac{a-b}{a+b}\gamma_N
    +
    O\left(\frac{\log N}{N}\right).
    \label{eq:second-day-red-variance-gamma}
    \end{align}

    If \(\gamma_N\ge0\), then
    \(v_t^{\gR}\le (2ab/(a+b))\log N\), and therefore
    \[
    \frac{m_t^{\gR}}{\sqrt{v_t^{\gR}}}
    \ge
    \frac{\gamma_N\sqrt{\log N}}{\sqrt{2ab/(a+b)}}
    \ge
    \frac{\max\{K,0\}}{\sqrt{2ab/(a+b)}}
    \ge x_K.
    \]
    If \(\gamma_N<0\), then necessarily \(K<0\) and
    \(K\le\gamma_N\sqrt{\log N}<0\). In this case,
    \eqref{eq:second-day-red-variance-gamma} yields
    \[
    v_t^{\gR}
    =
    \left(
    \frac{2ab}{a+b}
    +o(1)
    \right)\log N,
    \]
    uniformly over the present configurations. Consequently,
    \[
    \frac{m_t^{\gR}}{\sqrt{v_t^{\gR}}}
    =
    \frac{\gamma_N\sqrt{\log N}+o(1)}
    {\sqrt{2ab/(a+b)+o(1)}}
    \ge x_K+o(1).
    \]
    Thus, in both cases,
    \[
    \frac{m_t^{\gR}}{\sqrt{v_t^{\gR}}}\ge x_K+o(1).
    \]
    Applying \eqref{eq:gaussian-qR} from
    Corollary~\ref{cor:gaussian-flip-probabilities}, uniformly in the present
    case, gives
    \begin{align*}
    q_t^{\gR}
    &\le
    \Phi\Big(-m_t^{\gR}/\sqrt{v_t^{\gR}}\Big)
    +
    O(1/\sqrt{\log (N)})
    \le
    \Phi(-x_K+o(1)) + O(1/\sqrt{\log (N)})
    =
    q_K+o(1).
    \end{align*}
    
    We next bound \(p_t^{\gB}\), the probability that a blue vertex flips to red.
    By Corollary~\ref{cor:large-deviation-tail-flip-probabilities},
    \[
    p_t^{\gB}=N^{-I(A_t^{\gB},B_t^{\gR})+o(1)}.
    \]
    By the camp-size identities,
    \[
    A_t^{\gB}
    =
    \frac{a(a+\gamma_N)}{a+b}-\frac aN,
    \qquad
    B_t^{\gR}
    =
    \frac{b(b-\gamma_N)}{a+b}.
    \]
    The difference
    \(\sqrt{A_t^{\gB}}-\sqrt{B_t^{\gR}}\) is increasing in
    \(\gamma_N\). Since \(\gamma_N\ge K/\sqrt{\log N}\), it follows uniformly
    that
    \[
    \sqrt{A_t^{\gB}}-\sqrt{B_t^{\gR}}
    \ge
    \frac{a-b}{\sqrt{a+b}}+o(1).
    \]
    Hence \(I(A_t^{\gB},B_t^{\gR})\) is bounded away from zero, and thus
    \(p_t^{\gB}=o(1)\).
    Combining the bounds \(q_t^{\gR}\le q_K+o(1)\) and
    \(p_t^{\gB}=o(1)\) with \eqref{eq:R-next-expectation}, and using
    \(|\gR_t|\le (b/(a+b)+o(1))N\), we obtain
    \begin{align*}
    \E\left[|\gR_{t+1}|\,\middle|\,\rvy_t\right]
    &=
    |\gR_t|q_t^{\gR}+|\gB_t|p_t^{\gB}
    \le
    \left(\frac{b}{a+b}q_K+o(1)\right)N
    \le
    \rho N+o(N),
    \end{align*}
    where the last inequality uses \(bq_K/(a+b)<\rho\). This proves
    \eqref{eq:second-day-mean-rho} in Case 2.
    
    Therefore, \eqref{eq:second-day-mean-rho} holds in both cases, completing the
    proof.
    \end{proof}

\section{Deferred Proofs in Section~\ref{sec:polylogarithmic-days-to-unanimity}}\label{app:deferred-super-polylog}

\subsection{Proof of Lemma~\ref{lem:constant-window-entrance}}

\begin{proof}
    If \(\tau=0\), there is nothing to prove. By
    Lemma~\ref{lem:almost-linear-jump-one-step-bound}, there are constants
    \(c_0,C_0>0\), depending only on \(a,b\), such that the relevant one-step
    lower bound and its conditional failure probability are uniform over
    \(1<u_t\le h_N\), where we denote
    \[
    d_N\coloneqq
    c_0N\frac{\exp(-C_0h_N^2)}{1+h_N},
    \qquad
    q_N\coloneqq
    2\exp\left(
    -c_0N\frac{\exp(-C_0h_N^2)}{(1+h_N)^2}
    \right).
    \]
    
    Consider the process until it enters the constant-time window or the
    lower increment \(d_N\) fails for the first time. Before either event occurs,
    the weighted advantage is nondecreasing, so the hypothesis of
    Theorem~\ref{thm:super-polylog-unanimity} and the condition \(\tau>t\) give
    \(1<u_t\le h_N\). Moreover,
    \eqref{eqn:affine-delta-tilde-delta} and
    \(h_N=o(\sqrt{\log N})\) imply \(\Delta_t>0\) for all sufficiently large
    \(N\). Lemma~\ref{lem:almost-linear-jump-one-step-bound} therefore gives
    \[
    \P\left(
    \widetilde{\Delta}_{t+1}-\widetilde{\Delta}_t<d_N
    \,\middle|\,\mathcal F_t
    \right)
    \le q_N
    \]
    at every such time.
    
    Choose \(A>C_0\) sufficiently large. Then, with \(T_N\) as in the lemma
    statement,
    \[
    T_Nd_N
    \ge
    h_N\frac{N}{\sqrt{\log N}}
    \]
    for all sufficiently large \(N\). Consequently, if none of the first \(T_N\)
    lower-increment bounds fails, the process must enter the
    constant-time window by time \(T_N\). A union bound over the first
    failure time now gives the desired estimate. Indeed, for
    \(0\le t<T_N\), define
    \[
    \mathcal H_t
    \coloneqq
    \{\tau>t\}
    \cap
    \bigcap_{s=0}^{t-1}
    \left\{
    \widetilde{\Delta}_{s+1}-\widetilde{\Delta}_s\ge d_N
    \right\},
    \qquad
    \mathrm{BE}_t
    \coloneqq
    \mathcal H_t
    \cap
    \left\{
    \widetilde{\Delta}_{t+1}-\widetilde{\Delta}_t<d_N
    \right\}.
    \]
    Since \(\mathcal H_t\in\mathcal F_t\), the conditional estimate above
    implies
    \[
    \begin{aligned}
    \P(\mathrm{BE}_t)
    &=
    \E\left[
    \indi{\mathcal H_t}
    \P\left(
    \widetilde{\Delta}_{t+1}-\widetilde{\Delta}_t<d_N
    \,\middle|\,\mathcal F_t
    \right)
    \right] \le q_N\P(\mathcal H_t)
    \le q_N.
    \end{aligned}
    \]
    If \(\tau>T_N\), at least one of the events
    \(\mathrm{BE}_0,\ldots,\mathrm{BE}_{T_N-1}\) must occur. Therefore,
    \[
    \begin{aligned}
    \P(\tau>T_N)
    &\le
    \sum_{t=0}^{T_N-1}\P(\mathrm{BE}_t)
    \le T_Nq_N\\
    &=
    2\left\lceil\exp(Ah_N^2)\right\rceil
    \exp\left(
    -c_0N\frac{\exp(-C_0h_N^2)}{(1+h_N)^2}
    \right)\\
    &\le
    2\exp(Ah_N^2)
    \exp\left(
    -cN\frac{\exp(-C h_N^2)}{(1+h_N)^2}
    \right).
    \end{aligned}
    \]
    The last inequality holds for all sufficiently large \(N\) after decreasing
    \(c\) and increasing \(C\), which absorb the ceiling and fixed numerical
    factors. This proves the lemma.
    \end{proof}

    \subsection{Proof of Lemma~\ref{lem:constant-window-completion}}

    \begin{proof}
        On \(\{\tau<\infty\}\), the defining condition for \(\tau\) is the
        time-shifted form of \eqref{eqn:three-day-unanimity-condition} with \(H=1\).
        Conditional on \(\mathcal F_\tau\), apply
        Theorem~\ref{thm:three-day-unanimity}\textup{(i)} at time \(\tau\). The
        independence of future update graphs and time homogeneity give constants
        \(c=c(a,b)>0\) and \(\xi=\xi(a,b)>0\) such that
        \[
        \P\left(
        \gR_{\tau+3}\neq\emptyset
        \,\middle|\,
        \mathcal F_{\tau}
        \right)
        \le
        2\exp(-cN)+2\exp\left(-(\log N)^2\right)+2N^{-\xi}.
        \]
        Since \(2\exp(-cN)\le2\exp(-(\log N)^2)\) for all sufficiently large \(N\), the
        claimed bound follows.
        \end{proof}

\section{Deferred Proofs in Section~\ref{sec:polynomial-days-to-unanimity}}\label{app:deferred-polynomial-days}

\subsection{Proof of Lemma~\ref{lem:pre-super-polylog-one-step-bound}}

\begin{proof}
Fix \(t<\tau\), condition on \(\mathcal F_t\), and suppose that
\(\Delta_t>0\). By \eqref{eq:tau-def},
\[
\widetilde{\Delta}_t
<
-\frac{N}{\sqrt{\log\log N}}
<0.
\]
The condition \(\Delta_t>0\) gives \(|\gB_t|\ge N/2\), while
\eqref{eqn:red-size-from-advantage} gives
\[
|\gR_t|
=
\frac{bN-\widetilde{\Delta}_t}{a+b}
\ge
\frac{bN}{a+b}.
\]
Thus both camps have linear size, uniformly over all configurations under
consideration.

We first record the rate-form estimate for the red-to-blue flip probability.
By Lemma~\ref{lem:expected-variance-degree-differences},
\[
m_t^\gR
=
(\widetilde{\Delta}_t+a)\frac{\log N}{N},
\qquad
v_t^\gR\asymp\log N.
\]
Consequently, for every \(t<\tau\) and all sufficiently large \(N\),
\[
c\sqrt{\frac{\log N}{\log\log N}}
\le
-x_t^\gR
\le
C\sqrt{\log N}.
\]
In any subrange where
\(-x_t^\gR=o(\sqrt{\log N})\),
Corollary~\ref{cor:moderate-deviation-tail-flip-probabilities} and the sharper
Mills bounds in \eqref{eq:moderate-deviation-tail-general} give
\[
p_t^\gR
=
N^{-I(A_t^\gR,B_t^\gB)+o(1)}.
\]
Indeed, the Mills-ratio factor lies between powers of \(\log N\), and hence is
\(N^{o(1)}\). On the complementary logarithmic scale,
Corollary~\ref{cor:large-deviation-tail-flip-probabilities} gives the same
rate-form estimate directly. That corollary is uniform over the present
linear-camp configurations, so the resulting estimate is uniform for
all \(t<\tau\).

Replacing \(|\gR_t|-1\) by \(|\gR_t|\) changes the rate by
\(O(N^{-1/2})\), so
\[
I(A_t^\gR,B_t^\gB)=I_t+O(N^{-1/2}).
\]
Consequently, for every fixed \(\eta>0\) and all sufficiently large \(N\),
\[
N^{-I_t-\eta}
\le
p_t^\gR
\le
N^{-I_t+\eta}.
\]
Applying Lemma~\ref{lem:one-step-difference-bound} and using
\(|\gR_t|\asymp N\) yields
\[
\mu_t
\ge
c\theta_t N^{1-I_t-\eta},
\qquad
\nu_t
\le
CN^{1-I_t+\eta},
\qquad
\frac{\mu_t}{\nu_t}
\ge
c\theta_t.
\]
This proves \eqref{eqn:pre-super-polylog-rate-bounds}.
\end{proof}

\subsection{Proof of Lemma~\ref{lem:conditional-bernstein-mgf}}

\begin{proof}
Condition on \(\mathcal F_{t}\). By \eqref{eqn:gBR-gRB-definition}, each
red-to-blue flip increases \(\Delta_{t}\) by \(2\), while each blue-to-red flip
decreases \(\Delta_{t}\) by \(2\). Thus
\[
\Delta_{t+1}-\Delta_t=2(\gRB_t-\gBR_t).
\]
Since the conditional law of \(\gG_{t+1}\) given \(\mathcal F_t\) depends only
on \(\rvy_t\), taking conditional expectation with respect to \(\mathcal F_t\)
in this identity and using the definition of \(\mu_t\) in
\eqref{eqn:mu-tilde-mu-definition}, we have
\[
\mu_t
=
2\left[
\E(\gRB_t\mid\mathcal F_t)-\E(\gBR_t\mid\mathcal F_t)
\right].
\]
Consequently,
\[
\xi_{t}
=
2\left[\gRB_{t}-\E(\gRB_{t}\mid\mathcal F_{t})\right]
-
2\left[\gBR_{t}-\E(\gBR_{t}\mid\mathcal F_{t})\right].
\]
Therefore
\(-\lambda \xi_{t}
=2\lambda\left[\gBR_{t}-\E(\gBR_{t}\mid\mathcal F_{t})\right]
-2\lambda\left[\gRB_{t}-\E(\gRB_{t}\mid\mathcal F_{t})\right]\).

Conditional on \(\mathcal F_{t}\), the randomness comes only from the independent
edge indicators of \(\gG_{t+1}\). The flip indicators whose sum is \(\gRB_{t}\)
form a read-\(2\) family, because the update of each red vertex is determined
by the edge indicators incident to that vertex, and each edge can affect only
the two endpoint updates. The same argument applies to the flip indicators
whose sum is \(\gBR_{t}\). Therefore Lemma~\ref{lem:read-k-mgf}, with \(k=2\),
gives constants \(C>0\) and \(\theta_0>0\) such that for every
\(|\theta|\le\theta_0\),
\begin{subequations}
\begin{align}
    \E\left[
        \exp\left(
            \theta
            \left[
                \gRB_{t}-\E(\gRB_{t}\mid \mathcal F_{t})
            \right]
        \right)
        \,\middle|\,
        \mathcal F_{t}
    \right]
    &\le
    \exp\left(
        C\theta^2\E(\gRB_{t}\mid \mathcal F_{t})
    \right),\\
    \E\left[
        \exp\left(
            \theta
            \left[
                \gBR_{t}-\E(\gBR_{t}\mid \mathcal F_{t})
            \right]
        \right)
        \,\middle|\,
        \mathcal F_{t}
    \right]
    &\le
    \exp\left(
        C\theta^2\E(\gBR_{t}\mid \mathcal F_{t})
    \right).
\end{align}
\end{subequations}

Now apply Cauchy--Schwarz:
\[
\begin{aligned}
    &\,\E\left[
        \exp\left(
            -\lambda \xi_{t}
        \right)
        \,\middle|\,
        \mathcal F_{t}
    \right]                                                   \\
    =&\,
    \E\left[
        \exp\left(
            2\lambda
            \left[
                \gBR_{t}-\E(\gBR_{t}\mid \mathcal F_{t})
            \right]
        \right)
        \exp\left(
            -2\lambda
            \left[
                \gRB_{t}-\E(\gRB_{t}\mid \mathcal F_{t})
            \right]
        \right)
        \,\middle|\,
        \mathcal F_{t}
    \right]                                                   \\
    \leq &\,
    \left(
        \E\left[
            \exp\left(
                4\lambda
                \left[
                    \gBR_{t}-\E(\gBR_{t}\mid \mathcal F_{t})
                \right]
            \right)
            \,\middle|\,
            \mathcal F_{t}
        \right]
    \right)^{1/2}
    \left(
        \E\left[
            \exp\left(
                -4\lambda
                \left[
                    \gRB_{t}-\E(\gRB_{t}\mid \mathcal F_{t})
                \right]
            \right)
            \,\middle|\,
            \mathcal F_{t}
        \right]
    \right)^{1/2}.
\end{aligned}
\]
Choose \(\lambda_0\coloneqq\theta_0/4\). Then for
\(0\le \lambda\le \lambda_0\), the two MGF bounds above apply with
\(\theta=4\lambda\) and \(\theta=-4\lambda\). Thus
\[
\begin{aligned}
    \E\left[
        \exp\left(
            -\lambda \xi_{t}
        \right)
        \,\middle|\,
        \mathcal F_{t}
    \right]
    &\le
    \exp\left(
        C\lambda^2
        \E(\gBR_{t}\mid \mathcal F_{t})
    \right)
    \exp\left(
        C\lambda^2
        \E(\gRB_{t}\mid \mathcal F_{t})
    \right)                                                   \\
    &=
    \exp\left(
        C\lambda^2
        \left[
            \E(\gRB_{t}\mid \mathcal F_{t})
            +
            \E(\gBR_{t}\mid \mathcal F_{t})
        \right]
    \right).
\end{aligned}
\]
Since the conditional law of \(\gG_{t+1}\) given \(\mathcal F_{t}\) depends only
on \(\rvy_{t}\), the definition of \(\nu_{t}\) gives
\(\nu_{t}=\E[\gRB_{t}+\gBR_{t}\mid\mathcal F_{t}]\). We obtain
\(\E[\exp(-\lambda \xi_{t})\mid\mathcal F_{t}]\le\exp(K\lambda^2\nu_{t})\)
after renaming the constant. This proves the lemma.
\end{proof}

\subsection{Proof of Lemma~\ref{lem:supermartingale}}

\begin{proof}
Fix \(\ell\ge0\). For \(0\le m\le \ell-1\), the random variable
\(\xi_{s+m}\) is \(\mathcal F_{s+m+1}\)-measurable and \(\nu_{s+m}\) is
\(\mathcal F_{s+m}\)-measurable. Hence \(M_{s+\ell}\),
\(\sum_{m=0}^{\ell-1}\nu_{s+m}\), and \(W_{s+\ell}\) are
\(\mathcal F_{s+\ell}\)-measurable. Also \(W_{s+\ell}\ge0\).

It remains to verify the one-step supermartingale inequality. For every
\(\ell\ge0\),
\[
M_{s+\ell+1}=M_{s+\ell}+\xi_{s+\ell},
\qquad
\sum_{m=0}^{\ell}\nu_{s+m}
=
\sum_{m=0}^{\ell-1}\nu_{s+m}+\nu_{s+\ell}.
\]
Therefore
\[
W_{s+\ell+1}
=
W_{s+\ell}
\exp\left(
-\lambda\xi_{s+\ell}
-K\lambda^2\nu_{s+\ell}
\right).
\]
Since \(W_{s+\ell}\) and \(\nu_{s+\ell}\) are
\(\mathcal F_{s+\ell}\)-measurable, conditioning on \(\mathcal F_{s+\ell}\)
gives
\[
\begin{aligned}
    \E\left[
        W_{s+\ell+1}
        \,\middle|\,
        \mathcal F_{s+\ell}
    \right]
    &=
    W_{s+\ell}
    \exp\left(-K\lambda^2\nu_{s+\ell}\right)
    \E\left[
        \exp\left(
            -\lambda \xi_{s+\ell}
        \right)
        \,\middle|\,
        \mathcal F_{s+\ell}
    \right].
\end{aligned}
\]
The choice \(0\le\lambda\le\lambda_0\) allows us to apply
Lemma~\ref{lem:conditional-bernstein-mgf} at time \(s+\ell\). Thus
\[
\E\left[
\exp\left(-\lambda \xi_{s+\ell}\right)
\,\middle|\,
\mathcal F_{s+\ell}
\right]
\le
\exp\left(K\lambda^2\nu_{s+\ell}\right).
\]
Substituting this estimate in the previous display yields
\[
\E\left[
W_{s+\ell+1}
\,\middle|\,
\mathcal F_{s+\ell}
\right]
\le
W_{s+\ell}.
\]
Hence \((W_{s+\ell})_{\ell\ge0}\) is a nonnegative supermartingale with respect to
\((\mathcal F_{s+\ell})_{\ell\ge0}\).
\end{proof}

\subsection{Proof of Lemma~\ref{lem:no-doubling-no-lower-exit}}

\begin{proof}
Let \(E\coloneqq\{\tau_s^+>L,\,\tau_s^->L\}\). By
\eqref{eqn:tau-s-plus} and \eqref{eqn:tau-s-minus}, on \(E\), for every
integer \(0\le m\le L\),
\[
\chi\Delta_s<\Delta_{s+m}<2\Delta_s,
\qquad
s+m<\tau.
\]
Therefore both color classes have linear size and
\(I_{s+m}\le \overline{I}_{s}\). Moreover, by
\eqref{eqn:theta-definition},
\[
\theta_{s+m}\ge c\theta_s.
\]
By Lemma~\ref{lem:pre-super-polylog-one-step-bound}, applied with parameter
\(\eta\),
\[
\mu_{s+m}
\ge
c\theta_{s+m}N^{1-I_{s+m}-\eta}
\ge
c\theta_s N^{1-\overline{I}_{s}-\eta}.
\]
Hence, on \(E\),
\[
\sum_{m=0}^{L-1}\mu_{s+m}
\ge
cL\theta_s N^{1-\overline{I}_{s}-\eta}
\ge
cA_0\Delta_s.
\]
Here the final inequality follows from the definition of \(L\) in
\eqref{eqn:L-definition}.
Choose \(A_0>0\) sufficiently large so that, on \(E\),
\(\sum_{m=0}^{L-1}\mu_{s+m}\ge10\Delta_s\).

Choose \(\lambda\coloneqq c_0\theta_{s}\), where \(c_0>0\) is small enough that
\(\lambda\le\lambda_0\) and
\(K\lambda^2\nu_{t}\le(\lambda/2)\mu_{t}\)
whenever the event \(E\) occurs and \(s\le t<s+L\). This follows from
Lemma~\ref{lem:pre-super-polylog-one-step-bound}, because on \(E\), again by
\eqref{eqn:theta-definition}, we have \(\theta_t\ge c\theta_s\), and hence
\(\mu_{t}\ge c\theta_{s}\nu_{t}\).
By Lemma~\ref{lem:supermartingale}, \((W_{s+\ell})_{\ell\ge0}\) is a
nonnegative supermartingale. Since \(L\) is deterministic after conditioning
on \(\mathcal F_{s}\), \(\E[W_{s+L}\mid\mathcal F_{s}]\le W_s=1\).

On \(E\), since \(\tau_s^+>L\), \eqref{eqn:tau-s-plus} gives
\(\Delta_{s+L}<2\Delta_s\). Hence
\(\Delta_{s+L}-\Delta_s<\Delta_s\).
Thus
\[
\begin{aligned}
    -\lambda M_{s+L}
    -
    K\lambda^2
    \sum_{m=0}^{L-1}\nu_{s+m}
    &=
    -\lambda
    \left[
        \Delta_{s+L}-\Delta_{s}
        -
        \sum_{m=0}^{L-1}\mu_{s+m}
    \right]
    -
    K\lambda^2
    \sum_{m=0}^{L-1}\nu_{s+m} \\
    &\ge
    -\lambda \Delta_s
    +
    \lambda
    \sum_{m=0}^{L-1}\mu_{s+m}
    -
    \frac{\lambda}{2}
    \sum_{m=0}^{L-1}\mu_{s+m} \\
    &=
    -\lambda \Delta_s
    +
    \frac{\lambda}{2}
    \sum_{m=0}^{L-1}\mu_{s+m} \\
    &\ge
    4\lambda \Delta_s.
\end{aligned}
\]
Therefore, on \(E\), \(W_{s+L}\ge \exp(4\lambda \Delta_s)\). Since
\(\E\left[W_{s+L}\,\middle|\,\mathcal F_{s}\right]\le 1\),
we get \(\P(E\mid\mathcal F_{s})\le\exp(-4\lambda \Delta_s)\). Finally,
\(\lambda=c_0\theta_{s}\), so
\(\P(\tau_s^+>L,\,\tau_s^->L\mid\mathcal F_{s})\le\exp(-c\Delta_s\theta_{s})\).
This proves the lemma.
\end{proof}

\subsection{Proof of Lemma~\ref{lem:lower-exit-estimate}}

\begin{proof}
Condition on \(\mathcal F_s\). Then \(L\) is fixed. Let
\[
T\coloneqq\tau_s^-\wedge \tau_s^+\wedge L.
\]
For every integer \(m<T\), the definitions
\eqref{eqn:tau-s-plus} and \eqref{eqn:tau-s-minus} give
\[
\chi\Delta_s<\Delta_{s+m}<2\Delta_s,
\qquad
s+m<\tau.
\]
In particular, both color classes have linear size along this time interval.
Moreover, by \eqref{eqn:theta-definition},
\[
\theta_{s+m}
=
\min\left\{1,\frac{\Delta_{s+m}\log(N)}{N}\right\}
\ge
    c\theta_s.
\]
Applying Lemma~\ref{lem:pre-super-polylog-one-step-bound} yields
\[
\mu_{s+m}
\ge
c\theta_{s+m}\nu_{s+m}
\ge
c\theta_s\nu_{s+m},
\qquad 0\le m<T.
\]

Choose \(\lambda=c_0\theta_s\), with \(c_0>0\) sufficiently small. Since
\(\theta_s\le1\), we may take \(c_0\le\lambda_0\), so \(\lambda\le\lambda_0\).
After decreasing \(c_0\) further if necessary, the last display implies that,
for every \(m<T\),
\[
K\lambda^2\nu_{s+m}
\le
\frac{\lambda}{2}\mu_{s+m}.
\]
By Lemma~\ref{lem:supermartingale}, \((W_{s+\ell})_{\ell\ge0}\) is a
nonnegative supermartingale. Since \(T\le L\), Doob's optional stopping theorem
gives
\[
\E[W_{s+T}\mid\mathcal F_s]\le W_s=1.
\]

Let \(E_-\coloneqq\{\tau_s^-\le L,\ \tau_s^-<\tau_s^+\}\). On \(E_-\), we have
\(T=\tau_s^-\), and therefore
\[
\Delta_{s+T}-\Delta_s\le -(1-\chi)\Delta_s.
\]
Using the definition of \(M_{s+T}\) in \eqref{eqn:Wsl-Msl} and the definition
of \(\xi_t\) in \eqref{eqn:xi-definition}, we also have
\[
M_{s+T}
=
\sum_{m=0}^{T-1}\xi_{s+m}
=
\Delta_{s+T}-\Delta_s-\sum_{m=0}^{T-1}\mu_{s+m}.
\]
Thus, on \(E_-\),
\[
M_{s+T}
\le
-(1-\chi)\Delta_s-\sum_{m=0}^{T-1}\mu_{s+m}.
\]
Combining this estimate with the choice of \(\lambda\), we get
\[
\begin{aligned}
    -\lambda M_{s+T}
    -
    K\lambda^2
    \sum_{m=0}^{T-1}\nu_{s+m}
    &\ge
    \lambda(1-\chi)\Delta_s
    +
    \lambda
    \sum_{m=0}^{T-1}\mu_{s+m}
    -
    K\lambda^2
    \sum_{m=0}^{T-1}\nu_{s+m} \\
    &\ge
    \lambda(1-\chi)\Delta_s
    +
    \frac{\lambda}{2}
    \sum_{m=0}^{T-1}\mu_{s+m}
    \ge
    \lambda(1-\chi)\Delta_s.
\end{aligned}
\]
Therefore \(W_{s+T}\ge\exp(\lambda(1-\chi)\Delta_s)\) on \(E_-\). Since
\(W_{s+T}\ge0\),
\[
1
\ge
\E[W_{s+T}\mid\mathcal F_s]
\ge
\exp(\lambda(1-\chi)\Delta_s)\P(E_-\mid\mathcal F_s).
\]
Hence
\(\P(E_-\mid\mathcal F_s)
\le\exp(-\lambda(1-\chi)\Delta_s)\). Recalling that
\(\lambda=c_0\theta_s\), and decreasing the constant \(c>0\), gives
\[
\P(\tau_s^-\le L,\,\tau_s^-<\tau_s^+\mid\mathcal F_s)
\le
\exp(-c\Delta_s\theta_s).
\]
\end{proof}

\subsection{Proof of Lemma~\ref{lem:block-doubling}}

\begin{proof}
Condition on \(\mathcal F_{s}\). Since \(s<\tau\), the stopping times
\(\tau_s^+\) and \(\tau_s^-\) are defined from the current configuration.

We use the decomposition
\(\{\tau_s^+>L\}\subseteq
\{\tau_s^+>L,\tau_s^->L\}\cup\{\tau_s^-\le L,\tau_s^-<\tau_s^+\}\).
By Lemma~\ref{lem:no-doubling-no-lower-exit}, with \(L\) defined in
\eqref{eqn:L-definition}, we have
\begin{align}
    \P\left(
        \tau_s^+>L,\tau_s^->L
        \,\middle|\,
        \mathcal F_{s}
    \right)
    &\le
    \exp\left(
        -c \Delta_s\theta_{s}
    \right),\\
\intertext{On the other hand, by Lemma~\ref{lem:lower-exit-estimate},}
    \P\left(
        \tau_s^-\le L,\tau_s^-<\tau_s^+
        \,\middle|\,
        \mathcal F_{s}
    \right)
    &\le
    \exp\left(
        -c \Delta_s\theta_{s}
    \right).
\end{align}
The desired bound follows from the union bound.
\end{proof}

\section{Deferred Proofs in Section~\ref{sec:necessary-conditions}}
\label{app:deferred-polynomial-lower-bound}

\subsection{Proof of Lemma~\ref{lem:local-exponent-control}}

\begin{proof}
Recall that
\[
x_0=\frac{\Delta_0}{N},
\qquad
x_\kappa
=
\frac{a/b-\kappa-1}{a/b-\kappa+1},
\qquad
x_*=\frac{a-b}{a+b}.
\]
The assumption \(1<\rho_0\le a/b-\kappa\) gives
\(0<x_0\le x_\kappa<x_*\). Set
\[
d_\kappa
\coloneqq
\frac{x_*-x_\kappa}{2}>0,
\qquad
\mathcal J_\kappa
\coloneqq
[0,x_\kappa+d_\kappa].
\]
The interval \(\mathcal J_\kappa\) is a compact subset of \((-1,x_*)\). For
every \(-1<x<x_*\), the function \(\mathcal I\) is continuously
differentiable and
\[
\mathcal I'(x)
=
-\left(
\sqrt{\frac{a(1-x)}{2}}
-
\sqrt{\frac{b(1+x)}{2}}
\right)
\left(
\sqrt{\frac{a}{2(1-x)}}
+
\sqrt{\frac{b}{2(1+x)}}
\right)
<0.
\]
Hence \(\mathcal I\) is decreasing on \((-1,x_*)\) and, by continuity, on
\([-1,x_*)\); it is Lipschitz on \(\mathcal J_\kappa\). Let
\(L_\kappa>0\) be a Lipschitz constant there, and
write \(\eta=\varepsilon/2\). Choose
\[
\sigma
\coloneqq
\min\left\{
d_\kappa,
\frac{\eta}{L_\kappa},
\frac{1-x_\kappa}{2}
\right\}.
\]
This constant depends only on \(a,b,\kappa,\varepsilon\), and
\(0<\sigma<1-x_0\). Moreover,
\[
x_0+\sigma
\le
x_\kappa+d_\kappa
<x_*,
\qquad
x_0+\sigma<1.
\]
In particular, \(\Delta_0+\sigma N<N\).

Suppose now that \(\Delta_t\le\Delta_0+\sigma N\), and put
\(x=\Delta_t/N\). By the camp-size identities,
\[
I_t
=
I\left(\frac{a(1-x)}{2},\frac{b(1+x)}{2}\right)
=
\mathcal I(x),
\]
where the last equality holds because \(x\le x_0+\sigma<x_*\). Since
\(\mathcal I\) is decreasing,
\[
I_t
\ge
\mathcal I(x_0+\sigma)
\ge
\mathcal I(x_0)-L_\kappa\sigma
\ge
I_0-\eta.
\]
This proves the lemma.
\end{proof}

\subsection{Proof of Lemma~\ref{lem:red-to-blue-flip-count-before-zeta}}

\begin{proof}
Fix \(0\le t<T\), and condition on \(\mathcal F_t\). On the event
\(\{t<\zeta\}\),
\[
\frac{\Delta_t}{N}
<
x_0+\sigma
\le
x_\kappa+d_\kappa
<x_*.
\]
Consequently,
\[
\frac{a|\gR_t|}{N}
>
\frac{b|\gB_t|}{N},
\qquad
|\gR_t|
=
\frac{N-\Delta_t}{2}
\ge
cN
\]
for a constant \(c=c(a,b,\kappa,\varepsilon)>0\). Recall that the exponent in
the exact red-to-blue flip probability uses
\[
A_t^\gR
=
\frac{a|\gR_t|}{N}-\frac{a}{N},
\qquad
B_t^\gB
=
\frac{b|\gB_t|}{N}.
\]
The strict separation from \(x_*\) ensures that
\(A_t^\gR>B_t^\gB\) for all sufficiently large \(N\). For fixed \(v\ge0\),
the rate function satisfies
\[
\frac{\partial}{\partial u}I(u,v)
=
1-\sqrt{\frac vu}
\in[0,1]
\qquad
\text{whenever }u>v.
\]
The mean-value theorem therefore gives
\[
I(A_t^\gR,B_t^\gB)
\ge
I\left(
\frac{a|\gR_t|}{N},
\frac{b|\gB_t|}{N}
\right)
-
\frac aN
=
I_t-\frac aN.
\]

Applying the exact Chernoff bound
\eqref{eq:binomial-difference-chernoff-upper} to
\eqref{eqn:D-Rt-distribution}, we obtain
\[
p_t^\gR
\le
N^{-I(A_t^\gR,B_t^\gB)}
\le
N^{-I_t+a/N}
\le
CN^{-I_t}.
\]
It follows from the definition of \(\gRB_t\) that, on \(\{t<\zeta\}\),
\[
\E[\gRB_t\mid\mathcal F_t]
=
|\gR_t|p_t^\gR
\le
CN^{1-I_t}.
\]
By Lemma~\ref{lem:local-exponent-control},
\(I_t\ge I_0-\varepsilon/2\) on this event. Since
\(\{t<\zeta\}\in\mathcal F_t\), the tower property yields
\[
\begin{aligned}
\E\left[\gRB_t\indi{t<\zeta}\right]
&=
\E\left[
\indi{t<\zeta}
\E[\gRB_t\mid\mathcal F_t]
\right]\\
&\le
CN^{1-I_0+\varepsilon/2}.
\end{aligned}
\]
Summing over \(0\le t<T\) and using
\(T\le N^{I_0-\varepsilon}\), we conclude that
\[
\begin{aligned}
\E\left[
\sum_{t=0}^{T-1}\gRB_t\indi{t<\zeta}
\right]
&\le
CTN^{1-I_0+\varepsilon/2}\\
&\le
CN^{1-\varepsilon/2}.
\end{aligned}
\]
This proves the lemma.
\end{proof}

\subsection{Proof of Lemma~\ref{lem:zeta-is-late}}

\begin{proof}
On the event \(\{\zeta\le T\}\), the definition of \(\zeta\) gives
\(\Delta_\zeta-\Delta_0\ge\sigma N\). Since every red-to-blue flip increases
the unweighted advantage by \(2\), while every blue-to-red flip decreases it
by \(2\),
\[
\begin{aligned}
\sigma N
&\le
\Delta_\zeta-\Delta_0\\
&=
2\sum_{t=0}^{\zeta-1}(\gRB_t-\gBR_t)\\
&\le
2\sum_{t=0}^{T-1}\gRB_t\indi{t<\zeta}.
\end{aligned}
\]
Therefore,
\[
\{\zeta\le T\}
\subseteq
\left\{
\sum_{t=0}^{T-1}\gRB_t\indi{t<\zeta}
\ge
\frac{\sigma N}{2}
\right\}.
\]
Markov's inequality and
Lemma~\ref{lem:red-to-blue-flip-count-before-zeta} now give
\[
\begin{aligned}
\P(\zeta\le T)
&\le
\frac{2}{\sigma N}
\E\left[
\sum_{t=0}^{T-1}\gRB_t\indi{t<\zeta}
\right]\\
&\le
CN^{-\varepsilon/2}.
\end{aligned}
\]
The constant may depend on \(a,b,\kappa,\varepsilon\), through the choice of
\(\sigma\). This proves the lemma.
\end{proof}

\section{Connectivity of Binary Stochastic Block Models}

\begin{lemma}[Connectivity of Binary Stochastic Block Models]\label{lem:connectivity-sparse-sbm}
Let \(\gG\sim\mathrm{SBM}(N,\alpha,\beta)\) with vertex partition
\(\gV=\gB\cup\gR\), where \(|\gB|=N_{\gB}\), \(|\gR|=N_{\gR}\),
\(N_{\gB}+N_{\gR}=N\), and \(N_{\gB},N_{\gR}\asymp N\). Suppose
\[
\alpha=\frac{a\log N+c_1}{N},
\qquad
\beta=\frac{b\log N+c_2}{N},
\]
where \(a>b>0\) are fixed constants and \(0\le c_1,c_2\le \log N\). Define
\[
\lambda_{\gB,N}\coloneqq \frac{aN_{\gB}+bN_{\gR}}{N},
\qquad
\lambda_{\gR,N}\coloneqq \frac{bN_{\gB}+aN_{\gR}}{N},
\qquad
\lambda_{*,N}\coloneqq\min\{\lambda_{\gB,N},\lambda_{\gR,N}\}.
\]
If \(\lambda_{*,N}\ge1\) for all sufficiently large \(N\), then
\[
\P(\gG\text{ is disconnected})
\le
E_{\gB,N}+E_{\gR,N}+O\left(\frac{\log N}{N}\right),
\]
where
\begin{subequations}
\begin{align}
E_{\gB,N}
&\coloneqq
\exp\left(
(1-\lambda_{\gB,N})\log N
-\frac{N_{\gB}c_1+N_{\gR}c_2}{N}
+o(1)
\right),\\
E_{\gR,N}
&\coloneqq
\exp\left(
(1-\lambda_{\gR,N})\log N
-\frac{N_{\gR}c_1+N_{\gB}c_2}{N}
+o(1)
\right).
\end{align}
\end{subequations}
Consequently, \(\gG\) is connected with high probability if
\(\liminf_{N\to\infty}\lambda_{*,N}>1\), or if \(\lambda_{*,N}\ge1\) for all
sufficiently large \(N\) and \(\min\{c_1,c_2\}\to\infty\). If
\(N_{\gB}/N_{\gR}\to\gamma\ge1\), the leading-order threshold is
\[
\min\left\{\frac{\gamma a+b}{\gamma+1},
\frac{\gamma b+a}{\gamma+1}\right\}\ge1.
\]
\end{lemma}
\begin{proof}[Proof of Lemma~\ref{lem:connectivity-sparse-sbm}]
Let \(\rX_{s,t}\) be the number of connected components of \(\gG\) with
exactly \(s\) vertices in \(\gB\) and \(t\) vertices in \(\gR\). If \(\gG\)
is disconnected, then it has a connected component with at most \(N/2\)
vertices. Hence
\begin{align*}
\P(\gG\text{ is not connected})
&\le
\E\rX_{1,0}+\E\rX_{0,1}
+\sum_{2\le s+t\le N/2}\E\rX_{s,t}.
\end{align*}

We first estimate the isolated vertices. For a blue vertex to be isolated,
all \(N_{\gB}-1\) possible same-color edges and all \(N_{\gR}\) possible
cross-color edges must be absent. Therefore
\begin{align*}
\E\rX_{1,0}
&=
N_{\gB}(1-\alpha)^{N_{\gB}-1}(1-\beta)^{N_{\gR}}\\
&\le
\exp\left\{
\log N
-\frac{(N_{\gB}-1)(a\log N+c_1)+N_{\gR}(b\log N+c_2)}{N}
\right\}\\
&\le
\exp\left(
(1-\lambda_{\gB,N})\log N
-\frac{N_{\gB}c_1+N_{\gR}c_2}{N}
+o(1)
\right),
\end{align*}
where we used \(N_{\gB}\le N\), \((a\log N+c_1)/N=o(1)\), and
\(\log(1-x)\le -x\). Similarly,
\[
\E\rX_{0,1}
\le
\exp\left(
(1-\lambda_{\gR,N})\log N
-\frac{N_{\gR}c_1+N_{\gB}c_2}{N}
+o(1)
\right).
\]

It remains to rule out non-isolated connected components. Fix
\(s,t\ge0\), put \(k=s+t\), and assume \(2\le k\le N/2\). Let
\[
\omega(s,t)\coloneqq
\frac{
a\{s(N_{\gB}-s)+t(N_{\gR}-t)\}
+b\{s(N_{\gR}-t)+t(N_{\gB}-s)\}
}{N}.
\]
The event counted by \(\rX_{s,t}\) requires the chosen \(k\) vertices to
contain a spanning tree and to have no edges to their complement. Since
\(\max\{\alpha,\beta\}\le C\log N/N\) for a constant \(C=C(a,b)\), Cayley's
bound and \(1-x\le e^{-x}\) give
\begin{align*}
\E\rX_{s,t}
&\le
\binom{N_{\gB}}{s}\binom{N_{\gR}}{t}
k^{k-2}
\left(\frac{C\log N}{N}\right)^{k-1}
\exp\{-\omega(s,t)\log N\}.
\end{align*}
Using
\(\binom{N_{\gB}}{s}\binom{N_{\gR}}{t}\le \binom Nk\le (eN/k)^k\), we obtain
\[
\log \E\rX_{s,t}
\le
\log N+Ck\log\log N-\omega(s,t)\log N.
\]

The threshold condition controls the small components. Indeed,
\[
\omega(s,t)
=s\lambda_{\gB,N}+t\lambda_{\gR,N}
-\frac{a(s^2+t^2)+2bst}{N}
\ge
k\lambda_{*,N}-O(k^2/N).
\]
Thus, for each fixed \(K\), if \(2\le k\le K\), then
\[
\E\rX_{s,t}
\le
C_K(\log N)^{k-1}N^{1-k}.
\]
After summing over the finitely many pairs \((s,t)\) with \(2\le s+t\le K\),
we obtain
\[
\sum_{2\le s+t\le K}\E\rX_{s,t}
=O\left(\frac{\log N}{N}\right).
\]

For the remaining component sizes, set \(\delta=\min\{a,b\}>0\). Since every
edge from a candidate component to its complement has logarithmic coefficient
at least \(\delta\),
\[
\omega(s,t)\ge \frac{\delta k(N-k)}{N}.
\]
Choose \(\eta>0\) small enough that the exact expansion above gives
\(\omega(s,t)\ge k/2\) whenever \(k\le\eta N\). Then, after increasing the
fixed constant \(K\), every \(K<k\le\eta N\) satisfies
\(\log\E\rX_{s,t}\le -3\log N\) for all sufficiently large \(N\). If
\(\eta N<k\le N/2\), then \(\omega(s,t)\ge\delta\eta N/2\), and the same
displayed bound again gives \(\log\E\rX_{s,t}\le -3\log N\). There are at
most \(N^2\) possible pairs \((s,t)\), so the contribution of all components
with \(K<s+t\le N/2\) is \(O(N^{-1})\). Combining this estimate with the two
isolated-vertex estimates proves the stated disconnected-probability bound.

If \(\liminf_{N\to\infty}\lambda_{*,N}>1\), then the two isolated-vertex
terms tend to zero even with \(c_1=c_2=0\). If \(\lambda_{*,N}\ge1\) and
\(\min\{c_1,c_2\}\to\infty\), then the isolated-vertex terms tend to zero
because \(N_{\gB},N_{\gR}\asymp N\). The final threshold formula follows by
taking the limit in \(\lambda_{\gB,N}\) and \(\lambda_{\gR,N}\) when
\(N_{\gB}/N_{\gR}\to\gamma\).
\end{proof}

\section{Technical Lemmas}

\begin{lemma}[Bennett's inequality, {\cite[Theorem $2.9.2$]{vershynin2018high} }]\label{lem:Bennett}
Let $\rX_1,\dots, \rX_n$ be independent random variables. Assume that
$|\rX_i-\E \rX_i|\le K$ almost surely for every \(i\), and put
$\sigma^2=\sum_{i=1}^{n}\Var(\rX_i)$ and
$h(u)\coloneqq(1+u)\log(1+u)-u$. Then, for any \(t>0\),
\begin{subequations}
\begin{align}
\P \Bigg( \sum_{i=1}^{n} (\rX_i-\E \rX_i) \geq t \Bigg)
&\leq
\exp \Bigg( - \frac{\sigma^2}{K^2} h \bigg( \frac{Kt}{\sigma^2} \bigg)\Bigg),
\\
\P \Bigg( \sum_{i=1}^{n} (\rX_i-\E \rX_i) \leq -t \Bigg)
&\leq
\exp \Bigg( - \frac{\sigma^2}{K^2} h \bigg( \frac{Kt}{\sigma^2} \bigg)\Bigg).
\end{align}
\end{subequations}
The lower-tail bound follows from the upper-tail bound applied to
\(\rY_i=-\rX_i\).
\end{lemma}

\begin{corollary}[Bennett--Bernstein form]
\label{cor:Bennett-Bernstein}
Let \(\rX_1,\dots,\rX_n\) be independent random variables such that
\(|\rX_i-\E \rX_i|\le K\) almost surely for every \(i\). Let
\(\sigma^2=\sum_{i=1}^n \Var(\rX_i)\).
Then for every \(t>0\),
\(\P(\sum_{i=1}^n(\rX_i-\E\rX_i)\ge t)
\le\exp(-t^2/[2(\sigma^2+Kt/3)])\).
The same bound holds for the lower tail.
\end{corollary}

\begin{proof}
By Lemma~\ref{lem:Bennett},
\(\P(\sum_{i=1}^n(\rX_i-\E\rX_i)\ge t)
\le \exp(-(\sigma^2/K^2)h(Kt/\sigma^2))\),
where \(h(u)=(1+u)\log(1+u)-u\).
Using the standard bound \(h(u)\ge u^2/[2(1+u/3)]\) for \(u\ge0\), we obtain
\((\sigma^2/K^2)h(Kt/\sigma^2)\ge t^2/[2(\sigma^2+Kt/3)]\).
This proves the upper-tail bound. The lower-tail bound follows by applying the same
argument to \(-\rX_i\).
\end{proof}

\begin{lemma}[Berry--Esseen \cite{esseen1956moment}]\label{lem:berry-esseen}
    Let $\rX_{1}, \rX_{2}, \ldots, \rX_{n}$ be independent random variables with zero means, variances $\sigma_{1}^{2}, \sigma_{2}^{2}, \ldots, \sigma_{n}^{2}$, respectively, and finite absolute third moments $\E[|\rX_{j}|^{3}] = \rho_{j} < \infty$. Define the normalized random variable \(\rS_{n}\coloneqq\sum_{j=1}^{n}\rX_{j}/\sqrt{\sum_{j=1}^{n}\sigma_{j}^{2}}\). Let $F_{n}(x)$ and $\Phi(x)$ denote the cumulative distribution functions \emph{(CDFs)} of $\rS_{n}$ and the standard normal distribution, respectively. There exists a universal constant $\gC_{\mathrm{BE}}$ such that, for any \(n\),
    \[\sup_{x\in \R} |F_{n}(x) - \Phi(x)| \leq \gC_{\mathrm{BE}} \sum_{j=1}^{n} \rho_{j} ( \sum_{j=1}^{n} \sigma_{j}^{2})^{-3/2},
    \]
    where the original result of Esseen \cite{esseen1956moment} gives
    \(\gC_{\mathrm{BE}}=7.59\). Subsequent work improved this constant; we use
    the bound \(\gC_{\mathrm{BE}}=0.56\) from
    \cite[Theorem~1]{tyurin2009new} throughout this paper.
\end{lemma}

\begin{lemma}[G\"artner--Ellis left-tail bound, {\cite[Lemma~H.5]{Abbe2022LPT}}]
\label{lem:gartner-ellis-left-tail}
Let \(\{S_n\}_{n\ge1}\) be random variables such that
\(\Lambda_n(t)\coloneqq\log \E e^{tS_n}\) exists for every
\(t\in[-R_n,R_n]\), where \(\{R_n\}_{n\ge1}\) is a positive sequence
with \(R_n\to\infty\). Suppose that there is a convex function
\(\Lambda:\R\to\R\) and a positive sequence \(\{a_n\}_{n\ge1}\) with
\(a_n\to\infty\) such that, for every \(t\in\R\),
\[
\lim_{n\to\infty}\frac{\Lambda_n(t)}{a_n}=\Lambda(t).
\]
Then, for every \(c<\Lambda'(0)\),
\[
\lim_{n\to\infty}\frac{1}{a_n}\log \P(S_n\le c a_n)
=
-\sup_{t\in\R}\{ct-\Lambda(t)\}.
\]
\end{lemma}

\begin{theorem}[Generalized H\"older inequality, {\cite[Theorem~2.1]{finner1992generalization}}]
\label{thm:finner-generalized-holder}
Let \(n,m\in\mathbb N\), \(I_n\coloneqq\{1,\ldots,n\}\), and
\(M\coloneqq\{1,\ldots,m\}\). For each \(i\in I_n\), let
\((\Omega_i,\mathcal A_i,\mu_i)\) be a measure space. For
\(\varnothing\ne S\subseteq I_n\), write
\[
\Omega_S\coloneqq\prod_{i\in S}\Omega_i,\qquad
\mathcal A_S\coloneqq\bigotimes_{i\in S}\mathcal A_i,\qquad
\mu_S\coloneqq\bigotimes_{i\in S}\mu_i.
\]
Let \(\varnothing\ne S_j\subseteq I_n\) and \(p_j\ge1\) for \(j\in M\). For
each \(i\in I_n\), put \(M_i\coloneqq\{j\in M:i\in S_j\}\), and assume that
\(\sum_{j\in M_i}p_j^{-1}=1\). If
\(f_j\in L^{p_j}(\Omega_{S_j},\mathcal A_{S_j},\mu_{S_j})\) for every
\(j\in M\), then
\(\prod_{j\in M}\abs{f_j}\in
L^1(\Omega_{I_n},\mathcal A_{I_n},\mu_{I_n})\), and
\begin{align}
\int_{\Omega_{I_n}}\prod_{j\in M}\abs{f_j}\,\de\mu_{I_n}
\le
\prod_{j\in M}
\left(\int_{\Omega_{S_j}}\abs{f_j}^{p_j}\,\de\mu_{S_j}\right)^{1/p_j}.
\label{eq:finner-generalized-holder}
\end{align}

To characterize equality in \eqref{eq:finner-generalized-holder}, assume without loss of generality and for the sake of simplicity that \(M_r\ne M_s\) for all \(r,s\in I_n\) with
\(r\ne s\), and that
\(\int_{\Omega_{S_j}}\abs{f_j}^{p_j}\,\de\mu_{S_j}>0\) for every \(j\in M\).
Then equality holds in \eqref{eq:finner-generalized-holder} if and only if, for every \(j\in M\) and \(i\in S_j\), there exist functions
\(f_{ji}\in L^{p_j}(\Omega_i,\mathcal A_i,\mu_i)\) and constants
\(A_{ji}>0\) such that
\begin{enumerate}[label=(\alph*)]
\item \(\abs{f_j}=\prod_{i\in S_j}\abs{f_{ji}}\), \(\mu_{S_j}\)-a.e.,
for every \(j\in M\);
\item \(A_{ri}\abs{f_{ri}}^{p_r}=A_{si}\abs{f_{si}}^{p_s}\),
\(\mu_i\)-a.e., for every \(i\in I_n\) and every \(r,s\in M_i\) with
\(r\ne s\).
\end{enumerate}
\end{theorem}

\printbibliography
\end{document}